\documentclass[11pt,reqno,tbtags,a4paper]{amsart}
\usepackage{amssymb}
\usepackage{xpunctuate}
\usepackage{url}
\usepackage[square,numbers]{natbib}

\bibpunct[, ]{[}{]}{;}{n}{,}{,}

\title[distance profile of simply generated trees]
{The distance profile of rooted and unrooted simply generated trees}

\date{31 August, 2020; revised 18 June, 2021}
\author{Gabriel Berzunza Ojeda}
\address{Department of Mathematical Sciences, University of Liverpool, Mathematical Sciences Building, L69 7ZL Liverpool, United Kingdom}
\email{gabriel.berzunza-ojeda@liverpool.ac.uk}

\author{Svante Janson}
\address{Department of Mathematics, Uppsala University, PO Box 480,
SE-751~06 Uppsala, Sweden}
\email{svante.janson@math.uu.se}
\newcommand\urladdrx[1]{{\urladdr{\def~{{\tiny$\sim$}}#1}}}
\urladdrx{http://www2.math.uu.se/~svante/}

\subjclass[2010] %really 2020
{60C05; 05C05; 60J65}
\numberwithin{equation}{section}

\renewcommand\le{\leqslant}
\renewcommand\ge{\geqslant}

\allowdisplaybreaks

\theoremstyle{plain}% default
\newtheorem{theorem}{Theorem}[section]
\newtheorem{lemma}[theorem]{Lemma}
\newtheorem{proposition}[theorem]{Proposition}
\newtheorem{corollary}[theorem]{Corollary}
\newtheorem{claim}{Claim}

\theoremstyle{definition}

\newtheorem{exampleqqq}[theorem]{Example}
\newenvironment{example}{\begin{exampleqqq}}
  {\hfill\qedsymbol\end{exampleqqq}}
\newtheorem{remarkqqq}[theorem]{Remark}
\newenvironment{remark}{\begin{remarkqqq}}
  {\hfill\qedsymbol\end{remarkqqq}}
\newtheorem{problem}[theorem]{Problem}

\theoremstyle{remark}

\newenvironment{romenumerate}[1][-10pt]{% optional argument changes indentation
\addtolength{\leftmargini}{#1}\begin{enumerate}% gives (i), (ii) etc.
 \renewcommand{\labelenumi}{\textup{(\roman{enumi})}}%
 }{\end{enumerate}}

\newcounter{oldenumi}
{\setcounter{oldenumi}{\value{enumi}}
\begin{romenumerate} \setcounter{enumi}{\value{oldenumi}}}
{\end{romenumerate}}

\newcounter{thmenumerate}

\newcounter{xenumerate}   %no left indentation; thus wider lines

\newcommand\pfitemx[1]{\par\emph{#1}:}

\newcommand{\refT}[1]{Theorem~\ref{#1}}
\newcommand{\refTs}[1]{Theorems~\ref{#1}}
\newcommand{\refC}[1]{Corollary~\ref{#1}}

\newcommand{\refL}[1]{Lemma~\ref{#1}}
\newcommand{\refLs}[1]{Lemmas~\ref{#1}}
\newcommand{\refR}[1]{Remark~\ref{#1}}

\newcommand{\refS}[1]{Section~\ref{#1}}
\newcommand{\refSs}[1]{Sections~\ref{#1}}
\newcommand{\refSS}[1]{Section~\ref{#1}}
\newcommand{\refP}[1]{Problem~\ref{#1}}

\newcommand{\refApp}[1]{Appendix~\ref{#1}}
\begingroup
  \divide\count255 by 60
  \multiply\count255 by -60
  \advance\count255 by \time
  \ifnum \count255 < 10 \xdef\klockan{\the\count1.0\the\count255}
  \else\xdef\klockan{\the\count1.\the\count255}\fi
\endgroup

\newcommand\nopf{\qed}   % for theorem without proof
\DeclareMathOperator*{\sumx}{\sum\nolimits^{*}}

\newcommand{\sumio}{\sum_{i=0}^\infty}

\newcommand{\sumko}{\sum_{k=0}^\infty}
\newcommand{\summo}{\sum_{m=0}^\infty}

\newcommand{\sumi}{\sum_{i=1}^\infty}
\newcommand{\sumj}{\sum_{j=1}^\infty}
\newcommand{\sumk}{\sum_{k=1}^\infty}
\newcommand{\summ}{\sum_{m=1}^\infty}
\newcommand{\sumn}{\sum_{n=1}^\infty}
\newcommand{\sumin}{\sum_{i=1}^n}

\newcommand\set[1]{\ensuremath{\{#1\}}}

\newcommand\xpar[1]{(#1)}
\newcommand\bigpar[1]{\bigl(#1\bigr)}
\newcommand\Bigpar[1]{\Bigl(#1\Bigr)}

\newcommand\lrpar[1]{\left(#1\right)}
\newcommand\sqpar[1]{[#1]}
\newcommand\bigsqpar[1]{\bigl[#1\bigr]}
\newcommand\Bigsqpar[1]{\Bigl[#1\Bigr]}

\newcommand\xcpar[1]{\{#1\}}

\newcommand\Bigcpar[1]{\Bigl\{#1\Bigr\}}

\newcommand\bigabs[1]{\bigl\lvert#1\bigr\rvert}
\newcommand\Bigabs[1]{\Bigl\lvert#1\Bigr\rvert}

\newcommand\lrabs[1]{\left\lvert#1\right\rvert}
\def\rompar(#1){\textup(#1\textup)}    % usage: \rompar(...)
\newcommand\xfrac[2]{#1/#2}

\def\xexp(#1){e^{#1}}
\newcommand\ceil[1]{\lceil#1\rceil}

\newcommand\floor[1]{\lfloor#1\rfloor}

\newcommand\setn{\set{1,\dots,n}}

\newcommand\ntoo{\ensuremath{{n\to\infty}}}

\newcommand\ktoo{\ensuremath{{k\to\infty}}}
\newcommand\mtoo{\ensuremath{{m\to\infty}}}

\newcommand\ttoo{\ensuremath{{t\to\infty}}}
\newcommand\xtoo{\ensuremath{{x\to\infty}}}
\newcommand\bmin{\land}
\newcommand\bmax{\lor}
\newcommand\norm[1]{\lVert#1\rVert}

\newcommand\normmx[1]{\norm{#1}'_2}
\newcommand\bignorm[1]{\bigl\lVert#1\bigr\rVert}

\newcommand\downto{\searrow}

\newcommand\punkt{\xperiod}    % xpunctuate
\newcommand\iid{i.i.d\punkt}    
\newcommand\ie{i.e\punkt}
\newcommand\eg{e.g\punkt}

\newcommand\cf{cf\punkt}
\newcommand{\as}{a.s\punkt}
\newcommand{\aex}{a.e\punkt}
\newcommand\ii{\mathrm{i}}

\newcommand{\tend}{\longrightarrow}
\newcommand\dto{\overset{\mathrm{d}}{\tend}}
\newcommand\pto{\overset{\mathrm{p}}{\tend}}

\newcommand\eqd{\overset{\mathrm{d}}{=}}

\newcommand\Op{O_{\mathrm p}}

\newcommand\bbR{\mathbb R}

\newcommand\bbN{\mathbb N}
\newcommand\bbNo{\bbN_0}

\newcommand\bbZ{\mathbb Z}

\newcounter{CC}
\newcommand{\CC}{\stepcounter{CC}\CCx} %new constant C_i
\newcommand{\CCx}{C_{\arabic{CC}}}     %repeats the last C_i
\newcounter{cc}
\newcommand{\cc}{\stepcounter{cc}\ccx} %new constant c_i
\newcommand{\ccx}{c_{\arabic{cc}}}     %repeats the last c_i
\newcommand{\ccdef}[1]{\xdef#1{\ccx}}     %defines #1 as the last c_i
\newcommand{\ccname}[1]{\cc\ccdef{#1}}    %new c_i and defines #1 as it
\newcommand\E{\operatorname{\mathbb E{}}}
\newcommand\Ez{\E\,}
\renewcommand\P{\operatorname{\mathbb P{}}}

\newcommand\Po{\operatorname{Po}}
\newcommand\Bi{\operatorname{Bi}}

\newcommand\Be{\operatorname{Be}}

\newcommand\diam{\operatorname{diam}}

\newcommand\ga{\alpha}

\newcommand\gd{\delta}

\newcommand\gl{\lambda}
\newcommand\gL{\Lambda}
\newcommand\go{\omega}

\newcommand\gs{\sigma}

\newcommand\gss{\sigma^2}

\newcommand\eps{\varepsilon}

\newcommand\cE{\mathcal E}

\newcommand\cH{\mathcal H}

\newcommand\cL{{\mathcal L}}
\newcommand\cM{\mathcal M}

\newcommand\cT{{\mathcal T}}

\newcommand\tC{\widetilde C}
\newcommand\tD{\widetilde D}

\newcommand\indic[1]{\boldsymbol1\xcpar{#1}}

\newcommand\etta{\boldsymbol1}

\newcommand\qw{^{-1}}
\newcommand\qww{^{-2}}
\newcommand\qq{^{1/2}}
\newcommand\qqw{^{-1/2}}
\newcommand\qqcw{^{-3/2}}

\newcommand\intoi{\int_0^1}
\newcommand\intot{\int_0^t}
\newcommand\intoo{\int_0^\infty}
\newcommand\intoom{\int_{-1}^\infty}
\newcommand\intoooo{\int_{-\infty}^\infty}
\newcommand\intpi{\int_0^{\pi}}
\newcommand\intpipi{\int_{-\pi}^{\pi}}
\newcommand\oi{\ensuremath{[0,1]}}

\newcommand\ooo{[0,\infty)}
\newcommand\oooq{[0,\infty]}
\newcommand\oooo{(-\infty,\infty)}

\newcommand\dtv{d_{\mathrm{TV}}}

\newcommand\dd{\,\mathrm{d}}

\newcommand\lhs{left-hand side}
\newcommand\rhs{right-hand side}

\newcommand\GW{Galton--Watson}
\newcommand\GWt{\GW{} tree}
\newcommand\cGWt{conditioned \GW{} tree}
\newcommand\cmGWt{conditioned modified \GW{} tree}
\newcommand\mGWt{modified \GW{} tree}
\newcommand\GWp{\GW{} process}
\newcommand\GWf{\GW{} forest}
\newcommand\cGWf{conditioned \GW{} forest}

\newcommand\xoo{_1^\infty}
\newcommand\zoo{_0^\infty}
\newcommand\fT{\mathfrak{T}}
\newcommand\fL{\mathfrak{L}}
\newcommand\fU{\mathfrak{U}}

\newcommand\rx{o_+}
\newcommand\ry{o_-}
\newcommand\Tx{T_+}
\newcommand\Ty{T_-}
\newcommand\cttx[1]{\cT_{(#1)}}
\newcommand\cttxn[1]{\cT_{n,(#1)}}
\newcommand\ctt{\cttx1}

\newcommand\uctn{\cT^\circ_{n}}
\newcommand\uctnw{\cT^{\bw,\circ}_{n}}
\newcommand\uctnwl{\hcT^{\bw,\circ}_{n}}
\newcommand\uctnwx{\hcT^{\bw,\bullet}_{n}}
\newcommand\uctnwy{\hcT^{\bw,*}_{n}}
\newcommand\ctnx{\cT_{n,1}}
\newcommand\ctny{\cT_{n,2}}
\newcommand\ctnp{\cT_{n,+}}
\newcommand\ctnm{\cT_{n,-}}
\newcommand\dout{d^+}
\newcommand\hw{\widehat{w}}

\newcommand\sg{simply generated}
\newcommand\sgq{\sg{} }
\newcommand\sgt{\sgq tree}
\newcommand\sgf{\sgq forest}
\newcommand\sgnct{\sgq non-crossing tree}

\newcommand\ctaw{\cT^{\bw,\bullet}}
\newcommand\ctb{\cT\x}
\newcommand\ctbw{\cT^{\bw,*}}
\newcommand\ctbn{\ctb_n}

\newcommand\ctcw{\cT^{\bw,\bullet\bullet}}
\newcommand\hcT{\widehat{\cT}}
\newcommand\phio{\phi^0}
\newcommand\x{^{*}}

\newcommand\za{a}
\newcommand\zb{u}
\newcommand\ZA{A}
\newcommand\ZB{U}
\newcommand\coi{C\oi}
\newcommand\coo{C[0,\infty)}
\newcommand\cooq{C[0,\infty]}
\newcommand\moo{\cM([0,\infty))}

\newcommand\bphi{\boldsymbol{\phi}}
\newcommand\bphio{\boldsymbol{\phi}^0}
\newcommand\tbphi{\bphi'}
\newcommand\tbphio{\bphi^{0\prime}}
\newcommand\tphi{\phi'}
\newcommand\tphio{\phi^{0\prime}}
\newcommand\be{\mathbf{e}}
\newcommand\bes{\be^{[s]}}
\newcommand\bb{\mathbf{b}}
\newcommand\bp{\mathbf{p}}
\newcommand\bpo{\mathbf{p}^0}
\newcommand\bv{\mathbf{v}}
\newcommand\bw{\mathbf{w}}
\newcommand\tbp{\tilde{\bp}}
\newcommand\po{{p}^0}
\newcommand\SG{\cT}
\newcommand\SGx[1]{\SG^{#1}}

\newcommand\SGphi{\SGx{\bphi}}
\newcommand\SGphiphio{\SGx{\bphi,\bphio}}
\newcommand\SGphin{\SGphi_n}
\newcommand\SGxx[2]{\SG^{#1,#2}}
\newcommand\SGphinm{\SGphi_{n,m}}
\newcommand\GGW{\SG}
\newcommand\GGWx[1]{\GGW^{#1}}
\newcommand\GGWxn[1]{\GGW_n^{#1}}
\newcommand\GGWp{\GGWx{\bp}}
\newcommand\GGWpn{\GGWp_n}
\newcommand\GGWppo{\GGWx{\bp,\bpo}}
\newcommand\GGWppon{\GGWppo_n}
\newcommand\GGWpnm{\GGWp_{n,m}}
\newcommand\GGWpnmx[1]{\GGWp_{n,m;#1}}
\newcommand\GGWpnmy[1]{\GGWp_{n,m;(#1)}}
\newcommand\ddd{\mathsf{d}}
\newcommand\GLX{\gL_\be}
\newcommand\hGLX{\widehat{\GLX}}
\newcommand\GLXd{\GLX'}
\newcommand\LX{L_\be}
\newcommand\LXs{L_{\bes}}
\newcommand\hLX{\widehat{\LX}}
\newcommand\Wiener{\operatorname{Wie}}
\newcommand\llx{[\ell]}
\newcommand\xll{^{\llx}}
\newcommand\olw{\overline{w}}

\newcommand\Lb{{L}_{\bb}}

\newcommand\hLb{\widehat{\Lb}}

\newcommand\Hax[1]{\cH_{#1}}
\newcommand\Ha{\cH_\ga}
\newcommand\HaA{\cH_\ga[0,A]}
\newcommand\Haoo{\cH_\ga[0,\infty)}

\newcommand\normHa[1]{\norm{#1}_{\Ha}}
\newcommand\normHaA[1]{\norm{#1}_{\HaA}}
\newcommand\normHaoo[1]{\norm{#1}_{\Haoo}}
\newcommand\normHax[2]{\norm{#1}_{\Ha#2}}
\newcommand\cdotx{x}
\newcommand\HCx[1]{\Holder($#1$)}
\newcommand\HCga{\Holder($\ga$)}
\newcommand\HCh{\Holder($\frac12$)}
\newcommand\bL{\bar{L}}
\newcommand\TAU[1]{\bar\fT^{#1}}
\newcommand\aT{\widetilde{\cT}}
\newcommand\bU{\overline U}
\newcommand\Tbe{T_{\be}}
\newcommand\hh{\hat h}
\newcommand\SOB[1]{\cL^2_{#1}}
\newcommand\BES[1]{B^{2,2}_{#1}}
\newcommand\BESoo[1]{B^{\infty,\infty}_{#1}}
\newcommand\sumxij{{\sum_{i,j}}^\ast}
\newcommand\sumxijl{{\sum_{i,j,l}}^\ast}
\newcommand\sumxijlr{{\sum_{i,j,l,r}}^\ast}
\newcommand\hLn{\widehat{L_n}}
\newcommand\hgLn{\widehat{\gL_n}}
\newcommand\hzLn{\widehat{L_n^\bbZ}}
\newcommand\hzgLn{\widehat{\gL^\bbZ_n}}
\newcommand\hzf{\widehat{f^\bbZ}}
\newcommand\htau{\widehat{\tau}}
\newcommand\bx{\bar x}
\newcommand\cn{\tilde c_n} %?
\newcommand{\Holder}{H\"older}

\newcommand\CS{Cauchy--Schwarz}
\newcommand\CSineq{\CS{} inequality}

\begin{document}

% Some suggestions:
% 05 Combinatorics 
% 05C Graph theory [For applications of graphs, see 68R10, 90C35, 94C15]
% 05C05 Trees
% 05C07 Vertex degrees
% 05C35 Extremal problems [See also 90C35]
% 05C40 Connectivity
% 05C65 Hypergraphs
% 05C80 Random graphs
% 05C90 Applications
% 05C99 None of the above, but in this section 
% 
% 60 Probability theory and stochastic processes
% 60C Combinatorial probability
% 60C05 Combinatorial probability
% 
% 60F Limit theorems [See also 28Dxx, 60B12]
% 60F05 Central limit and other weak theorems
% 60F17 Functional limit theorems; invariance principles
% 
% 60J Markov processes
% 60J65 Brownian motion [See also 58J65]
% 60J80 Branching processes (Galton-Watson, birth-and-death, etc.)

\begin{abstract} 
It is well-known that the height profile of a critical 
conditioned Galton--Watson tree with finite
offspring variance converges, after a suitable normalization, to the local time
of a standard Brownian excursion. In this work, we study the
distance profile, defined as the profile of all distances between pairs of
vertices. We show that after a proper rescaling the distance profile converges
to a continuous random function that can be described as the density of
distances between random points in the Brownian continuum random tree.

We show that this limiting function a.s.\ is H\"older continuous of any order
$\alpha<1$, and that it is a.e.\ differentiable.
We note that it cannot be differentiable at $0$, but leave as open questions
whether it is Lipschitz, and whether it is
continuously differentiable on the half-line $(0,\infty)$.

The distance profile is naturally defined also for unrooted trees contrary to
the height profile that is designed for rooted trees. 
This is used in our proof, and 
we prove the corresponding convergence result for the distance profile of
random unrooted simply generated trees. As a 
minor purpose of the present work, we also formalize the notion of unrooted
simply generated trees and include some simple results relating them to rooted
simply generated trees, which might be of independent interest.
\end{abstract}

\maketitle

  \section{Introduction}\label{S:intro}

Consider a random \sgt.
(For definitions of this and other concepts in the introduction, see 
\refSs{Snot}--\ref{Sroot}.)
Under some technical conditions, amounting to the tree being equivalent to a
critical \cGWt{} with finite offspring variance,
the (height)
profile of the tree converges in distribution, 
as a random function in $\coo$. Moreover, the limiting random function can
be identified with the local time of a standard Brownian excursion;
this was 
conjectured by \citet{AldousII} and %Conjecture 4 
proved by \citet{DrmotaG} (under a stronger assumption),
see also \citet[Section 4.2]{Drmota}, and in general by \citet{Kersting}
in a paper that unfortunately remains unpublished.
See further \citet{Pitman} for related results and a proof in a special case.
See also \citet{Kersting} for extensions when the offspring variance is
infinite, 
a case not considered in the present paper.

\begin{remark}\label{RDrmota}
To be precise,
\cite{DrmotaG} and \cite{Drmota} assume that the offspring distribution
for the \cGWt{} has a finite exponential moment. 
As said in \cite[footnote on page 127]{Drmota}, the analysis can be
extended, but it seems that the proof of tightness in \cite{Drmota}, 
which is based on estimating fourth moments, requires 
a finite fourth moment of the offspring distribution.

Note also that \citet[``a shortcut'' pp.~123--125]{Drmota} besides the proof
from
\cite{DrmotaG} also 
gives an alternative proof 
that combines tightness 
(taken from the first proof) and 
the convergence of the contour process
to a Brownian excursion shown by \citet{AldousIII},
and thus avoids some intricate calculations in the first proof.
We will use this method of proof below.
\end{remark}

Using notation introduced below, the result can be stated as follows.
\begin{theorem}[\citet{DrmotaG}, \citet{Kersting}]
\label{T0}
Let\/ $L_n$ be the (height) profile of a \cGWt{} %$\GGWpn$ 
of order $n$, with an
offspring distribution that has mean $1$ and finite variance $\gss$.
Then, as \ntoo,
\begin{align}\label{t0}
  n\qqw L_n(x n\qq) \dto \frac{\gs}2\LX\Bigpar{\frac{\gs}2 x},
\end{align}
in the space $\cooq$, where $\LX$ is a random function that can be identified
with the local time of a standard Brownian excursion $\be$;
this means that
for every bounded measurable
$f:\ooo\to\bbR$, 
\begin{align}\label{t0d}
\intoo f(x) \LX(x) \dd x
=\intoi f\bigpar{\be(t)}\dd t
.\end{align}
\end{theorem}

\begin{remark}\label{Rcooq}
This result is often stated with convergence \eqref{t0} in the space
$\coo$;
the version stated here with $\cooq$  is somewhat stronger but follows easily.
(Note that the maximum is a continuous functional on $\cooq$ but not on $\coo$.)
See further \refS{SSspaces}.
\end{remark}

The profile discussed above is the profile of the distances from the
vertices to the root. Consider now instead the \emph{distance profile}, 
defined as the profile of all distances between pairs of points.
(Again, see \refS{Snot} for details.)
One of our main results is the following analogue of 
\refT{T0}.

\begin{theorem}
  \label{T1}
Let\/ $\gL_n$ be the distance profile of a \cGWt{} of order $n$, 
with an offspring distribution
that has mean $1$ and finite variance $\gss>0$.
Then, as \ntoo, 
\begin{align}
  \label{t1}
n^{-3/2}\gL_n\bigpar{x n\qq}
\dto
\frac{\gs}2\GLX\Bigpar{\frac{\gs}2 x},
\end{align}
in the space $\cooq$,
where $\GLX(x)$ is a 
continuous
random function that can be described as the density of
distances between random points in the Brownian continuum random tree  
\cite{AldousI,AldousII,AldousIII}; equivalently, 
for a standard Brownian excursion $\be$, 
%we can define $\GLX$ such that, a.s., 
we have for every bounded measurable
$f:\ooo\to\bbR$, 
\begin{align}\label{t1d}
\intoo f(x) \GLX(x) \dd x
=2\iint_{0<s<t<1} f\bigpar{\be(s)+\be(t)-2\min_{u\in[s,t]} \be(u) }\dd s\dd t
.\end{align}
\end{theorem}

The random distance profile $\gL_n$
was earlier studied in \cite{SJ222}, where the estimate
\eqref{egln} below was shown.

\begin{remark}
  It is easy to see that the random function $\GLX$ really is random and not
  deterministic, \eg{} as a consequence of \refT{Twie}.
However, we do not know its distribution, although
the expectation $\E\GLX(x)$ is given in \refL{LNC}.
In particular, the following problem is open.
(See \cite[Section 4.2.1]{Drmota} for such results, in several different
forms,  for $\LX$.)
\end{remark}

\begin{problem}
 Find a description of the (one-dimensional)  distribution of $\GLX(x)$ for
fixed $x>0$. 
\end{problem}

We have so far discussed rooted trees. However, the distance profile is
defined also for unrooted trees, and we will find it convenient to use
unrooted trees in parts of the proof. This leads us to consider 
random \emph{unrooted \sgt{s}}.

Simply generated families of rooted trees were introduced by \citet{MM},
leading to the notion of simply generated random rooted
trees, see \eg{} \citet{Drmota}.
This class of random rooted trees is one of the most popular classes of
random trees, and these trees have been frequently studied in many different
contexts by many authors.
Simply generated random \emph{unrooted} trees have been much less studied,
but
they have occured \eg{} in a work on non-crossing trees by
\citet{KortchemskiM} (see also \citet{MarckertP}).
Nevertheless, we have not found a general treatment of them, 
so a minor purpose of the present paper is to 
do this in some detail, both for use in the paper and for future reference.
We thus include (\refSs{Sunroot}--\ref{Smark3} and \refApp{Agen}) a
general discussion of random unrooted \sgq trees, with some simple results
relating them to rooted \sgt{s},
allowing the transfer of many results for rooted \sgq trees to the
unrooted case. 
Moreover, as part of the proof of \refT{T1}, we prove
the corresponding result (\refT{Teo}) for random unrooted \sgt{s}.

As a preparation for the unrooted case, we also give (\refS{Smod})
some results (partly from \citet{KortchemskiM})
on modified rooted \sgt{s} (\GWt{s}),
where the root has different weights (offspring distribution) than all other
vertices. 

The central parts of the proof of \refT{T1} are given in
Sections \ref{S1step}--\ref{SDR},
where we use both rooted and unrooted trees.
As a preparation, in Section \ref{Smod2},
we extend \refT{T0}  to \cmGWt{s}.
We later also extend \refT{T1}  to \cmGWt{s} (\refT{T1m}). 
 
We end the paper with some comments and further results related to our main
results. In Section \ref{WienerI}, we discuss 
a simple application to the Wiener
index of unrooted \sgt{s}. Section \ref{Momentdp} contains some important
moment estimations of the distance profile for \cGWt{s} as well as for its
continuous counterpart $\GLX$. In Section \ref{SHolder}, we establish
H\"older continuity properties of the continuous random function $\LX$ and
$\GLX$. 
It is known that $\LX$ is \as{} 
H\"older continuous of order $\alpha$ 
(abbreviated to \HCga{}) 
for $\ga<\frac12$,  
but not for $\ga\ge\frac12$. We show that $\GLX$ is smoother; it is \as{}
\HCga{} for $\ga<1$, and it is \aex{} differentiable (\refT{TF}).
We do not know whether it is Lipschitz, or even continuously differentiable
on $\ooo$, but we show that it is does not \as{} have a two-sided derivative
at $0$ (\refT{TNC}), and we state some open problems.

Finally, some further remarks are given in Section \ref{Sfurther}.

\section{Some notation}\label{Snot}

Trees are finite except when explicitly said to be infinite. 
Trees may be \emph{rooted} or \emph{unrooted}; in a rooted tree,
the root is denoted $o$. The rooted trees may be \emph{ordered} or not.
The unrooted trees will always be \emph{labelled}; we do not consider unrooted
unlabelled trees in the present paper.

If $T$ is a tree, then its number of vertices is denoted by $|T|$; this is
called the \emph{order} or the \emph{size} of $T$. (Unlike some authors, we
do not distinguish between order and size.)
The notation $v\in T$ means that $v$ is a vertex in $T$.

The degree of a vertex $v\in T$ is denoted $d(v)$. % or $d_v$.
In a rooted tree, we also define the outdegree $\dout(v)$ %=\dout_v$
as the number of children of $v$; thus 
\begin{align}\label{dout}
  \dout(v)=
  \begin{cases}
    d(v)-1, & v\neq o,
\\
d(v), & v=o.
  \end{cases}
\end{align}
A \emph{leaf} in an unrooted tree is a vertex $v$ with  $d(v)=1$. In a
rooted tree, we instead require $\dout(v)=0$; this may make a difference only
for the root.

A \emph{fringe subtree} 
in a rooted tree is a subtree consisting of some vertex $v$
and all its descendants. We regard $v$ as the root of the fringe tree.
The \emph{branches} of a rooted tree are the fringe trees rooted at the
children of 
the root. The number of branches  thus equals the degree $d(o)$ of the root.

Let $\fT_n$ be the set of all ordered rooted trees of order $n$, and
let $\fT:=\bigcup\xoo\fT_n$.
Note that $\fT_n$ is a finite set; 
we may 
identify the vertices of an ordered rooted tree
by finite strings of positive integers,
such that the root is the empty string 
and the children of $v$ are $vi$, $i=1,\dots,d(v)$. 
(Thus, an ordered rooted tree is regarded as a subtree of the infinite
Ulam--Harris tree.)
In fact, it is well-known that $|\fT_n|=\frac{1}{n}\binom{2n-2}{n-1}$, the
Catalan number $C_{n-1}$.

Let $\fL_n$ be the set of all unrooted trees of order $n$, with the labels
$1,\dots,n$; thus $\fL_n$ is the set of all trees on $[n]:=\setn$.
$\fL_n$ is evidently finite and by Cayley's formula $|\fL_n|=n^{n-2}$.
Let $\fL:=\bigcup\xoo\fL_n$.

A \emph{probability sequence} is the same as a probability distribution on
$\bbNo:=\set{0,1,2,\dots}$, \ie, a sequence $\bp=(p_k)\zoo$ with $p_k\ge0$
and $\sumko p_k=1$.
The \emph{mean} 
$\mu(\bp)$ and \emph{variance} $\gss(\bp)$ of a probability sequence
are defined to be the mean and variance of a random variable with
distribution $\bp$, \ie,
\begin{align}\label{mubp}
  \mu(\bp):=\sumko kp_k, &&&
\gss(\bp):=\sumko k^2p_k-\mu(\bp)^2.
\end{align}

We use $\dto$ and $\pto$ for convergence in distribution and in probability,
respectively, for a sequence of random variables in some metric space; see
\eg{} \cite{Billingsley}.
Also, $\eqd$ means convergence in distribution. 

The total variation distance between two random variables $X$ and $Y$
in a metric space 
(or rather between their distributions) 
is defined by
\begin{align}\label{dtv}
  \dtv(X,Y):=\sup_A\bigabs{\P(X\in A)-\P(Y\in A) },
\end{align}
taking the supremum over all measurable subsets $A$. It is well-known that
in a complete separable metric space,
there exists a coupling of $X$ and $Y$ (\ie, a joint distribution with the
given marginal distributions) such that
\begin{align}\label{dtv2}
  \P(X\neq Y) = \dtv(X,Y),
\end{align}
and this is best possible.

$\Op(1)$ denotes a sequence of real-valued 
random variables $(X_n)_n$ that is stochastically
bounded, \ie, for every $\eps>0$, there exists $C$ such that 
$\P(|X_n|>C)\le\eps$. This is equivalent to $(X_n)_n$ being tight. 
For tightness in more general metric spaces, see \eg{} \cite{Billingsley}.

Let $f$ be a real-valued function defined on an interval $I\subseteq\bbR$.
The \emph{modulus of continuity} of $f$ is the function $\ooo\to[0,\infty]$
defined by, for $\gd\ge0$,
\begin{align}\label{go}
  \go(\gd;f)=\go(\gd;f;I)
:=\sup\bigpar{|f(s)-f(t)|: s,t\in I, |s-t|\le\gd}.
\end{align}
%We write (for the cases we are interested in)
%\begin{align}\label{a2}
%  \go(\gd;f;a,b)
%:=
%  \begin{cases}
%    \go(\gd;f;[a,b]), & -\infty<a<b<\infty
%\\
%    \go(\gd;f;[a,b)), & -\infty<a<b=\infty
%  \end{cases}
%\end{align}

If  $x$ and $y$ are  real numbers, $x\bmin y:=\min\set{x,y}$
and $x\bmax y:=\max\set{x,y}$.

$C$ denotes unspecified constants that may vary from one occurrence to the
next. They may depend on parameters such as weight sequences or offspring
distributions, but they never depend on the size of the trees.
Sometimes we write \eg{} $C_r$ to emphasize that the constant depends on the
parameter $r$.

Unspecified limits are as \ntoo.

% $\floor x$, % anv}nt; beh{vs ej f{rklaras? 

\subsection{Profiles}\label{SSprof}

For two vertices $v$ and $w$ in a tree $T$, let $\ddd(x,y)=\ddd_T(x,y)$
denote the distance between $v$ and $w$, \ie, the number of edges in the
unique path joining $v$ and $w$.
In particular, in a rooted tree, $\ddd(v,o)$ is the distance to the root,
often called the \emph{depth} (or sometimes \emph{height}) of $v$.%
\footnote{We use different fonts to distinguish the distance $\ddd$ 
from the  degree $d$; 
note also that the distance has two arguments and the degree only one.}

For a rooted tree $T$, the \emph{height} of $T$ is
$H(T):=\max_{v\in T}\ddd(v,o)$, \ie, the maximum depth.
The \emph{diameter} of a tree $T$, rooted or not, is
$\diam(T):=\max_{v,w\in T}\ddd(v,w)$.

The \emph{profile} of a rooted tree is the function $L=L_T:\bbR\to\ooo$
defined by
\begin{align}\label{pro}
  L(i):=\bigabs{\set{v\in T:\ddd(v,o)=i}},
\end{align}
for integers $i$, extended by linear interpolation to all real $x$.
(We are mainly interested in $x\ge0$, and trivially $L(x)=0$ for $x\le -1$,
but it will be convenient to allow negative $x$.)
The linear interpolation can be written
\begin{align}\label{Ltau}
  L(x):=\sumio L(i)\tau(x-i),
\end{align}
where $\tau$ is the triangular function $\tau(x):=(1-|x|)\bmax0$.

Note that $L(0)=1$, and that $L$ is a continuous function with compact
support
$[-1,H(T)+1]$.
Furthermore,  since $\int\tau(x)\dd x=1$,
\begin{align}\label{pro1}
%  \intoo L(t)\dd t = \sumio L(i) -\frac12 = |T|-\frac12.
  \intoom L(x)\dd x = \sumio L(i)  = |T|,
\end{align}
where we integrate from $-1$ because of the linear interpolation;
we have
$ \intoo L(x)\dd x = \sumio L(i) -\frac12 = |T|-\frac12$.
%(The $-\frac12$ is an inconvenient 
%consequence of the linear interpolation; it is
%asymptotically negligible.)

The \emph{width} of $T$ is defined as
\begin{align}
  \label{width}
W(T):=\max_{i\in\bbN_0} L(i)=\max_{x\in\bbR} L(x).
\end{align}

Similarly, in any tree $T$, rooted or unrooted, we define the
\emph{distance profile} as the function $\gL=\gL_T:\ooo\to\ooo$
defined by
\begin{align}\label{dpro}
\gL(i):=\bigabs{\set{(v,w)\in T:\ddd(v,w)=i}}
\end{align}
for integers $i$, again extended by linear interpolaton to all real $x\ge0$.
For definiteness, we count ordered pairs in \eqref{dpro}, and we include the
case $v=w$, so $\gL(0)=|T|$. 
$\gL$ is a continuous function on $\ooo$ with support $[-1,\diam(T)+1]$.
We have, similarly to \eqref{pro1},
\begin{align}\label{dpro1}
%  \intoo \gL(x)\dd x = \sumio \gL(i) -\frac {|T|}2 
%%= |T|\bigpar{|T|-\tfrac12}
%= |T|^2-\tfrac12|T|
  \intoom \gL(t)\dd t = \sumio \gL(i)
= |T|^2
.\end{align}

If $T$ is an unrooted tree, let $T(v)$ denote the rooted tree obtained by
declaring $v$ as the root, for $v\in T$. Then, as a consequence of 
\eqref{pro} and \eqref{dpro},  
\begin{align}\label{dpropro}
  \gL_T(x)=\sum_{v\in T} L_{T(v)}(x).
\end{align}
Hence, the distance profile can be regarded as the sum (or, after
normalisation, average) of the profiles for all possible choices of a root.

\begin{remark}\label{Rstep}
Alternatively, one might extend $L$ to  a step function by $L(x):=L(\floor x)$,
and similarly for $\gL$.
The asymptotic results are the same (and equivalent by simple arguments),
with $L$ and $\gL$ elements 
of $D\oooq$ instead
of $\cooq$ and limit theorems taking place in that space.
This has some advantages,
%would avoid the inconvenient $-\frac12$ in \eqref{pro1} and \eqref{dpro1}, 
but for technical reasons (\eg{} 
simpler tightness criteria), we prefer to 
work in the space $\cooq$ of continuous functions.
\end{remark}

\begin{remark}
Another version of $\gL$  would count unordered pairs of
distinct vertices. The two versions are obviously equivalent and our results
hold for the alternative version too, 
\emph{mutatis mutandis}.
\end{remark}

\subsection{Brownian excursion and its local time}\label{SSbex}
The standard Brownian excursion $\be$ is a random continuous
function $\oi\to\ooo$ such that $\be(0)=\be(1)=0$ and $\be(t)>0$ for
$t\in(0,1)$.
Informally, $\be$ can be regarded as a Brownian motion conditioned on these
properties; 
this can be formalized as an appropriate limit
\cite{DurrettIM1977}.
There are several other, quite different but equivalent, definitions, 
see \eg{} \cite[XII.(2.13)]{RY}, \cite[Example II.1d)]{Blu},  
and \cite[Section 4.1.3]{Drmota}.

The \emph{local time} $\LX$ of $\be$ is a continuous random function that
is defined (almost surely) as a functional of $\be$ satisfying
\begin{align}\label{local}
  \intoo f(x)\LX(x)\dd x = \intoi f\bigpar{\be(t)}\dd t,
\end{align}
for every bounded (or non-negative) measurable function $f:\ooo\to\bbR$.
In particular, \eqref{local} yields, for any $x\ge0$ and $\eps>0$,
\begin{align}\label{local1}
  \int_x^{x+\eps}\LX(y)\dd y = \intoi \indic{\be(t)\in[x,x+\eps)}\dd t
\end{align}
and thus
\begin{align}\label{local2}
\LX(x) = \lim_{\eps\to0} \frac{1}{\eps}\intoi \indic{\be(t)\in[x,x+\eps)}\dd t.
\end{align}
Hence, 
$\LX(x)$ can be regarded as the occupation density of $\be$ at the
value $x$.

Note that the existence (almost surely) 
of a function $\LX(x)$ satisfying \eqref{local}--\eqref{local2} is far from
obvious; this is part of the general theory of local times for
semimartingales, see \eg{} \cite[Chapter VI]{RY}.
The existence also follows from (some of) the proofs of \refT{T0}.

\subsection{Brownian continuum random tree}\label{SSCRT}

Given a continuous function $g:\oi\to\ooo$ with $g(0)=g(1)=0$, one can
define a pseudometric $\ddd$ on $\oi$ by
\begin{align}\label{ddd}
\ddd(s,t)=  \ddd(s,t;g):=g(s)+g(t)-2\min_{u\in [s,t]}g(u),
\qquad 0\le s\le t\le 1.
\end{align}
By identifying points with distance 0, we obtain a metric space $T_g$, 
which is a
compact real tree, see \eg{} \citet[Theorem 2.2]{LeGall2005}. We denote the
natural quotient map $\oi\to T_g$ by $\rho_g$, and let
$T_g$ be rooted at $\rho_{g}(0)$. 
%(i.e., the class of points that are identified with $0$).
The \emph{Brownian continuum random tree} 
defined by \citet{AldousI,AldousII,AldousIII} 
can be defined as the
random real tree $T_{\be}$ constructed in this way from the random 
Brownian excursion $\be$, see \cite[Section 2.3]{LeGall2005}. 
(\citet{AldousI,AldousII,AldousIII} used another definition, and another scaling
corresponding to $T_{2\be}$.)
Note that using \eqref{ddd}, %the \rhs{} of 
\eqref{t1d} can be written
\begin{align}\label{t1dx}
\intoo f(x) \GLX(x) \dd x
=
\iint_{s,t\in\oi} f\bigpar{\ddd(s,t;\be)}\dd s\dd t,
\end{align}
for any bounded (or non-negative) measurable function $f$.
This means that
$\GLX$ is the density of the distance in $T_\be$ between
two random points, chosen independently with the probability measure on
$T_\be$ induced by the uniform measure on $\oi$.
This justifies the equivalence of the two definitions of $\GLX$
stated in \refT{T1}.
As for the local time $\LX$, the existence (almost surely) of a 
continuous function $\GLX$ satisfying \eqref{t1dx} is far from trivial;
this will be a consequence of our proof.

An important feature of the Brownian continuum random tree is its re-rooting
invariance property. More precisely, fix $s \in [0,1]$ and set
\begin{align} \label{PathT}
\be^{[s]}(t)=
  \begin{cases}
\ddd(s,s+t;\be), & 0\le t < 1-s
\\
\ddd(s,s+t-1;\be), &1-s \le t \le 1.
  \end{cases}
\end{align}
Note that $\be^{[s]}$ is a random continuous
function $\oi\to\ooo$ such that $\be^{[s]}(0)=\be^{[s]}(1)=0$ and \as{}
$\be^{[s]}(t)>0$ for 
$t\in(0,1)$; clearly, $\be^{[0]} = \be$. 
By  \citet[Lemma 2.2]{DuquesneLeGall-aspects}, 
the compact real tree $T_{\be^{[s]}}$ is then canonically identified with
the $T_{\be}$ tree re-rooted at the vertex $\rho_{\be}(s)$. 
%(i.e., the class of points that are identified with $s$). 
\citet[Proposition~4.9]{MM-map}
(see also  \citet[Theorem 2.2]{DuquesneLeGall-rerooting}) 
have shown that for every
fixed $s \in [0,1]$, 
\begin{align} \label{Reroot}
  \be^{[s]}\eqd \be \quad \text{and}  \quad
T_{\be^{[s]}} =T_{\be},
\end{align}
in distribution. Thus the re-rooted tree $T_{\be^{[s]}}$ is a version of the
Brownian continuum random tree.

\begin{remark} \label{RemarkReroot}
Indeed, Aldous \cite[(20)]{AldousII} already observed that the
Brownian continuum random tree is invariant under uniform re-rooting and
that this property corresponds to the invariance of the law of the Brownian
excursion under the path transformation \eqref{PathT}
if $s = U$ is uniformly random on $[0,1]$ and
independent of $\be$.  
\end{remark}

As a consequence of the previous re-rooting invariance property, we deduce
the following explicit expression for the continuous function $\GLX$. For
every fixed $s \in [0,1]$, let $\LXs$ denote the local time of
$\be^{[s]}$, which is perfectly defined thanks to \eqref{Reroot}. It follows
from  \eqref{ddd},  \eqref{t1dx} and \eqref{PathT} that
\begin{align*}
\intoo f(x) \GLX(x) \dd x
=
\intoi\intoi f\bigpar{\be^{[s]}(t)}\dd s \dd t
=
\int_{0}^{1} \intoo f(x) \LXs(x) \dd x\dd s,
\end{align*}
for any bounded (or non-negative) measurable function $f$,
or equivalently,
\begin{align} \label{dproI}
\GLX(x)  =
\int_{0}^{1} \LXs(x) \dd s, \hspace*{4mm} x \ge 0.
\end{align}
In accordance
with the discrete analogue of $\GLX$ in \eqref{dpropro}, the identity
\eqref{dproI} shows that $\GLX$ can be regarded as the average of the
profiles for all possible choices of a root in $T_{\be}$. 

\subsection{The function spaces $\coo$ and $\cooq$}\label{SSspaces}

Recall that $\coo$ is the space of continuous functions on $\ooo$, and that
convergence in $\coo$ means uniform convergence on each compact interval
$[0,b]$. As said in \refR{Rcooq}, we prefer to state our results in the
space $\cooq$ of functions that are continuous on the extended half-line
$\oooq$. These are the functions $f$ in $\coo$ such that the limit
$f(\infty):=\lim_{\xtoo}f(x)$ exists; in our case, this is a triviality
since all random functions on both sides of \eqref{t0} and \eqref{t1}, and in
similar later statements, have compact support, and thus trivially extend
continuously to $\oooq$ with $f(\infty)=0$.
The important difference between $\coo$ and $\cooq$ is instead the topology:
convergence in $\cooq$ means uniform convergence on $\oooq$ (or,
equivalently, on $\ooo$). 

In particular, the supremum
is a continuous functional on $\cooq$, but not on $\coo$ (where it also may
be infinite). 
Thus,
convergence of the width  (after rescaling),
follows immediately from \refT{T0}
(see also the proof of \refT{TMGW1});
if this was stated with
convergence in $\coo$, a small extra argument would be needed
(more or less equivalent to showing convergence in $\cooq$).

In the random setting, the difference between the two topologies can be 
expressed as in the following lemma.
See also  \cite[Proposition 2.4]{SJIII}, for the similar case of the spaces
$D\oooq$ and $D\ooo$.

\begin{lemma}\label{Lcooq}
Let $X_n(t)$ and $X(t)$ be random functions in $\cooq$. Then $X_n(t)\dto X(t)$
in $\cooq$ if and only if
\begin{romenumerate}
  
\item\label{Lcooqa} 
$X_n(t)\dto X(t)$ in $\coo$, and
\item\label{Lcooqb} 
$X_n(t)\pto X_n(\infty)$, as \ttoo, uniformly in $n$; i.e., for every $\eps>0$,
  \begin{align}\label{lcooqb}
    \sup_{n\ge1}\P\bigpar{\sup_{u<t<\infty}|X_n(t)-X_n(\infty)|>\eps}\to 0,
\qquad \text{as }u\to\infty.
  \end{align}
\end{romenumerate}
%\nopf
\end{lemma}

\begin{proof}
A straightforward exercise.
\end{proof}

In our cases, such as \eqref{t0} and \eqref{t1}, the condition \eqref{lcooqb}
is easily verified
from convergence (or just tightness) of the 
normalized height $H_n/\sqrt n$, 
which can be used to bound the support of the \lhs{s}.
Hence, convergence in $\coo$ and $\cooq$ are essentially equivalent.

Note that $\cooq$ is a separable Banach space, and that it is
isomorphic to $\coi$ by a change of variable; thus general results for
$\coi$ may be transferred.
Note also that all functions that we are interested in lie in the 
(Banach) subspace
$C_0\ooo:=\set{f\in\cooq:f(\infty)=0}$.
Hence, the results may just as well be stated as convergence in distribution
in  $C_0\ooo$.

\section{Rooted \sgt{s}}\label{Sroot}

As a background, we recall first the definition of random rooted \sgt{s}
and the almost equivalent \cGWt{s},
see \eg{} \cite{Drmota} or \cite{SJ264} for further details, and \cite{AN}
for more on \GWp{es}.

\subsection{Simply generated trees}
Let $\bphi=(\phi_k)\zoo$ be a given sequence of non-negative weights,
with $\phi_0>0$ and $\phi_k>0$ for at least one $k\ge2$.
(The latter conditions exclude only trivial cases when the random
tree $\SGphin$
defined below either does not exist or is a deterministic path.)

For any rooted ordered tree $T\in\fT$, define the weight of $T$ as
\begin{align}\label{wTr}
  \phi(T):=\prod_{v\in T} \phi_{\dout(v)}.
\end{align}

For a given $n$, we define the random rooted \emph{\sgt} $\SGphin$ of order $n$
as a random tree in $\fT_n$ with probability proportional to its weight;
\ie,
\begin{align}\label{Pr}
  \P(\SGphin=T) :=\frac{\phi(T)}{\sum_{T'\in\fT_n}\phi(T')},
\qquad T\in\fT_n
.\end{align}
We consider only $n$ such that at least one tree $T$ with $\phi(T)>0$ exists.

A weight sequence $\tbphi=(\tphi_k)\zoo$ with
\begin{align}\label{equiv}
  \tphi_k=a b^k\phi_k,
\qquad k\ge0,
\end{align}
for some $a,b>0$
is said to be
\emph{equivalent} to $(\phi_k)\zoo$. It is easily seen that equivalent
weight sequences define the same random tree $\SGphin$, i.e.,
$\SGx{\tbphi}_n\eqd\SGx{\bphi}_n$.

\subsection{\GWt{s}}\label{SSGW}

Given a probability sequence $\bp=(p_k)\zoo$, the \emph{\GWt} $\GGWp$ is the
family tree of a \GWp{} with offspring distribution $\bp$. This means that
$\GGWp$ is a random ordered rooted tree, which is constructed as follows: 
Start with a root and give it a random number of children with the
distribution $\bp$. Give each new vertex a random number of children with the
same distribution and independent of previous choices, 
and continue as long as there are new vertices.
In general, $\GGWp$ may be an infinite tree. We will mainly consider the
\emph{critical} case when the expectation $\mu(\bp)=1$,
and then it is well-known that $\GGWp$ is finite a.s.
(We exclude the trivial case when $p_1=1$.)

The size $|\GGWp|$ of a \GWt{} is random. Given $n\ge1$, the
\emph{\cGWt} $\GGWpn$ is defined as $\GGWp$ conditioned on $|\GGWp|=n$.
(We consider only $n$ such that this happens with positive probability.)
Consequently, $\GGWpn$ is a random ordered rooted tree of size $n$.
It is easily seen that a \cGWt{} $\GGWpn$ equals (in distribution) 
the \sgt{} with the  weight sequence $\bp$, 
and thus we use the same notation $\GGWpn$ for both.

A (conditioned) \GWt{} is \emph{critical} if its offspring distribution
$\bp$ has mean $\mu(\bp)=1$.
We will in the present paper mainly consider \cGWt{s} that are critical and
have a finite variance $\gss(\bp)$; this condition is rather mild, as is
seen in the following subsection.

\subsection{Equivalence}\label{SS=}
A random simply generated tree $\SGphin$ 
with a weight sequence $(\phi_k)\zoo$ that is a probability sequence
equals, as just said,
the \cGWt{} $\GGWx{\bphi}_n$. Much more generally,
any weight sequence $\bphi$ such that its generating function
\begin{align}\label{Phi}
  \Phi(z):=\sumko \phi_k z^k
\end{align}
has positive radius of convergence
is equivalent to some probability weight sequence;
hence, $\SGphin$ can be regarded as a \cGWt{} in this case too.
Moreover, in many cases we can choose an equivalent probability weight
sequence that has mean 1 and finite variance;
see  \eg{} \cite[\S4]{SJ264}.

We will use this to switch between \sgt{s} and \cGWt{s} 
without comment in the sequel; we will use the name that seems best
and most natural in different contexts.

\subsection{Simply generated forests}\label{GWf}

The \GWp{} above starts with one individual. More generally, we may start
with $m$ individuals, which we may assume are numbered $1,\dots,m$; 
this yields a \emph{\GWf} 
consisting of $m$
independent copies of $\GGWp$. Conditioning on the total size being $n\ge m$, 
we obtain a \emph{\cGWf} $\GGWpnm$,
which thus consists of $m$ random trees 
$\GGWpnmx{1}, \dots,\GGWpnmx{m}$ with $|\GGWpnmx{1}|+\dots+|\GGWpnmx{m}|=n$.
Conditioned on the sizes $|\GGWpnmx{1}|,\dots,|\GGWpnmx{m}|$, the trees are
independent 
\cGWt{s} with the given sizes. 

More generally, given any weight sequence $\bphi$,
a random \emph{\sgf} $\SGphinm$
is a random forest with $m$ rooted trees and total size $n$, chosen with
probability proportional to its weight, defined as in
\eqref{wTr}.
Again, conditioned on their sizes, the trees are independent \sgt{s}.

Thus, the distribution of the sizes of the
trees in the forest is of major importance.
Consider the \GW{} case, and
let $\GGWpnmy{1}, \dots,\GGWpnmy{m}$ 
denote the trees arranged in decreasing order:
$|\GGWpnmy{1}|\ge\dots\ge|\GGWpnmy{m}|$.
(Ties are resolved  randomly, say; this applies tacitly 
to all similar situations.)
We have the following general result, which was proved by 
\citet[Lemma 5.7(iii)]{KortchemskiM} under an additional regularity
hypothesis. 
%assumed to be in the domain of attraction of a stable law

\begin{lemma}\label{LGWf}
  Let $m\ge1$ be fixed, and consider 
the \cGWf{} $\GGWpnm$ as \ntoo.
Then 
\begin{align}
|\GGWpnmy{i}|=
  \begin{cases}
n-\Op(1), & i=1
\\
\Op(1), & i=2,\dots,m.
  \end{cases}
\end{align}
\end{lemma}

\begin{proof}
  Suppose first
that $\mu(\bp)=1$. 
Suppose also, for simplicity, that $p_m>0$.
Consider the \cGWt{} $\GGWp_{n+1}$, and condition on the event $\cE_m$ that
the root degree is $m$.
Conditioned on $\cE_m$, there are $m$ branches, which form a \cGWf{} $\GGWpnm$.

As \ntoo, the random tree $\GGWp_{n+1}$ converges in distribution to an
infinite random tree $\hcT$ 
(the size-biased Galton--Watson tree defined by \citet{Kesten}),
see \cite[Theorem~7.1]{SJ264}.
Moreover, $\P(\cE_m)\to mp_m>0$ by \cite[Theorem~7.10]{SJ264}.
Hence, $\GGWp_{n+1}$ conditioned on $\cE_m$ converges in distribution to 
$\hcT$ conditioned on $\cE_m$. In other words, the forest
$\GGWpnm$ converges in distribution 
to the branches of $\hcT$ conditioned on having exactly
$m$ branches; denote this random limit by $(\cT_1,\dots,\cT_m)$.
By the Skorohod coupling theorem \cite[Theorem~4.30]{Kallenberg},
we may (for simplicity) assume that this convergence is \as.
The convergence here is in the local topology used in \cite{SJ264},
which means \cite[Lemma 6.2]{SJ264} that for any fixed $\ell\ge1$,
if $T\xll$ denotes the tree $T$ truncated at height $\ell$, then
\as, for sufficiently large $n$,
$\GGWx{\bp,\llx}_{n,m;i}=\cT_i\xll$.

The infinite tree $\hcT$ has exactly one infinite branch;
thus there exists a (random) $j\le m$ such that $\cT_j$ is infinite but
$\cT_i$ is finite for $i\neq j$.
Truncating the trees at an $\ell$ chosen larger than the heights $H(\cT_i)$
for all $i\neq j$, we see that for large $n$, 
$\GGWx{\bp}_{n,m;i}=\cT_i$. 
Thus, $|\GGWx{\bp}_{n,m;i}|=O(1)$ for $i\neq j$,
and necessarily the remaining branch
$\GGWx{\bp}_{n,m;j}$ has size $n-O(1)$.
Hence, for large enough $n$, 
$\GGWx{\bp}_{n,m;(1)}=\GGWx{\bp}_{n,m;j}$. 

Consequently, 
$\GGWx{\bp}_{n,m;(2)},\dots,\GGWx{\bp}_{n,m;(m)}$ converge
\as, and thus in distribution, to the $m-1$ finite branches of $\hcT$,
arranged in decreasing order and
conditioned on $\cE_m$. In particular, their sizes converge in distribution, 
and are thus $\Op(1)$.

We assumed for simplicity $p_m>0$. In general, we may select a rooted tree
$T$ with $\ge m$ leaves, such that $p_{\dout(v)}>0$ for every $v\in T$.
Fix $m$ leaves $v_1,\dots,v_m$ in $T$, and consider the \cGWt{}
$\GGWp_{n+|T|-m}$ conditioned on the event $\cE_T$ that it consists of $T$
with subtrees added at $v_1,\dots,v_m$. Then these subtrees form a \cGWf{}
$\GGWpnm$, and we can argue as above, conditioning $\hcT$ on $\cE_T$.

This completes the proof when
 $\mu(\bp)=1$.
If $\mu(\bp)>1$, there always exists an equivalent probability weight
$\tbp$ with $\mu(\tbp)=1$, and the result follows.
If $\mu(\bp)<1$, the same may hold, and if it does not hold, then there is
a similar infinite limit tree $\hcT$
\cite[Theorem~7.1]{SJ264}; in this case, $\hcT$ has one vertex of
infinite degree, but the proof above holds with minor modifications.
\end{proof}

\begin{remark}\label{Rsmall}
  The proof shows that 
in the case $\mu(\bp)=1$,
the small trees 
$\GGWx{\bp}_{n,m;(2)},\dots,\GGWx{\bp}_{n,m;(m)}$ in the forest
converge in distribution 
to $m-1$ independent copies of the unconditioned \GWt{} $\GGWp$, arranged in
decreasing order.
More generally, the small trees converge in distribution to independent
\GWt{s} for a probability distribution equivalent to $\bp$.
(This too was shown in \cite[Lemma 5.7(iii)]{KortchemskiM} under
stronger assumptions.)
\end{remark}

\begin{remark}
  In the standard case $\mu(\bp)=1$, $\gss(\bp)<\infty$, it is also easy to
show \refL{LGWf} using the fact $\P(|\GGWp|=n)\sim c n^{-3/2}$, for some $c>0$,
which is a well-known consequence of the local limit theorem,
\cf{} \eqref{b1}--\eqref{b2}.
\end{remark}

\begin{problem}\label{Psmall}
A \sgf{}  $\SGphi_{n,m}$ is covered by \refL{LGWf} when
the generating function \eqref{Phi} has positive radius of convergence,
since then it is equivalent to a \cGWf.
We conjecture that 
\refL{LGWf} holds for \sgf{s} also in the case when
the generating function  has radius of convergence 0,
but we leave this as an open problem.
\end{problem}

\section{Modified \sgt{s}}\label{Smod}

One frequently meets random trees where the root has a special distribution,
see for example \cite{MarckertP,KortchemskiM}.
Thus, let $\bphi$ and $\bphio$ be two weight sequences, where $\bphi$ is as
above, and $\bphio=(\phio_k)\zoo$ satisfies $\phio_k\ge0$, with strict
inequality for at least one $k$.
We modify \eqref{wTr} and now define the weight of a tree $T\in\fT$ as
\begin{align}\label{wTm}
  \phi^*(T):=\phio_{\dout(o)}\prod_{v\neq o} \phi_{\dout(v)}.
\end{align}
The random \emph{modified \sgt} $\SGxx{\bphi}{\bphio}_n$ is defined as in 
\eqref{Pr}, using the modified weight \eqref{wTm}.

We say that a pair $(\tbphi,\tbphio)$ is equivalent to
$(\bphi,\bphio)$ if \eqref{equiv} holds and similarly
\begin{align}\label{equiv0}
  \tphio_k=a_0 b^k\phio_k,
\qquad k\ge0.
\end{align}
It is important that the same $b$ is used in \eqref{equiv} and
\eqref{equiv0}, while $a$ and $a_0$ may be different.
It is easy to see that equivalent pairs of weight sequences define the same
modified \sgt.

Similarly, given two probability sequences $\bp=(p_k)\zoo$ and
$\bpo=(\po_k)\zoo$, we define the \emph{\mGWt} $\GGWx{\bp,\bpo}$
and \emph{\cmGWt} $\GGWx{\bp,\bpo}_n$ as in \refSS{SSGW}, 
but now giving children to the
root  with distribution $\bpo$, and to everyone else 
with distribution $\bp$.

Again, as indicated by our notation, we have an equality: 
the \cmGWt{}
$\GGWx{\bp,\bpo}_n$ equals the modified \sgt{}
with weight sequences $\bp$ and $\bpo$.
Conversely, 
if two weight sequences $\bphi$ and $\bphio$ both have positive radius of
convergence, then it is possible (by taking $b$ small enough) to find
equivalent weight sequences $\tbphi$ and $\tbphio$ that are probability
sequences, and thus
$
\SGx{\bphi,\bphio}_n
= \SGx{\tbphi,\tbphio}_n
$
 can be interpreted as a \cmGWt.

%The following lemma  should be obvious from the definitions.

\begin{lemma}\label{LM1}
  Consider a modified \sgt{} $\SGphiphio_n$ and denote its branches
  by $T_1,\dots,T_{d(o)}$. 
  \begin{romenumerate}
    
  \item \label{LM1a}
Conditioned on the root degree $d(o)$, the branches form a 
simply generated forest $\SGphi_{n-1,d(o)}$.
  \item \label{LM1b}
Conditioned on the root degree $d(o)$ 
and the   sizes $|T_i|$ of the branches, 
the branches are independent \sgt{s}
$\SGphi_{|T_i|}$.
  \end{romenumerate}
\end{lemma}

\begin{proof}
  Exercise.
\end{proof}
Note that \refL{LM1} applies also to the \sgt{} $\SGphin$ (by taking
$\bphio=\bphi$). Thus, conditioned on the root degree 
%and the sizes of the branches, 
the branches have the same distribution for $\SGphin$ and
$\SGphiphio_n$. 
Hence, the distribution of the root degree is of central importance.
The following lemma is partly shown by 
\citet[Proposition 5.6]{KortchemskiM} in greater
generality (the stable case),
although we add the estimate \eqref{lm2b}.

\begin{lemma}[mainly \citet{KortchemskiM}]\label{LM2}
  Suppose that $\bp$ is a probability sequence with mean $\mu(\bp)=1$ and
  variance $\gss(\bp)\in(0,\infty)$ and that $\bpo$ is a probability
  sequence with finite mean $\mu(\bpo)$.
Then the root degree $d(o)$ in the \cmGWt{} $\GGWxn{\bp,\bpo}$ converges in
distribution to a random variable $\tD$ with distribution
\begin{align}\label{lm2}
  \P(\tD=k)
= \frac{k\po_k}{\sumj j\po_j}
= \frac{k\po_k}{\mu(\bpo)}
.\end{align}
In other words, 
for every fixed $k\ge0$,
\begin{align}
  \label{lm2a}
\P(d(o)=k)\to\P(\tD=k),\qquad\text{as \ntoo} 
.\end{align}
Moreover, if $n$ is large enough, we have uniformly
\begin{align}\label{lm2b}
  \P(d(o)=k)\le 2\P(\tD=k), 
\qquad k\ge1
.\end{align}

As a consequence, $\E\tD<\infty$ if and only if $\gss(\bpo)<\infty$.
\end{lemma}

\begin{proof}
This uses well-known standard arguments, 
but we give a full proof for
completeness; see also \cite{KortchemskiM}.
%Note first that by a well-known formula by \cite{Otter}, see \eg{}
%\cite[\S15}{SJ264} and further references there,
  Let $D$ be the root degree in the modified \GWt{} $\GGWx{\bp,\bpo}$.
If $D=k$, then the rest of the tree consists of $k$ independent copies
of $\GGWx{\bp}$. 
Thus, the conditional probability 
$\P\bigpar{|\GGWx{\bp,\bpo}|=n\mid D=k}$ equals the probability that a \GWp{}
started with $k$ individuals has in total $n-1$ individuals; hence, by a
formula by \citet{Dwass},
see \eg{} \cite[\S15]{SJ264} and the further references there,
\begin{align}\label{b1}
  \P\bigpar{|\GGWx{\bp,\bpo}|=n\mid D=k}
=
\frac{k}{n-1}\P\bigpar{S_{n-1}=n-k-1},
\end{align}
where $S_{n-1}$ denotes the sum of $n-1$ independent random variables with
distribution $\bp$.

 Suppose for simplicity that the distribution $\bp$ is aperiodic, \ie, not
supported on any subgroup $d\bbN$. (The general case follows similarly
using standard modifications.)
It then follows by the local limit theorem, 
see \eg{}
\cite[Theorem 1.4.2]{Kolchin} or 
\cite[Theorem VII.1]{Petrov},
that, as \ntoo, 
\begin{align}\label{b2}
  \P\bigpar{S_{n-1}=n-k-1}
= \frac{1}{\sqrt{2\pi\gss n}}\bigpar{e^{-k^2/(2n\gss)}+o(1)},
\end{align}
uniformly in $k$.
Consequently, combining \eqref{b1} and \eqref{b2}
 with $\P(D=k)=\po_k$, 
\begin{align}\label{b3}
  \P\bigpar{|\GGWx{\bp,\bpo}|=n \text{ and } D=k}
&=
\frac{k\po_k}{n-1}\P\bigpar{S_{n-1}=n-k-1}
\notag\\&
= c\frac{k\po_k}{n^{3/2}}\bigpar{e^{-k^2/(2n\gss)}+o(1)},
\end{align}
uniformly in $k$, where $c:=(2\pi\gss)\qqw$.

Summing over $k$ we find as \ntoo, using 
$\sum k\po_k<\infty$ and
monotone convergence,
\begin{align}\label{b4}
  \P\bigpar{|\GGWx{\bp,\bpo}|=n}
&=
%\notag\\&
\frac{c}{n^{3/2}}\Bigpar{\sumk {k\po_k}e^{-k^2/(2n\gss)}+o(1)}
\sim \frac{c}{n^{3/2}}\sumk {k\po_k}.
\end{align}
Thus, combining \eqref{b3} and \eqref{b4}, for any fixed $k\ge1$,
as \ntoo,
\begin{align}\label{b6}
  \P\bigpar{D=k\mid|\GGWx{\bp,\bpo}|=n}
&=
\frac{ \P\bigpar{|\GGWx{\bp,\bpo}|=n\text{ and } D=k}}
{ \P\bigpar{|\GGWx{\bp,\bpo}|=n}}
\notag\\&
\to \frac{k\po_k}{\sumj j\po_j}%\bigpar{1+o(1)},
.\end{align}
The limits on the \rhs{} sum to 1, and thus the result \eqref{lm2} follows.

Moreover, \eqref{b3} and \eqref{b4} also yield
\begin{align}\label{b8}
  \P\bigpar{D=k\mid|\GGWx{\bp,\bpo}|=n}
&=
\frac{ \P\bigpar{|\GGWx{\bp,\bpo}|=n\text{ and } D=k}}
{ \P\bigpar{|\GGWx{\bp,\bpo}|=n}}
\notag\\&
\le \frac{k\po_k}{\sumj j\po_j}\bigpar{1+o(1)},
\end{align}
uniformly in $k$. In particular, \eqref{lm2b} holds for all $k$ if $n$ is
large enough.

Finally, by \eqref{lm2}, $\E\tD=\sum k\P(\tD=k)<\infty$ if and only if
$\sum_k k^2\po_k<\infty$.
\end{proof}

It follows from \refL{LM2} 
that the tree is overwhelmingly dominated by one branch.
(Again, this was shown by \citet{KortchemskiM}
in greater generality.)
\begin{lemma}[essentially {\citet[Proposition 5.6]{KortchemskiM}}]\label{LM3}
  Suppose that $\bp$ is a probability sequence with mean $\mu(\bp)=1$ and
  variance $\gss(\bp)\in(0,\infty)$ and that $\bpo$ is a probability
  sequence with finite mean $\mu(\bpo)$.
Let $\cttxn1,\dots,\cttxn{d(o)}$ be the branches of $\GGWppon$ arranged in
decreasing order. 
Then
\begin{align}\label{lm3}
  |\cttxn1|=n-\Op(1).
\end{align}
\end{lemma}

\begin{proof}
Let $D_n=d(o)$ and condition on $D_n=m$, for a fixed $m$.
Then, by \refLs{LM1}\ref{LM1a} and \ref{LGWf}, 
\eqref{lm3} holds. In other words, for every $\eps>0$, there exists
$C_{m,\eps}$ such that
\begin{align}
  \P\bigpar{n-|\cttxn1|>C_{m,\eps}\mid D_n=m} <\eps.
\end{align}
By \refL{LM2}, $D_n\dto\tD$, and thus 
$(D_n)_n$ is tight, \ie, $\Op(1)$, so
there exists $M$ such that
$\P(D_n>M)<\eps$ for all $n$. Consequently, if $C_{\eps}:=\max_{m\le M}
C_{m,\eps}$, 
\begin{align}
  \P\bigpar{n-|\cttxn1|>C_{\eps}}
&=
\E \P\bigpar{n-|\cttxn1|>C_{\eps}\mid D_n}
%\notag\\&
\le \eps + \P(D_n>M)
\notag\\&
\le 2\eps,
\end{align}
which completes the proof.
\end{proof}

\refLs{LM1}--\ref{LM3} make it possible to transfer many
results that are known 
for \sgt{s} (\cGWt{s}) to the modified version. 
See \refS{Smod2} for a few examples.

\begin{problem}\label{PLM2+}
Does \refL{LM2} (and thus \refL{LM3})
hold without assuming finite variance $\gss(\bp)<\infty$.
\ie, assuming only $\mu(\bp)=1$ and $\mu(\bpo)<\infty$?
As said above, \eqref{lm2} was shown also when the variance is infinite by
\citet{KortchemskiM}, 
but they then assume that $\bp$  is in the domain of
attraction of a stable distribution.
What happens without this regularity assumption?
\end{problem}

\begin{remark}\label{RPLM2}
  We assume in \refL{LM2} that $\mu(\bpo)<\infty$.
We claim that if $\mu(\bpo)=\infty$, then
$d(o)\pto\infty$; in other words, $\P(d(o)=k)\to0$ for every fixed $k$,
which can be seen as the natural interpretation of \eqref{lm2}--\eqref{lm2a}
in this case.

We sketch a proof. First, from \eqref{b3} and Fatou's lemma (for sums),
 \begin{align} \label{b9}
\liminf_{\ntoo}
n^{3/2} \P\bigpar{|\GGWx{\bp,\bpo}|=n}
%=
%\liminf_{\ntoo}\sumko
%\P\bigpar{|\GGWx{\bp,\bpo}|=n  \text{ and } D=k}
&\ge  
\sumko \liminf_{\ntoo}n^{3/2} \P\bigpar{|\GGWx{\bp,\bpo}|=n  \text{ and } D=k}
\notag\\&
=\sumko c k\po_k
=\infty.
\end{align}
In other words, 
$%\begin{align} \label{b9}
 n^{3/2} \P\bigpar{|\GGWx{\bp,\bpo}|=n}
\to \infty
$. %\end{align}
Then, \eqref{b3} and \eqref{b9} yield, for any fixed $k\ge0$,
\begin{align}
  \P\bigpar{D =k \mid|\GGWx{\bp,\bpo}|=n}
&=
\frac{ \P\bigpar{|\GGWx{\bp,\bpo}|=n\text{ and } D=k}}
{ \P\bigpar{|\GGWx{\bp,\bpo}|=n}} \to 0.
\end{align}
This proves our claim. 
\end{remark}

\section{Unrooted \sgt{s}}\label{Sunroot}

We make definitions corresponding to \refS{Sroot} for unrooted trees.
In this case, we consider labelled trees, so that we can distinguish the
vertices. (This is not needed for ordered trees, since their vertices can be
labelled canonically as described in \refS{Snot}.)
Of course, we may then ignore the labelling when we want.

Let $(w_k)\zoo$ be a given sequence of non-negative weights,
with $w_1>0$ and  $w_k>0$ for some $k\ge3$.
(The weight $w_0$ is needed only for the trivial case $n=1$, and might be
ignored. We may take $w_0=0$ without essential loss of generality.) 

For any labelled tree $T\in\fL_n$, now define the weight of $T$ as
\begin{align}\label{wT}
  w(T):=\prod_{v\in T} w_{d(v)}.
\end{align}

Given $n\ge1$, 
we define the 
random \emph{unrooted simply generated tree} $\uctn=\uctnw$
as a labelled tree in $\cL_n$, chosen randomly with probability
proportional to the weight \eqref{wT}.
(We consider only $n$ such that at least one tree of positive weight exists.)

\begin{remark}\label{Requiv}
  Just as in the rooted case, replacing the weight sequence by an equivalent
  one (still defined as in \eqref{equiv}) gives the same random tree $\uctn$.
\end{remark}

In the following sections we give three (related but different) relations
with the more well-known rooted \sgt{s}.

\section{Mark a vertex}
\label{Smark1}

Let $\uctnw$ be a random unrooted \sgt{} as in \refS{Sunroot}, and mark one
of its $n$ vertices, chosen uniformly at random. 
Regard the marked vertex as a root, and denote the resulting rooted 
tree by $\ctaw_n$.

Thus, $\ctaw_n$ is a random unordered  rooted tree, where an unordered rooted
tree $T$ has 
probability proportional to its weight given by \eqref{wT}.

We make $\ctaw_n$ ordered by ordering the children of each vertex uniformly
at random; denote the resulting random labelled ordered rooted tree by 
$\ctbw_n$.
Since each vertex $v$ has $\dout(v)!$ possible orders, the probability that
$\ctbw_n$ equals a given ordered tree $T$ is proportional to the weight
\begin{align}\label{wtb}
  w\x(T) := 
\frac{w(T)}{\prod_{v\in T}\dout(v)!}
=\frac{w_{d(o)}}{d(o)!}\prod_{v\neq o}\frac{w_{d(v)}}{\dout(v)!}
=\frac{w_{d(o)}}{d(o)!}\prod_{v\neq o}\frac{w_{\dout(v)+1}}{\dout(v)!}.
\end{align}
The tree $\ctbw_n$ is constructed as a labelled tree, but each ordered rooted
tree $T\in\fT_n$ has the same number $n!$ of labellings, and they 
have the same weight \eqref{wtb} and thus appear
with the same probability. Hence, we may forget the labelling, and regard
$\ctbw_n$ as a random ordered tree in $\fT_n$, with probabilities
proportional to the weight \eqref{wtb}.
 This is  the same as the weight \eqref{wTm} with
\begin{align}\label{phiw}
  \phi_{k}&:=\frac{w_{k+1}}{k!},
\qquad k\ge0,
\\
\label{phio}
  \phio_{k}&:=\frac{w_{k}}{k!},
\qquad k\ge0
.\end{align}
Thus, $\ctbw_n=\SGx{\bphi,\bphio}_n$,
the modified \sgt{} defined in \refS{Smod}.

We recover $\uctnw$ from $\ctbn=\SGx{\bphi,\bphio}_n$
by ignoring the root (and adding a uniformly
random labelling). This yields thus a method to construct $\uctnw$.

\begin{example}\label{Enoncrossing}
\citet{MarckertP} studied uniformly random non-crossing trees of a given
size $n$, 
and found
that if they are regarded as ordered rooted trees, then they have the
same distribution as the \cmGWt{} $\GGWppon$, where
\begin{align}
  p_k&=4(k+1)3^{-k-2}, k\ge0,
\\
  \po_k&=2\cdot3^{-k}, k\ge1.
\end{align}
These weights are equivalent to $\phi_k=k+1$ and $\phio_k=1$, which are
given by \eqref{phiw}--\eqref{phio} with $w_k=k!$.
We may thus reformulate the result by \citet{MarckertP} as:
\emph{A uniformly random non-crossing tree is the same as a random unrooted
  \sgt{} with weights $w_k=k!$.}

More generally,
\citet{KortchemskiM} 
studied \emph{\sgnct{s}}, which are random non-crossing
trees with probability proportional to the weight \eqref{wT} for some weight
sequence $\bw=(w_k)_k$, and showed that they (under a condition)
are equivalent to \cmGWt{s}.
In fact, for any weight sequence $\bw$, the proofs in 
%\cite{MarckertP} and \cite{KortchemskiM} 
\cite[in particular Lemma 2]{MarckertP} and 
\cite[in particular Proposition 2.1]{KortchemskiM} 
show that
the \sgnct, regarded as an ordered rooted tree,
is the same as $\SGphiphio_n$ with
\begin{align}
  \phi_k&:=(k+1)w_{k+1},\quad k\ge0,
\\
  \phio_k&:=w_k,\quad k\ge0.
\end{align}
Thus, comparing with \eqref{phiw}--\eqref{phio}, 
it follows that the \sgnct{} is an unordered \sgt, with weight sequence
$\olw_k:=w_kk!$.

Note that non-crossing trees are naturally defined as unrooted trees. 
A root is  introduced in \cite{MarckertP,KortchemskiM} 
for the analysis, which as said above
makes the trees \cmGWt{s}
(or, more generally, modified \sgt{s}). 
This is precisely the
marking of an urooted \sgt{} discussed in the present section.
\end{example}

\section{Mark an edge}
%{Another relation with rooted simply generated trees} 
\label{Smark2}

In the random unrooted
tree $\uctnw$, mark a (uniformly) random edge, and give it a
direction; i.e, mark two adjacent vertices, say $\rx$ and $\ry$.
Since each tree $\uctnw$ has the same number $n-1$ of edges, the resulting
marked tree $\ctcw_n$ 
is distributed over all labelled trees on $[n]$ with a marked and
directed edge with probabilities proportional to the weight \eqref{wT}.

Now ignore the marked edge, and regard the 
tree $\ctcw_n$ as two rooted trees 
$\ctnx$ and $\ctny$
with roots $\rx$ and $\ry$, respectively.
Furthermore, order randomly the children of each vertex in each of these
rooted trees; this makes $\ctnx$ and $\ctny$ a pair of ordered trees, and each
pair $(\Tx,\Ty)$ of labelled ordered rooted trees with $|\Tx|+|\Ty|=n$ and
the labels $1,\dots,n$ appears with probability proportional to
\begin{align}\label{ww}
  \hw(\Tx,\Ty)&:=\hw(\Tx)\hw(\Ty)
\intertext{where, for a rooted tree $T$,}
\hw(T)&:=\prod_{v\in T}\frac{w_{\dout(v)+1}}{\dout(v)!}
.\label{hwT}
\end{align}
Using  again the definition \eqref{phiw},
we have by \eqref{wTr}, 
\begin{align}\label{hw2}
  \hw(T)
=\prod_{v\in T}\phi_{\dout(v)}
=\phi(T).
\end{align}
Moreover, since we now have ordered rooted trees, the vertices are
distinguishable, and each pair $(\Tx,\Ty)$  of ordered trees with
$|\Tx|+|\Ty|=n$ has the same number $n!$ of labellings. Hence, we may ignore
the labelling, and regard the marked tree $\ctcw_n$ as a pair of ordered trees
$(\ctnx,\ctny)$  %$(\Tx,\Ty)$   
with
$|\ctnx|+|\ctny|=n$ and probabilities proportional to the weight given by
\eqref{ww} and \eqref{hw2}.
This means that $\ctnx$ and $\ctny$, conditioned on their sizes, are two
independent random rooted simply generated trees, with the weight
sequence $\bphi$ given by
\eqref{phiw}; in other words, $(\ctnx,\ctny)$ is a \sgf{} 
$\SGphi_{n,2}$. 

Consequently, we can construct the random unrooted simply generated tree
$\uctnw$ by taking two random rooted simply generated  trees $\ctnx$ and $\ctny$
constructed in this way, with the right distribution of their sizes, 
and joining their roots. 

Note that $|\ctnx|=n-|\ctny|$ is random, with a distribution given by the
construction above. More precisely, if $\za_n$ is the total weight 
\eqref{hw2} summed over all ordered trees of order $n$, then 
\begin{align}\label{julie}
  \P\bigpar{|\ctnx|=m}=\frac{\za_m\za_{n-m}}{\sum_{k=1}^{n-1}\za_k\za_{n-k}}.
\end{align}

\begin{remark}\label{R2}
  If $\phi_2>0$, we can alternatively describe the result as follows: 
Use the weight sequence $(\phi_k)\zoo$ given by \eqref{phiw}
and 
take a random rooted simply generated tree $\SGphi_{n+1}$
of order $n+1$, 
conditioned on the root degree $=2$; 
remove the root and join its two neighbours to each other
(this is the marked edge).

If $\phi_2=0$, we can instead take any $k>2$ with $\phi_k>0$, 
and 
take a random rooted simply generated tree $\SGphi_{n+k-1}$
of order $n+k-1$, 
conditioned on the event that the root degree is $k$, and the $k-2$ last
children of the root are leaves; we remove the root and these children, and
join the first two children. 
\end{remark}

\begin{remark}\label{RSGW}
Suppose that
the weight sequence $(\phi_j)\zoo$ given by \eqref{phiw}
satisfies $\sumj\phi_j=1$, so $(\phi_j)\zoo$ 
is a probability distribution. % on $\bbN_0:=\set{0,1,2,\dots}$.
(Note that a large class of examples 
can be expressed with such weights,
see  \refR{Requiv}.)
Then the construction 
% of $(\ctnx,\ctny)$ 
above can be stated as follows:

\emph{
Consider a \GWp{} with offspring distribution $(\phi_k)\zoo$, starting with
\emph{two} individuals, and conditioned on the total progeny being $n$.
This creates a forest with two trees; join their roots to obtain $\uctnw$.
}

Note that it follows from from the arguments above that if we mark the edge
joining the two roots, then the marked edge will be distributed uniformly
over all edges in the tree $\uctnw$.
\end{remark}

In the construction above, $\ctnx$ and $\ctny$ have the same distribution by
symmetry. 
Now define $\ctnp$ as the largest and $\ctnm$ as the smallest of $\ctnx$ and
$\ctny$. 
%(If these have equal size, define \eg{} $\ctnp=\ctnx$ for definiteness.) 
The next lemma shows that (at least under a weak condition),
$\ctnm$ is stochastically bounded, so 
$\uctnw$ is dominated by the subtree $\ctnp$.

\begin{lemma}\label{L2}
  Suppose that the generating function $\Phi(z)$ in \eqref{Phi} has a positive
radius of convergence. Then, as \ntoo,
$\ctnm\dto\GGWx{\bp}$, an unconditioned \GWt{} with offspring distribution $\bp$
equivalent to
$(\phi_k)\zoo$. In particular, $|\ctnm|=\Op(1)$, and thus
$|\ctnp|=n-\Op(1)$.
\end{lemma}

\begin{proof}
  This is a special case of \refL{LGWf}, see also \refR{Rsmall}.
\end{proof}

As remarked in \refP{Psmall}, we conjecture that 
$|\ctnm|=\Op(1)$ also when the generating function has
radius of convergence 0, but we leave this as an open problem.

\section{Mark a leaf}\label{Smark3}
This differs from the preceding two sections in that we do not recover the
distribution of $\uctnw$ exactly, but only asymptotically.

Let $N_0(T)$ be the number of leaves %(nodes with outdegree 0) 
in an unrooted tree $T$.
Let $\uctnwl$ be a random unrooted labelled tree with probability
proportional to $N_0(T)w(T)$; in other words, we bias the distribution of
$\uctnw$ by the factor $N_0(T)$.

Let $\uctnwx$ be the random rooted tree obtained by marking a uniformly
random leaf in $\uctnwl$, regarding the marked leaf as the root.
Then, any pair $(T,o)$ with $T\in\fL_n$ and $o\in T$ with $d(o)=1$
will be chosen as $\uctnwx$ and its root with probability proportional to
the weight \eqref{wT}. 
We order the children of each vertex at random as in \refSs{Smark1} and
\ref{Smark2}, and obtain an ordered rooted tree $\uctnwy$.
%By symmetry, we may assume that the root has label $n$, and 
Then each tree with root degree $1$ appears with probability
proportional to \eqref{wtb}.

Consequently, if we ignore the labelling, $\uctnwy=\SGxx{\bphi}{\bphio}_n$, 
where $\bphi$ is given by \eqref{phiw}, and $\phio_k:=\gd_{k1}$ (with a
Kronecker delta). Equivalently, $\uctnwy$ has a root of degree 1, and its
single branch is a $\SGphi_{n-1}$.

Conversely, we may obtain $\uctnwl$ from $\SGphi_{n-1}$ by adding a new root
under the old one and then adding a random labelling.

\begin{remark}\label{Rleaf}
  The construction above can also be regarded as a variant of the one in
  \refS{Smark2}, where we mark an edge such that one endpoint is a leaf.
Then, in the notation there, $|\ctnm|=1$ and $\ctnp=\SGphi_{n-1}$.
\end{remark}

As said above, $\uctnwl$ does not have the distribution of $\uctnw$, but it is
not far from it.

\begin{lemma}\label{LL1}
Let $\bw$ be any weight sequence.
  As \ntoo, the total variation distance $\dtv(\uctnwl,\uctnw)\to0$.
In other words, there exists a coupling such that
$\P\bigpar{\uctnwl\neq\uctnw}\to0$.
\end{lemma}

\begin{proof}
We may construct $\uctnw$ as in \refS{Smark2} from two random ordered trees
$\ctnp$ and $\ctnm$, where $|\ctnm|=\Op(1)$.
Conditioned on $|\ctnm|=\ell$, for any fixed $\ell$,
we have 
$\ctnp\eqd \SGphi_{n-\ell}$, where $\bphi$ is given by \eqref{phiw}.
Thus, by \cite[Theorem 7.11]{SJ264} (see comments there for   earlier
references to special cases, and to further results), as \ntoo,
conditioned on $|\ctnm|=\ell$
for any fixed $\ell$, 
\begin{align}\label{ll1}
%\Bigpar{  \frac{N_0(\ctnp)}{n}\Bigm||\ctnm|=\ell}\eqd
  \frac{N_0(\ctnp)}{n}\eqd
\frac{N_o(\SGphi_{n-\ell})}{n}\pto \pi_0,
\end{align}
for some constant $\pi_0>0$. (If $\bphi$ is a probability sequence, then
$\pi_0=\phi_0.)$ 
Furthermore, $N_0(\uctnw)=N_0(\ctnp)+N_0(\ctnm)=N_0(\ctnp)+O(1)$, since 
$N_0(\ctnm)\le|\ctnm|=\ell$.
Consequently, still
conditioned on $|\ctnm|=\ell$
for any fixed $\ell$, 
\begin{align}\label{ll2}
  \frac{N_0(\uctnw)}{n}
\pto \pi_0>0.
\end{align}
Since $ |\ctnm|=\Op(1)$, it follows that 
\eqref{ll2} holds also unconditionally.

Since $\xfrac{N_0(\uctnw)}{n}\le1$, dominated convergence yields
\begin{align}\label{ll4}
 \frac{\E N_0(\uctnw)}{n}
=
\E  \frac{N_0(\uctnw)}{n}
\to \pi_0.
\end{align}
By \eqref{ll2} and \eqref{ll4},
\begin{align}\label{ll50}
\frac{N_0(\uctnw)}{\E N_0(\uctnw)}
\pto1,
\end{align}
and thus, by dominated convergence again,
\begin{align}\label{ll5}
\E\Bigabs{ \frac{N_0(\uctnw)}{\E N_0(\uctnw)}-1}
\to0. 
\end{align}

The definition of $\uctnwl$ by biasing means that for
any bounded (or non-negative) function $f:\fL_n\to\bbR$,
\begin{align}
  \E f(\uctnwl) =\frac{\E\bigsqpar{f(\uctnw)N_0(\uctnw)}}{\E N_0(\uctnw)}.
\end{align}
and thus, for any indicator function $f$,
\begin{align}
\bigabs{\E f(\uctnwl)-\E f(\uctnw)} 
&= 
\Bigabs{\E \Bigsqpar{f(\uctnw)\Bigpar{\frac{{N_0(\uctnw)}}{\E N_0(\uctnw)}-1}}}
\notag\\&
\le 
\E \Bigabs{\frac{{N_0(\uctnw)}}{\E N_0(\uctnw)}-1}.
\end{align}
Hence, taking the supremum over all $f=\etta_{A}$,
\begin{align}
\dtv(\uctnwl,\uctnw)
\le 
\E \Bigabs{\frac{{N_0(\uctnw)}}{\E N_0(\uctnw)}-1},
\end{align}
and the result follows by \eqref{ll5}.
\end{proof}

\refL{LL1} implies that any result on convergence in probability or
distribution for one of
${\uctnwl}$ and $\uctnw$ also hold for the other.

\section{Profile of \cmGWt{s}}\label{Smod2}

We will use the following extension of \refT{T0} to \cmGWt{s}.

\begin{theorem}
\label{T0m}
Let\/ $L_n$ be the profile of a \cmGWt{} $\GGWppon$ 
of order $n$, 
and assume that $\mu(\bp)=1$, $\gss(\bp)<\infty$ and
$\mu(\bpo)<\infty$.
Then, as \ntoo,
\begin{align}\label{t0m}
  n\qqw L_n(x n\qq) \dto \frac{\gs}2\LX\Bigpar{\frac{\gs}2 x},
\end{align}
in the space $\cooq$, where $\LX$ is, 
as in \refT{T0},
the local time of a standard Brownian excursion $\be$.
\end{theorem}

\begin{proof}
  Denote the branches of $\GGWppon$ by $\cT_1,\dots,\cT_{d(o)}$ 
and let $\cT_0$ be a single root. Then, 
regarding the branches as rooted trees, which means that their vertices have
their depths shifted by 1 from the original tree,
\begin{align}\label{ma}
%  L(x)=L_{\GGWppon}(x)=
L_n(x)=
\sum_{i=1}^{d(o)} L_{\cT_i}(x-1)+L_{\cT_0}(x)
.\end{align}
Let $\cttx1,\dots,\cttx{d(o)}$ be the branches arranged in decreasing order.
\refL{LM3} shows that $|\cttx1|=n-\Op(1)$.
Hence, \eqref{ma} and the trivial estimate $0\le L_T(x)\le |T|$ for any $T$
and $x$ yield
\begin{align}\label{mb}
\bigabs{L_n(x)-L_{\ctt}(x-1)} \le 
\sum_{i=2}^{d(o)}|\cttx{i}|+1
=n-|\ctt|
=\Op(1)
.\end{align}
Furthermore, conditioned on $|\ctt|=n-\ell$, for any fixed 
%$m$ and 
$\ell$,
$\ctt$ has the same distribution as $\GGWp_{n-\ell}$, and thus \refT{T0}
shows that, 
\begin{align}\label{tom}
  (n-\ell)\qqw L_{\ctt}(x (n-\ell)\qq) 
\dto \frac{\gs}2\LX\Bigpar{\frac{\gs}2 x},
\qquad\text{in }\cooq,
\end{align}
and it follows easily that, still conditioned,
\begin{align}\label{toa}
  n\qqw L_{\ctt}(x n\qq -1) 
\dto \frac{\gs}2\LX\Bigpar{\frac{\gs}2 x},
\qquad\text{in }\cooq
.\end{align}
Together with \eqref{mb}, this shows that for every fixed $\ell$,
%conditioned on $|\ctt|=n-\ell$, 
 \begin{align}\label{tob}
   \bigpar{L_n(x)\mid |\ctt|=n-\ell}
\dto \frac{\gs}2\LX\Bigpar{\frac{\gs}2 x},
\qquad\text{in }\cooq.
 \end{align}
It follows that \eqref{tob} holds also if we condition on $n-|\ctt|\le K$,
for any fixed $K$, and then \eqref{t0m} follows easily from $n-|\ctt|=\Op(1)$.
\end{proof}

%This leads to a result on the width, where we also give a moment estimate.
Recall that for \cGWt{s} $\GGWpn$ with $\mu(\bp)=1$ and $\gss(\bp)<\infty$,
the width divided by $\sqrt n$ converges in distribution: we have
\begin{align}\label{dag}
  n\qqw W(\GGWpn) \dto \gs W,
\end{align}
for some random variable $W$ (not depending on $\bp$).
In fact, 
as noted by \citet{DrmotaG},
this is an immediate consequence of \eqref{width} and \eqref{t0},
with
\begin{align}\label{wlx}
  W:=\tfrac12\max_{x\ge0} \LX(x).
\end{align}
It is also known that all moments converge, see
\cite{DrmotaG2004} (assuming an exponential moment)
and
\cite{SJ250} (in general).

The next theorem records that \eqref{dag} extends to \cmGWt{s},
together with two partial results on moments.

\begin{theorem} \label{TMGW1}
  Consider a \cmGWt{} $\GGWppon$
where $\mu(\bp)=1$,  $\gss(\bp)<\infty$ and $\gss(\bpo)<\infty$.
Then,  
 as \ntoo,
  \begin{align}\label{dW}
    n\qqw W(\GGWppon) &\dto \gs W,
\\\label{EW}
    n\qqw \E W(\GGWppon) &\to \gs \E W =\gs\sqrt{\pi/2},
\\
\E \bigsqpar{W(\GGWppon)^{2}}& = O(n).  \label{EW2}
  \end{align}
\end{theorem}

\begin{proof}
First, \eqref{dW} follows as in \cite{DrmotaG}:
$f\to\sup f$ is a continuous functional on $\cooq$,
and thus \eqref{dW} follows from \eqref{t0m}, \eqref{width} and \eqref{wlx}.

We next prove \eqref{EW2}.
Denote the branches of $\GGWppon$ by $\cT_1,\dots,\cT_{d(o)}$.
Assume $n>1$; then the width is attained above the root, and we have, for
every $i\le d(o)$,
\begin{align}\label{eva}
  W(\cT_i) \le W(\GGWppon) \le \sum_{i=1}^{d(o)}W(\cT_i).
\end{align}
Condition on $d(o)$ and $|\cT_1|,\dots,|\cT_{d(o)}|$ as in \refL{LM1}\ref{LM1b}.
For a random variable $X$, 
denote its conditioned $L^2$ norm by
\begin{align}\label{evb}
\normmx{X}:=\bigpar{\E\bigsqpar{ X^2\mid d(o),|\cT_1|,\dots,|\cT_{d(o)}|}}\qq.  
\end{align}
By \eqref{eva} and Minkowski's inequality, we have
\begin{align}\label{evc}
  \normmx{W(\GGWppon)}
\le \sum_{i=1}^{d(o)}\normmx{W(\cT_i)}.
\end{align}
Furthermore, by \refL{LM1}\ref{LM1b} and
\cite[Corollary 1.3]{SJ250}, if $|\cT_i|=n_i$,
\begin{align}\label{evd}
  \E \bigpar{W(\cT_i)^2 \mid d(o),|\cT_1|,\dots,|\cT_{d(o)}|}  %|T_i|=n_i}
= \E \bigsqpar{W(\GGWp_{n_i})^2}
\le C n_i
,\end{align}
and thus $\normmx{W(\cT_i)} \le C n_i\qq = C|\cT_i|\qq$.
Hence, by \eqref{evc},
\begin{align}\label{eve}
 \normmx{W(\GGWppon)}
%\notag\\&
%\le  \sum_{i=1}^{d(o)} \normmx{W(T_i)}
\le  \sum_{i=1}^{d(o)} C |\cT_i|\qq
%\le C d(o)\qq \sum_{i=1}^{d(o)}  |T_i|
\end{align}
and thus,  by the \CSineq,
\begin{align}\label{evf}
&{\E\bigsqpar{ W(\GGWppon)^2\mid d(o),|\cT_1|,\dots,|\cT_{d(o)}|}}
=\bigpar{ \normmx{W(\GGWppon)}}^2
\notag\\&\qquad
\le C\Bigpar{ \sum_{i=1}^{d(o)} |\cT_i|\qq}^2
\le C d(o) \sum_{i=1}^{d(o)}  |\cT_i|
\le C d(o) n
.\end{align}
Taking the expectation yields
\begin{align}\label{evg}
\E\bigsqpar{ W(\GGWppon)^2}
\le C n \E\bigsqpar{ d(o)}
.\end{align}
Furthermore, \eqref{lm2b} implies that for large $n$, 
$%\begin{align}\label{evh}
\E\sqpar{ d(o)} \le 2\E{ \tD}  
$, %\end{align}
where $\E\tD<\infty $ by \refL{LM2}.
 Thus 
$\E \sqpar{d(o)}\le C$, and 
\eqref{evg} yields
\begin{align}\label{evi}
\E\bigsqpar{ W(\GGWppon)^2}
\le C n,
\end{align}
showing \eqref{EW2}.

Finally, \eqref{EW2} implies that the variables on the \lhs{} of \eqref{dW}
are uniformly integrable 
\cite[Theorem 5.4.2]{Gut},
and thus \eqref{EW} follows from \eqref{dW}.
$\E W=\sqrt{\pi/2}$ is well-known, see \eg{} \cite{BPY}. 
\end{proof}

\begin{problem}\label{Pmoments}
  We conjecture that under the assumptions of \refT{TMGW1}, 
$\E\bigsqpar{ W(\GGWppon)^r}=O\bigpar{n^{r/2}}$ for any $r>0$, 
which implies convergence of all moments in \eqref{dW},
as shown for
the case $\bpo=\bp$ in \cite{SJ250}.
The proof above is easily generalized if $\E \tD^{r/2}=O(1)$, which is
equivalent to $\sum_k k^{1+r/2}\po_k<\infty$, but we leave the general case
as an open problem.
\end{problem}

\begin{problem}\label{Pgss2}
Is $\gss(\bpo)<\infty$ really needed in \refT{TMGW1}?
\end{problem}

\section{Distance profile, first step}\label{S1step}

We now turn to distance profiles. 
We begin with a weak version of \refT{T1}; 
recall the pseudometric $\ddd$ defined in \eqref{ddd}, and \eqref{t1dx}.
\begin{lemma}\label{LP}
  Consider a \cGWt{} $\GGWpn$
where $\mu(\bp)=1$ and  $\gss=\gss(\bp)<\infty$.
Then,  
 as \ntoo,
for any  continuous function with compact support  $f:\ooo\to\bbR$,
 \begin{align}\label{lp}
\intoo  n\qqcw \gL_{\GGWpn}\bigpar{x n\qq}f(x) \dd x   
\dto 
\intoi\intoi f\Bigpar{\frac{2}{\gs}\ddd\bigpar{s,t;\be}}\dd s\dd t.
 \end{align}
\end{lemma}

\begin{proof}
The function $f$ is bounded,  and also uniformly continuous,
\ie, its modulus of continuity $\go(\gd;f)$, defined in
\eqref{go}, satisfies $\go(\gd;f)\to0$ as $\gd\to0$.  
Thus, 
for any rooted tree $T\in\fT_n$,
noting that $\gL_T(x)\le n$ on $[-1,0]$
and using the analogue of \eqref{Ltau} for $\gL$,
\begin{align}\label{lp2}
\hskip2em&\hskip-2em
\intoo  n\qqcw \gL_T\bigpar{x n\qq}f(x) \dd x
=  n\qww\intoo \gL_T\xpar{x}f\bigpar{n\qqw x} \dd x
\notag\\&
=  n\qww\intoom f\bigpar{n\qqw x}\gL_T\xpar{x} \dd x+O\bigpar{n\qw}
\notag\\&
=  n\qww\sumio \int_{i-1}^{i+1}f\bigpar{n\qqw x}\gL_T(i)\tau(x-i) \dd x
   +O\bigpar{n\qw}
\notag\\&
%=  n\qww\sumio \int_{i}^{i+1}f\bigpar{n\qqw i}\gL_T(i)\tau(x-i) \dd x
%+ O\bigpar{\go(n\qqw;f)}
%\notag\\&
=  n\qww\sumio f\bigpar{n\qqw i}\gL_T(i)+ O\bigpar{\go(n\qqw;f)}
+O\bigpar{n\qw}
\notag\\&
=  n\qww\sum_{v,w\in T} f\bigpar{n\qqw\ddd(v,w)}+ o(1),
\end{align}
where (as throughout the proof)
$o(1)$ tends to 0 as \ntoo, uniformly in $T\in\fT_n$. 
Recall that the \emph{contour process} $C_T(x)$ of $T$
is a continuous function $C_T:[0,2n-2]\to\ooo$ that describes the distance
from the root to a particle that travels with speed 1 on the ``outside'' of
the tree. 
(Equivalently, it performs a depth first walk at integer times
$0,1,\dots,2n-2$.) 
For each vertex $v\neq o$, the particle travels through the edge leading
from $v$ towards the root during two time intervals of unit length 
(once in each direction). Thus, as is well-known,
\begin{align}\label{lp3}
  \int_0^{2n-2} f\bigpar{n\qqw C_T(x)}\dd x 
= 2\sum_{v\neq o} f\bigpar{n\qqw \ddd(v,o)}+O\bigpar{n\go(n\qqw; f)}.
\end{align}
We will use a bivariate version of this.
It is also well-known that if $v(i)$ is the vertex visited by the particle
at time $i$, then, for any integers $i,j\in[0,2n-2]$,
\begin{align}\label{lp4}
  \ddd\bigpar{v(i),v(j)}
=
\ddd(i,j;C_T),
\end{align}
where the first $\ddd$ is the graph distance in $T$, and the second is the
pseudometric defined by \eqref{ddd} (now on the interval $[0,2n-2]$).
Hence, the argument yielding \eqref{lp3} also yields
\begin{multline}\label{lp5}
\int_0^{2n-2}\int_0^{2n-2}f\bigpar{n\qqw\ddd(x,y;C_T)}\dd x\dd y
\\
=
4  \sum_{v,w\neq o} f\bigpar{n\qqw\ddd(v,w)} + O\bigpar{n^2\go(n\qqw;f)}.
\end{multline}
We use the standard rescaling of the contour process
\begin{align}\label{lp6}
  \tC_T(t):=n\qqw C_T\bigpar{(2n-2)t}, \qquad t \in [0,1],
\end{align}
and note that for any $g:[0,1] \rightarrow [0, \infty)$ with $g(0) = g(1) = 0$ and $c>0$, 
\begin{align}
  \label{dddc}
\ddd(s,t;cg)=c\ddd(s,t;g), \qquad s,t \in [0,1].
\end{align}
Thus, by \eqref{lp5} and  a change of variables,
\begin{align}\label{lp7}
\hskip4em&\hskip-4em
\int_0^{1}\int_0^{1}f\bigpar{\ddd(s,t;\tC_T)}\dd s\dd t
\notag\\&
=
\frac{1}{(2n-2)^2}\int_0^{2n-2}\int_0^{2n-2}f\bigpar{n\qqw\ddd(x,y;C_T)}\dd x\dd y
\notag\\&
=\frac{1}{(n-1)^2}  \sum_{v,w\neq o} f\bigpar{n\qqw\ddd(v,w)} 
+ O\bigpar{\go(n\qqw;f)}.
\notag\\&
=
\frac{1}{n^2}  \sum_{v,w\neq o} f\bigpar{n\qqw\ddd(v,w)} +o(1).
\end{align}
Combining \eqref{lp2} and \eqref{lp7}, we find
\begin{align}\label{lp8}
\hskip2em&\hskip-2em
\intoo  n\qqcw \gL_T\bigpar{x n\qq}f(x) \dd x
=
\int_0^{1}\int_0^{1}f\bigpar{\ddd(s,t;\tC_T)}\dd s\dd t
+o(1)
.\end{align}

We apply this to $T=\GGWpn$, and use the result by \citet{AldousII,AldousIII},
\begin{align}\label{aldous}
  \tC_{\GGWpn}(t)\dto \frac{2}{\gs}\be(t),
\qquad \text{in $\coi$}
.\end{align}
The functional $g\to\iint f\bigpar{d(s,t;g)}\dd s\dd t$ is continuous on $\coi$,
and the result \eqref{lp} follows from \eqref{lp8} and \eqref{aldous}
by the continuous mapping theorem, using also \eqref{dddc}.
\end{proof}

\section{Distance profile of unrooted trees}\label{SD}

We continue with the distance profile, now turning to unrooted \sgt{s}
for a while.
Throughout this section we assume that $\bw$ is a weight sequence 
and that $\bphi$ and $\bphio$ are the weight sequences given by \eqref{phiw}
and \eqref{phio}.
We assume that
the exponential generating function of $\bw$
has positive radius of convergence; this means that the generating function
$\Phi(z)$ in \eqref{Phi} has positive radius of convergence, which in turn 
implies that there exists a probability weight sequence $\bp$ equivalent to
$\bphi$. We assume furthermore that it is possible to choose $\bp$ such that
$\mu(\bp)=1$; $\bp$ will denote this choice. (For algebraic conditions on
$\Phi$ for such a $\bp$ to exist, see \eg{} \cite{SJ264}.)

We note that by \eqref{phiw}--\eqref{phio}, $\phio_k\le\phi_{k-1}$, $k\ge1$.
Hence, if $p_k=ab^k\phi_k$, then $\sum_k b^k\phio_k<\infty$, and it is
possible to find $a_0>0$ such that $\bpo:=\tbphio$ given by \eqref{equiv0}
also is a probability sequence; hence $\SGphiphio_n=\GGWppo_n$ is a
modified \GWt. Furthermore, $\po_k\le (a_0/a)p_k$, and thus
if $\gss(\bp)<\infty$, then $\gss(\bpo)<\infty$.

We begin with an unrooted version of \refL{LP}.

\begin{lemma}
  \label{LQ}
  Let $\bw$, $\bphi$ and $\bp$ be as above, and assume $\gss:=\gss(\bp)<\infty$.
Let $\gL_n$ be the distance profile of the unrooted \sgt{} $\uctnw$.
Then,  
 as \ntoo,
for any  continuous function with compact support  $f:\ooo\to\bbR$,
 \begin{align}\label{lq}
\intoo  n\qqcw \gL_{n}\bigpar{x n\qq}f(x) \dd x   
\dto 
\intoi\intoi f\Bigpar{\frac{2}{\gs}\ddd\bigpar{s,t;\be}}\dd s\dd t.
 \end{align}
\end{lemma}

\begin{proof}
 Consider the leaf-biased random tree $\uctnwl$ defined in \refS{Smark3}.
By \refL{LL1}, we may assume $\P(\uctnwl\neq\uctnw)\to0$ and thus it suffices
to show \eqref{lq} with $\gL_{\uctnwl}$ instead of $\gL_n$.
If $\ctnp$ denotes the unique branch of $\uctnwl$, then, trivially,
\begin{align}\label{lqq}
0\le  \gL_{\uctnwl}(x)-\gL_{\ctnp}(x) \le 2n-1, \qquad x\ge0,
\end{align}
and thus we may further reduce and replace $\gL_n$  in \eqref{lq} by
$\gL_{\ctnp}$. 
As shown in
\refS{Smark3},
$\ctnp\eqd\SGphi_{n-1}=\GGWp_{n-1}$, and the result now follows from
\refL{LP}, 
replacing $n$ there by $n-1$ and 
$x$ by $x=(n/(n-1))\qq x$, noting that 
$\sup_x|f(x)-f\bigpar{(n/(n-1))\qq x}|\to0$ as \ntoo.
%replacing $n$ there by $n-1$ and changing variables by
%$x=(n/(n-1))\qq y$, noting that then $\sup_x|f(x)-f(y)|\to0$.
\end{proof}

\begin{theorem}\label{Teo}
  Let $\bw$, $\bphi$ and $\bp$ be as above, and assume $\gss:=\gss(\bp)<\infty$.
Let $\gL_n$ be the distance profile of the unrooted \sgt{} $\uctnw$.
Then, as \ntoo, 
\begin{align}
  \label{teo}
n^{-3/2}\gL_n\bigpar{x n\qq}
\dto
\frac{\gs}2\GLX\Bigpar{\frac{\gs}2 x},
\end{align}
in the space $\cooq$,
where $\GLX(x)$ is as in \refT{T1}.
%we have for every bounded measurable
%$f:\ooo\to\bbR$, 
%\begin{align}\label{t1d}
%\intoo f(x) \GLX(x) \dd x
%=2\iint_{0<s<t<1} f\bigpar{\be(s)+\be(t)-2\min_{u\in[s,t]} \be(u) }\dd s\dd t
%.\end{align}
%
\end{theorem}

\begin{proof}
Let
\begin{align}
  Y_n(x):=n^{-3/2}\gL_n\bigpar{x n\qq}
=n^{-3/2} \gL_{\uctnw}\bigpar{xn\qq}.
\end{align}
Regard $Y_n$ as a random element of $\cooq$.
Define also the mapping $\psi:\cooq\to\moo$, the space of all 
locally finite  Borel measures on $\ooo$, defined by $\psi(h):=h(x)\dd x$;
\ie,  for $h\in\cooq$ and
$f\in\coo$ with compact support,
\begin{align}\label{ior}
  \intoo f(x) \dd \psi(h):=\intoo f(x)h(x)\dd x.
\end{align}
In other word, $\psi(h)$ has density $h$.

We give $\moo$ the vague topology, \ie, $\nu_n\to\nu$ in $\moo$ if
$\int f\dd\nu_n\to\int f\dd\nu$ for every $f\in\coo$ with compact support,
and note that $\moo$ is a Polish  space, see \eg{} 
\cite[Theorem~A2.3]{Kallenberg}. Clearly, the separable Banach space $\cooq$
is also a Polish space.
(Recall that a Polish space has a topology that can be defined by a complete
separable metric.)
It follows from the definition \eqref{ior} that $\psi$ is continuous
$\cooq\to\moo$. Furthermore, $\psi$ is injective, since the density of
a measure is \aex{} uniquely determined.

We will use the alternative method of proof in \cite[p.~123--125]{Drmota},
and show the following two  properties:
\begin{claim}\label{claim1}
  The sequence $Y_n$ is tight in $\cooq$.
\end{claim}

\begin{claim}\label{claim2}
The sequence of random measures
$\psi(Y_n)$ converges in distribution in 
$\moo$ to some random measure $\zeta$.
\end{claim}

It then follows from \cite[Lemma 7.1]{SJ185}
(see also \cite[Theorem 4.17]{Drmota})
that 
\begin{align}
  \label{YtoZ}
Y_n\dto Z
,\qquad \text{in }\cooq,
\end{align}
for some random $Z\in\cooq$ such that 
\begin{align}\label{zeta2}
\psi(Z)\eqd\zeta.  
\end{align}
It will then be easy to complete the proof.

\pfitemx{Proof of Claim \ref{claim1}}
  For $i=1,\dots,n$, let
$\cT(i)$ be $\uctnw$ rooted at $i$.
By symmetry, all $\cT(i)$ have the same distribution; moreover, they equal
in distribution $\uctnwx$ defined in \refS{Smark1} (which has a random root).
Hence, if we order each $\cT(i)$ randomly, we have by \refS{Smark1}
\begin{align}\label{teo1}
  \cT(i)\eqd \ctbw_n =\SGphiphio_n=\GGWppo_n
.\end{align}

By \eqref{dpropro},
\begin{align}\label{joyce}
Y_n(x)=
n^{-3/2} \gL_{\uctnw}\bigpar{xn\qq}=\frac{1}{n}\sumin n\qqw L_{T(i)}\bigpar{xn\qq}.
\end{align}
Since the sequence $n\qqw L_{\GGWppo_n}\bigpar{xn\qq}$ converges in $\cooq$ by
\refT{T0m}, it is tight in $\cooq$.
Furthermore,
\begin{align}
  \sup_x\bigabs{ n\qqw L_{\GGWppo_n}\bigpar{xn\qq}}
=n\qqw W(\GGWppo_n),
\end{align}
which are uniformly integrable by \eqref{EW2} in \refT{TMGW1}.
Hence, by \refL{LU2} (which we state and prove in \refApp{Atight})
and \refR{RU2}, 
\eqref{joyce} and \eqref{teo1} imply that the sequence
$Y_n$
%$n^{-3/2} \gL_{\uctnw}\bigpar{xn\qq}$
is tight in $\cooq$, proving Claim \ref{claim1}.

\pfitemx{Proof of Claim \ref{claim2}}
Let $f\in\coo$ have compact support.
Then, \refL{LQ} shows that
\begin{align}\label{zeta}
  \intoo f(x)\dd\psi(Y_n)
&
=\intoo f(x)Y_n(x)\dd x
\dto
\intoi\intoi f\Bigpar{\frac{2}{\gs}\ddd\bigpar{s,t;\be}}\dd s\dd t
\notag\\&
=\intoo f(x)\dd\zeta(x),
\end{align}
where $\zeta$ is the (random) probability measure on $\ooo$ defined as the
push-forward of Lebesgue measure on $\oi\times\oi$ by the map
$(s,t)\to(2/\gs)\ddd(s,t;\be)$; in other words, $\zeta$ is the conditional
distribution, given $\be$, of $(2/\gs)\ddd(U_1,U_2;\be)$ where $U_1$ and $U_2$
are independent uniform $U\oi$ random variables.
This convergence in distribution for each $f$ with compact support is
equivalent to convergence in $\moo$, see \cite[Theorem 16.16]{Kallenberg}.
Thus, $\psi(Y_n)\dto \zeta$ in $\moo$, proving Claim \ref{claim2}.

As said above, the claims imply \eqref{YtoZ}--\eqref{zeta2}.
Thus, by the definition of $\zeta$ above, see \eqref{zeta},
for any bounded measurable $f:\ooo\to\bbR$,
\begin{align}\label{loke}
\intoo f(x) Z(x)\dd x&
=\intoo f(x)\dd\psi(Z)
%\notag\\&
=\intoo f(x)\dd\zeta(x)
\notag\\&
=
\intoi\intoi f\Bigpar{\frac{2}{\gs}\ddd\bigpar{s,t;\be}}\dd s\dd t
.\end{align}
Define
\begin{align}
  \GLX(x):=\frac{2}{\gs} Z\Bigpar{\frac{2}{\gs}x}.
\end{align}
Then, \eqref{YtoZ} is the same as \eqref{teo}.
Furthermore,
replacing $f(x)$ by $f(\gs x/2)$ in \eqref{loke} yields after a simple
change of variables \eqref{t1dx} and thus \eqref{t1d}.
(In particular, $\GLX$ does not depend on the weight sequence $\bw$.)
\end{proof}

\section{Distance profile of rooted trees}\label{SDR}

Finally, we can prove \refT{T1} as a simple consequence of the corresponding
result \refT{Teo} for unrooted trees.

\begin{proof}[Proof of \refT{T1}]
Let $\bphi=\bp$, and use \eqref{phiw} to define the weight sequence $\bw$.
Then, \refT{Teo} applies to $\uctnw$.
 Consider, as in the proof of \refL{LQ}, the leaf-biased unrooted random
 tree $\uctnwl$. By \refL{LL1}, \refT{Teo} holds also for $\uctnwl$.

Let again $\ctnp\eqd \SGphi_{n-1}=\GGWp_{n-1}$ 
denote the unique branch of $\uctnwl$.
By \eqref{lqq}, \refT{Teo} holds also for $\ctnp$ and thus for
$\GGWp_{n-1}$.
Replace $n$ by $n+1$; then a change of variables shows that \refT{Teo} holds
for $\GGWp_n$ too, which is \refT{T1}.
\end{proof}

As part of the proof,
we have shown the corresponding result \refT{Teo} for unrooted trees. The result also extends easily to \cmGWt{s}, 
using the method by 
\citet{MarckertP} and
\citet{KortchemskiM}. 
\begin{theorem}
  \label{T1m}
Let\/ $\gL_n$ be the distance profile of a \cmGWt{} $\GGWppon$ 
of order $n$, 
and assume that $\mu(\bp)=1$, $\gss(\bp)<\infty$ and
$\mu(\bpo)<\infty$.
Then, 
\eqref{t1} holds
as \ntoo.
\end{theorem}

\begin{proof}
  This is a simple consequence of \refT{T1} and \refL{LM3}. It follows from
  \eqref{lm3} that 
  \begin{align}
\sup_x\bigabs{\gL_{\GGWppon}(x)-\gL_{\cttxn1}(x)}
\le 2n \bigpar{n-|\cttxn1|} = \Op(n),
  \end{align}
and thus \eqref{t1} for $\GGWppon$ follows from the same result for
$\cttxn1$,
which follows from \eqref{lm3} and \refT{T1} by conditioning on $|\cttxn1|$.
We omit the details.
\end{proof}

\section{Wiener index} \label{WienerI}

Recall that the \emph{Wiener index} of a tree $T$ is defined as 
\begin{align}\label{wien}
  \Wiener(T):=\frac12\sum_{v,w\in T} \ddd(v,w).
\end{align}
where $\ddd$ is the graph distance in $T$. Thus
\begin{align}\label{wien2}
  \Wiener(T)=\frac12\sumi i \gL_T(i)
= \frac12\intoom x\gL_T(x)\dd x.
\end{align}
%\int_{-1}^0 x\gL_T(x)\dd x = -\frac{1}{12}|T|  
%where the last (negligible) term comes from the linear interpolation 
%and $x\in\oi$.

Since the integrand $x$ in \eqref{wien2} is unbounded, convergence of the
Wiener index does not follow immediately from convergence of the profile,
but only a simple extra truncation argument is needed.
For a \cGWt{} $\GGWpn$, with $\mu(\bp)=1$ and $\gss(\bp)<\infty$, 
convergence is more easily proved directly from 
Aldous's result \eqref{aldous}  \cite{AldousII,AldousIII}, 
see \cite{SJ146},
but as an application of the results above, we show the corresponding result
for unrooted trees.

\begin{theorem}\label{Twie}
  Let $\bw$ and $\bp$ be as in \refT{Teo}.
Then
\begin{align}\label{twie}
  n^{-5/2}\Wiener\bigpar{\uctnw}
\dto
\frac{1}{\gs}\intoo x\GLX(x)\dd x
%=\frac{2}{\gs}\iint_{0<s<t<1} \bigpar{\be(s)+\be(t)-2\min_{u\in[s,t]}
%\be(u) }\dd s\dd t
%=\frac{2}{\gs}\iint_{0<s<t<1}d(s,t; \be)\dd s\dd t
=\frac{1}{\gs}\intoi\intoi d(s,t; \be)\dd s\dd t
.\end{align}
\end{theorem}

\begin{proof}
Let $T\in\fT_n$, and define a modified version by
\begin{align}\label{wien2'}
  \Wiener'(T):= \frac12\intoo x\gL_T(x)\dd x
=\Wiener(T)+O(n),
\end{align}
using \eqref{wien2} and noting that $\gL_T(x)\le n$ for $x\in[-1,0]$.
 It suffices to prove \eqref{twie} for $\Wiener'(\uctnw)$.
By \eqref{wien2'},
  \begin{align}\label{wien3}
n^{-5/2}\Wiener'(T)&
=\frac{n^{-5/2}}2\intoo x\gL_T(x)\dd x
%\notag\\&
=\frac12\intoo xn\qqcw\gL_T(n\qq x)\dd x
.  \end{align}
Define a truncated version by,
for $m \geq 0$,
  \begin{align}\label{wien4}
%{n^{-5/2}}
\Wiener_m(T)
:=\frac{n^{5/2}}2\intoo (x\bmin m)n\qqcw\gL_T(n\qq x)\dd x.
  \end{align}
Then, since the support of $\gL_T$ is $[-1,\diam(T)+1]$,
\begin{align}\label{wien5}
  \P\bigpar{\Wiener'(\uctnw)\neq\Wiener_m(\uctnw)}
\le 
\P\bigpar{\diam(\uctnw)> n\qq m-1}.
\end{align}
Since $\diam(\uctnw) = \Op(n\qq)$, as an easy consequence of 
any of the constructions in \refSs{Smark1}--\ref{Smark3} and known results
on the height of rooted \GWt{s}, see \eg{} \cite{Kolchin}, \cite{AldousII} or \cite[Theorem 1.2]{SJ250},
it follows that
\begin{align}\label{wien6}
\sup_n  \P\bigpar{\Wiener'(\uctnw)\neq\Wiener_m(\uctnw)}\to0,
\qquad\text{as \mtoo}
.\end{align}
Furthermore, for each fixed $m$, \eqref{wien4} and \refT{Teo} imply,
as \ntoo,
\begin{align}\label{wien7}
{n^{-5/2}}
\Wiener_m(\uctnw)
&\dto
\frac{1}2\intoo (x\bmin m)
\frac{\gs}2\GLX\Bigpar{\frac{\gs}2 x}\dd x
\notag\\&
=
\frac{1}\gs\intoo \bigpar{x\bmin \frac{\gs m}2}
\GLX\xpar{ x}\dd x.
\end{align}
The convergence in \eqref{twie} follows by \eqref{wien6} and \eqref{wien7},
see \cite[Theorem 4.2]{Billingsley}.
The equality in \eqref{twie} holds by \eqref{t1dx}.
%, which by monotone convergence holds for every positive $f$.
\end{proof}
Of course, the limit in \eqref{twie} agrees with the limit in \cite{SJ146}
for the rooted case.

\section{Moments of the distance profile} \label{Momentdp}

In this section we prove the following estimates on moments of the distance
profile for a \cGWt; we use again the simplified notation $L_n$ and $\gL_n$
for the profile and distance profile as in \refTs{T0} and \ref{T1}.
Throughout the section, $C$ and $c$ denote some positive constants that may
depend on the offspring distribution $\bp$ only; 
$C_r$ denotes constants depending on $\bp$ and the parameter $r$ only.
(As always, these may change from one occurrence to the next.)

\begin{theorem}\label{TM1}
Let\/ $\gL_n$ be the distance profile of a \cGWt{} of order $n$, 
with an offspring distribution 
$\bp$ such that $\mu(\bp)=1$ and $\gss(\bp)<\infty$.
%Then the following hold, for some constants $C_r$ and $c$.

\begin{romenumerate}
  
\item \label{TM1a}
Let $r\ge 1$ be a real number.
Then, for all $i, n \ge 1$, 
\begin{align} \label{Mooh}
\E[  \gL_n(i)^{r} ]\le C_r n^{3r/2}e^{-ci^{2}/n}.
\end{align}
%and some positive constants $c$ and $C$ depending on $r$ and $\bp$ only.

\item \label{TM1b}
Let $r\ge1$ be an integer, and suppose that $\bp$ has a finite $(r+1)th$
moment: 
\begin{align}
  \label{Moo}
\sum_{k} k^{1+r}p_{k} < \infty. 
\end{align}
Then, for all $i, n \ge 1$, 
\begin{align} \label{Mooa}
\E[  \gL_n(i)^{r} ]\le C_r i^{r} n^{r}.
\end{align}
%for some positive constant $C$ depending on $r$ and $\bp$ only. 
Furthermore, we may in this case combine \eqref{Mooh} and \eqref{Mooa} to
\begin{align} \label{Moob}
\E[ \gL_n(i)^{r} ]\le C_r i^r n^{r}e^{-ci^{2}/n},
\end{align}
for all $i, n \ge 1$.
\end{romenumerate}
\end{theorem}
The proof is given later.
Note that \eqref{Moob} trivially implies \eqref{Mooa} and (changing
$C_r$ and $c$) \eqref{Mooh}; conversely \eqref{Mooh} and \eqref{Mooa}
imply \eqref{Moob} by considering $i\le n\qq$ and $i>n\qq$ separately.  

The special case $r=1$ of \eqref{Mooa},
i.e., 
\begin{align}\label{egln}
\E  \gL_n(i)\le C in
,\end{align}
was proved in \cite[Theorem 1.3]{SJ222}; note that when $r=1$, \eqref{Moo}
holds automatically by our assumption $\gss(\bp)<\infty$.

\begin{remark}\label{RL}
 The estimates above are natural analogues of estimates for the profile
$L_n$.
First, as proved in \cite[Theorem 1.6]{SJ250}, 
under the conditions in \refT{TM1}\ref{TM1a},
\begin{align} \label{Looh}
\E[ L_n(i)^{r} ]&\le C_r n^{r/2}e^{-ci^{2}/n},
\end{align}
for all $i,n\ge1$.
Secondly, 
as proved in \cite[Theorem 1.13]{SJ167},
under the conditions in \refT{TM1}\ref{TM1b},
\begin{align}\label{Looa}
\E[L_n(i)^{r}] &\le C_r i^{r},
\end{align}
for all $i,n\ge1$.
%as proved in \cite[Theorem 1.13]{SJ167}.
The estimates \eqref{Looh}--\eqref{Looa} are used in our proof of \refT{TM1}.
\end{remark}

\begin{remark}
  We do not know whether the moment assumption \eqref{Moo} really is
  necessary for the result. This assumption is necessary for the
  corresponding estimate \eqref{Looa}
for the profile $L_n(i)$, as noted in \cite{SJ250}, but the argument there
does not apply to the distance profile. We state this as an open problem.
\end{remark}

\begin{problem}
  Does \eqref{Mooa} hold without the assumption \eqref{Moo}?
\end{problem}

\begin{remark}
  We also do not know whether the assumption that $r$ is an integer is
necessary in \refT{TM1}\ref{TM1a}; we conjecture that it is not. (This
assumption is used in the proof of \eqref{Looa} in \cite{SJ167}.) 
\end{remark}

As an immediate consequence of \refTs{T1} and \ref{TM1}, we obtain the
corresponding results for the asymptotic profile $\GLX$.

\begin{theorem}
  \label{TM2}
For any $r \ge 1$ and all $x\ge0$,
\begin{align}\label{tm2}
  \E[\GLX(x)^r]
\le C_r \min\bigpar{x^r, e^{-c x^2}}
.\end{align}
\end{theorem}

\begin{proof}
Fix an offspring distribution $\bp$ with $\mu(\bp)=1$ and all moments
finite. (For example, we may choose a well-known example such as $\Po(1)$ or
$\Bi(2,\frac12)$.)
Let $x\in(0,\infty)$. %, and consider only $n\ge 1/x$. 
Define $i_n:=\ceil{2\gs\qw xn\qq}$. Then $i_n/n\qq\to 2x/\gs$ as \ntoo, 
and \eqref{t1} implies
\begin{align}\label{tmm}
  n^{-3/2}\gL_n(i_n)
\dto \frac{\gs}2\GLX(x).
\end{align}
Hence,  Fatou's lemma and \eqref{Mooh} yield, for any $r\ge1$,
\begin{align}\label{tmp}
   \E[\GLX(x)^r] \le 
(2/\gs)^r \liminf_\ntoo \E\bigsqpar{  n^{-3r/2}\gL_n(i_n)^r}
\le C_r e^{-cx^2}.
\end{align}
Similarly, 
Fatou's lemma and \eqref{Mooa} yield
\begin{align}\label{tmq}
   \E[\GLX(x)^r] 
\le C_r x^r
\end{align}
for integer $r\ge1$;
equivalently, the $L^r$ norm is estimated by $\norm{\GLX(x)}_r \le C_r x$.
This estimate extends to all real $r\ge1$ by 
Lyapounov's inequality.
Hence, \eqref{tmp} and \eqref{tmq} both hold for all $r\ge1$, which yields
\eqref{tm2} for $x>0$.

Finally, 
the result \eqref{tm2} is trivial for $x=0$, since $\GLX(0)=0$ a.s.
\end{proof}

%\begin{align}
%  \E\GLX(x) \le Cx, \qquad x\ge0
%.\end{align}

\begin{remark}\label{RM2}
 The same argument shows that \refT{T0} and \eqref{Looh}--\eqref{Looa} imply
\begin{align}\label{rm2}
  \E[\LX(x)^r]
\le C_r \min\bigpar{x^r, e^{-c x^2}},
\end{align}
for any  $r \ge 1$ and all $x\ge0$.
\end{remark}

The proof of \refT{TM1} relies on an invariance property of the law of
\GWt{s} under random re-rooting 
proved by Bertoin and Miermont \cite[Proposition 2]{BertCut}. 
A pointed tree is a pair $(T, v)$, where $T$ is an ordered rooted tree
(also called planar rooted tree)
and $v$ is a vertex of $T$. 
We endow the space of pointed trees with the $\gs$-finite measure
$\P^{\bullet}$ defined by  
\begin{align} \label{Moa}
\P^{\bullet}((T,v)) = \P( \GGWp = T),
\end{align}
where $\GGWp$ is a \GWt{} with offspring distribution $\bp$ such that
$\mu(\bp)=1$. We let $\E^{\bullet}$ denote the expectation under this
``law''. In particular, the conditional law $\P^{\bullet}(\, \cdot \, \mid |T|
= n)$ on the space of pointed trees with $n$ vertices is well defined, and
equals the distribution of $(\GGWp_{n}, v)$ where given the \cGWt{}
$\GGWp_{n}$ of order $n$, $v$ is a uniform random vertex of $\GGWp_{n}$. 

Let us now describe the transformation of pointed trees of Bertoin and
Miermont \cite[Section 4]{BertCut}. 
They work with planted planar trees; 
the base in the planted tree is useful since it implicitly
specifies the ordering of the transformed tree. 
However, 
we ignore this detail
and formulate their transformation for rooted trees;
our formulation is easily seen to be equivalent to theirs. 
For any rooted planar tree $T$ and vertex $v$ of
$T$, let $T_{v}$ be the fringe subtree of $T$ rooted at $v$, and let $T^{v}$
be the subtree of $T$ obtained by deleting all the strict descendants of $v$
in $T$. We define a new pointed tree $(\hat{T},{\hat{v}})$ in the following
way.  
If $v$ is the root, we do nothing, and let $\bigpar{\hat{T},\hat{v}}=(T,v)$.
Otherwise,
first remove the edge $e(v)$ between $v$ and its parent $pr(v)$ in $T$, and
instead connect $v$ to the root of $T$. 
We then re-root the resulting tree at $pr(v)$ and obtain the new
rooted tree $\hat{T}$, which we point at $\hat v = v$. 
Note that $\hat T_{\hat v}=T_v$, and that
$\hat{T}^{\hat{v}}\setminus\set{\hat v}$ equals
$T^v\setminus\set{v}=T\setminus T_v$ rerooted at $pr(v)$.

\citet[Proposition 2 and its proof]{BertCut} establish that this
transformation preserves the measure $\P^\bullet$; 
this includes the following. 

\begin{proposition}[\citet{BertCut}] \label{PropositionBer}
Under $\P^{\bullet}$, $(\hat{T}^{\hat{v}}, T_{v})$ and $(T^{v}, T_{v})$ have
the same ``law''.  
Furthermore, the trees $T^{v}$ and $T_{v}$ are independent, with $T_{v}$
being a \GWt{} with offspring distribution $\bp$.
%$\GGWp$.  
\nopf
\end{proposition}

\begin{proof}[Proof of \refT{TM1}]
For a rooted plane tree $T$ and $i \ge 0$ an integer, 
it will be convenient to write $Z_{i}(T):=L_T(i)$
for the number of vertices at distance $i$ from the root of $T$. For a
vertex $v$ in $T$, note that the number of vertices at distance $i \ge 1$
from $v$ is less than or equal to $Z_{i-1}(\hat{T}^{\hat{v}}) +
Z_{i}(T_{v})$.
(Strict inequality may occur because of the extra vertex $\hat v$
added in $\hat{T}^{\hat{v}}$.) 
From \eqref{dpropro}, we thus obtain
that if $\gL_{T}$ is the distance profile of $T$ defined
in \eqref{dpro}, then
\begin{align}  \label{Mob}
\gL_{T}(i)\le \sum_{v \in T} \bigpar{ Z_{i-1}(\hat{T}^{\hat{v}}) + Z_{i}(T_{v})},
\qquad i \geq 1. 
\end{align}
If $r\ge1$ and $|T|=n$, then \eqref{Mob} and Jensen's inequality yield
\begin{align}  \label{Mobb}
\bigpar{n\qw\gL_{T}(i)}^r
\le \frac{1}{n}\sum_{v \in T} \bigpar{ Z_{i-1}(\hat{T}^{\hat{v}}) + Z_{i}(T_{v})}^r.
%\qquad i \geq 1. 
\end{align}
Consequently, using
\eqref{Moa},
\begin{align} \label{Moc}
\E[  \gL_n(i)^{r} ] & \le n^{r} \E^{\bullet} \bigsqpar{ \bigpar{ Z_{i-1}(\hat{T}^{\hat{v}}) + Z_{i}(T_{v})}^{r} \, \big | \, |T| = n } \notag \\
& \le 2^{r} n^{r}  \E^{\bullet} \bigsqpar{  Z_{i-1}(\hat{T}^{\hat{v}})^{r} + Z_{i}(T_{v})^{r} \, \big | \, |T| = n }.
\end{align} 

By Proposition \ref{PropositionBer}, we have on the one hand that
\begin{align} \label{Mod}
\E^{\bullet} \bigsqpar{  Z_{i-1}(\hat{T}^{\hat{v}})^{r} \, \big | \, |T| = n } & = \E^{\bullet} \bigsqpar{  Z_{i-1}(T^{v})^{r} \, \big | \, |T| = n } \notag \\
& \le \E^{\bullet} \bigsqpar{  Z_{i-1}(T)^{r} \, \big | \, |T| = n } \notag \\
& = \E [ Z_{i-1}( \GGWp_{n})^{r}  ].
\end{align} 
On the other hand, since $|T| = |T^{v}| + |T_{v}| -1$, we see from Proposition \ref{PropositionBer} that
\begin{align} 
\E^{\bullet} \bigsqpar{  Z_{i}(T_{v})^{r} \, \big | \, |T| = n } & = \sum_{m=1}^{n} \E^{\bullet} \bigsqpar{  Z_{i}(T_{v})^{r} \, \big | \, |T_{v}| = m,  |T^{v}| = n-m+1 }  \notag \\
& \hspace*{20mm} \times  \P^{\bullet} \bigsqpar{  |T_{v}| = m \, \big | \, |T| = n } \notag \\
& =  \sum_{m=1}^{n} \E \bigsqpar{  Z_{i}(\GGWp_{m})^{r} } \P^{\bullet} \bigsqpar{  |T_{v}| = m \, \big | \, |T| = n } \notag \\
& \le \sup_{1 \le m \le n} \E \bigsqpar{  Z_{i}(\GGWp_{m})^{r} } \label{Moe}. 
\end{align} 
Combining \eqref{Moc}, \eqref{Mod} and \eqref{Moe}, we have that 
\begin{align} \label{Mof}
\E[ \gL_n(i)^{r} ] 
&\le 2^{r} n^{r} \bigpar{\E \bigsqpar{ Z_{i-1}( \GGWp_{n})^{r}  }
+ \sup_{1 \le m \le n} \E \bigsqpar{  Z_{i}(\GGWp_{m})^{r} }}
\notag\\
&= 2^{r} n^{r} \bigpar{\E \bigsqpar{L_n(i-1)^{r} }
+ \sup_{1 \le m \le n} \E \bigsqpar{  L_m(i)^{r} }}
.\end{align} 
Therefore, 
\eqref{Mooh} and \eqref{Mooa} follow from
\eqref{Looh} and \eqref{Looa}, proved in 
\cite[Theorem 1.6]{SJ250}
and \cite[Theorem 1.13]{SJ167},
respectively.

Finally \eqref{Moob} follows from \eqref{Mooh} and \eqref{Mooa} as noted above.
\end{proof}

\section{H\"older continuity}\label{SHolder}

We now discuss H\"older continuity properties of the continuous random
functions $\LX$ and  $\GLX$.
We begin with $\LX$, the local time of $\be$.
The results are to a large extent known, although we do not know any
reference to the form of them stated here; nevertheless we treat also 
$\LX$ in detail, as a background to and preparation for the discussion of
$\GLX$ below.

It is well-known that the local time of Brownian motion at some fixed time
\as{} is H\"older continuous of order $\alpha$ 
for any $\alpha <1/2$,
%but not for any $\ga\ge\frac12$,
%see \cite[VI.(1.8) and VI.(1.12)]{RY}.
see \cite[VI.(1.8)]{RY}.
It is also known that this \Holder{} continuity 
extends by standard arguments to the 
local times of Brownian bridge and Brownian excursion, 
and thus (see \refT{T0}) to $\LX$.
We need a quantitative version of this.

For $\ga>0$ and an interval $I\subseteq\bbR$, let 
the \Holder\ space
$\Ha=\Ha(I)$ be the
space of functions $f:I\to\bbR$ such that
\begin{align}
  \normHa{f}:=\sup\Bigcpar{\frac{|f(x)-f(y)|}{|x-y|^\ga}: x,y\in I, x\neq y}
<\infty
\end{align}
This is a semi-norm. We may regard $\Ha$ as a space of functions modulo
constants, and then $\normHa\cdot$ is a norm and $\Ha$ a Banach space.

\begin{theorem}\label{TH1}
  Let $0<\ga<\frac12$ and let $A<\infty$.
Then
\begin{align}\label{th1}
  \Ez \normHaA{\LX} <\infty.
\end{align}
In particular, $\LX\in\Haoo$ \as
\end{theorem}

\begin{proof}
We let $\ga$ and $A$ be fixed. Constants $C$ below may depend on $\ga$ and $A$.
%We write for simplicity $\Ha =\HaA$.

Recall that $\LX(x)=L^x_1$, where $L^x_t$ is the local time of a (standard)
Brownian excursion $\be$. The proof will actually show the result for 
$L^x_t$ for any fixed $t\in\oi$.

We first note that 
if $\beta$ is a $3$-dimensional Bessel process started from $0$ (see e.g.\
\cite[Section VI.3]{RY}), 
then a standard Brownian excursion $\be$ is given by
$\be(1)=0$ and 
\begin{align} 
  \be(t) = (1- t) \beta \Bigpar{ \frac{t}{1-t} }, \hspace*{4mm} 0 \leq t < 1;
\end{align}
see \cite[II.(1.5)]{Blu}. 
One then can deduce from \cite[VI.(3.3)]{RY} and an
application of the It\^o integration by parts formula that $\be$ satisfies
\begin{align} \label{be2}
  \be(t) 
= \int_{0}^{t} \Bigpar{ \frac{1}{\be(s)} - \frac{\be(s)}{1-s} } \dd s +  B(t), 
\qquad 0 \le t \le 1,
\end{align}
where $B$ is a standard Brownian motion. 
In particular, $\be$ is a continuous
semi-martingale. 

We now follow the proof of \cite[VI.(1.7), see also VI.(1.32)]{RY}, but avoid
localizing at the cost of further calculations.
Let 
\begin{align}\label{bv}
  \bv(t):=
\frac{1}{\be(t)} - \frac{\be(t)}{1-t} ,
\qquad 0 \le t <1,
\end{align}
so \eqref{be2} can be written
\begin{align} \label{be4}
  \be(t) 
= B(t) +\int_{0}^{t} \bv(s) \dd s
%\qquad 0 \le t \le 1
.\end{align}
Then, by Tanaka's formula \cite[VI.(1.2)]{RY},
writing $x^+:=x\bmax 0$,
\begin{align}\label{be5}
  L^x_t = 2 \bigpar{\be(t)-x}^+
- 2 \bigpar{-x}^+
&
-2\intot \indic{\be(s)>x}\dd B(s)
\notag\\&\qquad
-2\intot \indic{\be(t)>x} \bv(s)\dd s.
\end{align}
Fix $t\in\oi$ and denote the four random functions of $x$ on the \rhs{} of
\eqref{be5} by 
$F_1(x),\dots,F_4(x)$. 
Trivially, the first two are in $\cH_1$ (= Lipschitz), with norm at most 2.
Hence,
\begin{align}\label{be6}
\normHaA{F_1}+\normHaA{F_2}\le C.  
\end{align}

For each $x$, $F_3(x)$ is a continuous martingale in $t$.
Thus, if $0\le x<y\le A$, the
Burkholder--Davis--Gundy inequality  \cite[IV.(4.1)]{RY}
and the occupation times formula \cite[VI.(1.9)]{RY}
yield, for any $p\ge2$,
\begin{align}\label{be7}
  \E |F_3(x)-F_3(y)|^p&
\le C_p \E \Bigsqpar{\Bigpar{\intot \indic{x<\be(s)\le y}\dd s}^{p/2}}
\notag\\&
= C_p \E \Bigsqpar{\Bigpar{\int_x^y L^z_t\dd z}^{p/2}}
\notag\\&
= C_p (y-x)^{p/2}\E \Bigsqpar{\Bigpar{\frac{1}{y-x}\int_x^y L^z_t\dd z}^{p/2}}
\notag\\&
\le C_p (y-x)^{p/2}\E \Bigsqpar{{\frac{1}{y-x}\int_x^y \bigpar{L^z_t}^{p/2}\dd z}}.
\notag\\&
= C_p (y-x)^{p/2}{\frac{1}{y-x}\int_x^y \E\bigpar{L^z_t}^{p/2}\dd z}.
\end{align}
We have $L^z_t \le L^z_1=\LX(z)$, and thus \eqref{rm2} 
implies $\E(L^z_t)^{p/2}\le C_p$ for all $z$; hence
\eqref{be7} yields
\begin{align}\label{be8}
  \E |F_3(x)-F_3(y)|^p&
\le C_p (y-x)^{p/2}.
\end{align}
Consequently, $F_3$ satisfies the Kolmogorov continuity criterion, 
and the version of it stated in \cite[I.(2.1)]{RY} shows that if $p$ is
chosen so large that $(p/2-1)/p > \ga$, then \eqref{be8} implies
\begin{align}\label{be9}
  \E\, \normHaA{F_3} 
\le \bigpar{ \E\, \normHaA{F_3} ^p}^{1/p}
<\infty.
\end{align}

Similarly, using the extension of the occupation times formula in 
\cite[VI.(1.15)]{RY}, if again $0\le x<y\le A$,
\begin{align}\label{be10}
  F_4(y)-F_4(x)
&=2\intot\indic{x<\be(s)\le y}\bv(s)\dd s
\notag\\&
=2\intot\indic{x<\be(s)\le y}
\Bigpar{\frac{1}{\be(s)} - \frac{\be(s)}{1-s}}\dd s
\notag\\&
=2\int_x^y\dd z
\intot \Bigpar{\frac{1}{z} - \frac{z}{1-s}}\dd L^z_s
.\end{align}
We now simplify and assume $0\le t\le \frac12$. Then \eqref{be10} implies
\begin{align}\label{be11}
\bigabs{  F_4(y)-F_4(x)}
\le2\int_x^y
\Bigpar{\frac{1}{z} + 2z}L^z_t \dd z
\le C\int_x^y \frac{1}{z} L^z_t \dd z
\le C\int_x^y \frac{\LX(z)}{z}  \dd z
.\end{align}
Let $p':=1/\ga$ and let $p:=(1-\ga)\qw>1$ be the conjugate exponent. 
Then, by \eqref{be11} and \Holder's inequality,
\begin{align}\label{be12}
\bigabs{  F_4(y)-F_4(x)}
\le C(y-x)^\ga\Bigpar{\int_x^y \frac{\LX(z)^p}{z^p}  \dd z}^{1/p}
%\le C(y-x)^\ga\Bigpar{\int_0^A \frac{\LX(z)^p}{z^p}  \dd z}^{1/p}
.\end{align}
Consequently,
\begin{align}\label{be13}
  \normHaoo{F_4}
\le C\Bigpar{\intoo \frac{\LX(z)^p}{z^p}  \dd z}^{1/p}.
\end{align}
Thus, using again \eqref{rm2},
\begin{align}\label{be14}
\Ez \normHaoo{F_4}^p
\le C\E{\intoo \frac{\LX(z)^p}{z^p}  \dd z}
= C\intoo \frac{\E\LX(z)^p}{z^p}  \dd z
<\infty
\end{align}
and thus
\begin{align}\label{be15}
  \Ez \normHaA{F_4}
\le   \Ez \normHaoo{F_4}<\infty.
\end{align}
Consequently, \eqref{be5}, \eqref{be6}, \eqref{be9} and \eqref{be15} yield
\begin{align}
\Ez \normHaA{L_t^{\cdotx}}<\infty
\end{align}
for $0\le t\le\frac12$.

We have for simplicity assumed $t\le\frac12$.
To complete the proof, we note that $\be$ is invariant under reflection:
$\be(1-t)\eqd\be(t)$ (as processes), 
and thus $L^x_1-L^x_{1/2}\eqd L^x_{1/2}$ (as processes in $x$).
Consequently,
\begin{align}\label{bre18}
\Ez \normHaA{\LX}=
  \Ez \normHaA{L_1^{\cdotx}}
\le  2 \Ez \normHaA{L_{1/2}^{\cdotx}}
<\infty,
\end{align}
showing \eqref{th1}. (The case $\frac12<t<1$ follows similarly.)

Finally, \eqref{th1} shows that \as, $\LX\in\HaA$ for every $A>0$.
Moreover, $\LX$ has finite support $[0, \sup\be]$, and thus $\LX\in\Haoo$.
\end{proof}

We have for simplicity considered a finite interval $[0,A]$ in \refT{TH1},
but the result can easily be  extended to $\Haoo$.
\begin{theorem}\label{THoo}
  Let $0<\ga<\frac12$.
Then
\begin{align}\label{thoo}
  \Ez \normHaoo{\LX} <\infty.
\end{align}
\end{theorem}

\begin{proof}
Note that the proof of \refT{TH1}
actually shows $\Ez\normHaA{\LX}^p<\infty$ for some $p>1$, see 
\eqref{be9} and \eqref{be14}.
Moreover, the same proof applied to the interval $[m,m+1]$ shows that for
any $m\ge1$,
\begin{align}\label{oro}
  \Ez\normHax{\LX}{[m,m+1]}^p\le Cm^p,
\end{align}
with $C$ independent of $m$, where the factor $m^p$ comes from \eqref{be11}.
Furthermore, $\LX=0$ on $[m,m+1]$ unless $\sup\be>m$. 
The explicit formula for the distribution function of $\sup\be$,
see \eg{} \cite{Chung-excursion}, \cite{Kennedy} or \cite[p.~114]{Drmota},
yields the well-known subgaussian decay 
\begin{align}\label{sug}
\P\bigpar{\sup\be>x}\le e^{-cx^2},
\qquad x\ge1.
\end{align}
(See also \cite{SJ250} for the corresponding result for
heights of \cGWt{s}.) 
Combining \eqref{oro} and \eqref{sug} with
\Holder's inequality, we obtain, with $1/q=1-1/p$, 
\begin{align}\label{hsa}
  \Ez\normHax{\LX}{[m,m+1]}
&
\le \Bigpar{\Ez\normHax{\LX}{[m,m+1]}^p}^{1/p}
\P\bigpar{\normHax{\LX}{[m,m+1]}\neq0}^{1/q}
\notag\\&
\le Cm \P\bigpar{\sup\be>m}^{1/q}
\le Cm e^{-cm^2},
\end{align}
Finally, it is easy to see that
\begin{align}\label{hsb}
\normHaoo{\LX}
\le \summo\normHax{\LX}{[m,m+1]}  
\end{align}
and thus, by \eqref{th1} and \eqref{hsa},
\begin{align}\label{hsc}
\Ez\normHaoo{\LX}
\le \summo\Ez\normHax{\LX}{[m,m+1]}  
\le C + \summ Cme^{-cm^2}<\infty.
\end{align}
\end{proof}

\begin{remark}\label{Rbridge}
The proofs above apply also to the local time %$\hL^x_t$ say, 
of the Brownian bridge. The main differences are that we consider functions
on $\oooo$ and that the term $1/\be(t)$ in
\eqref{bv} disappears. This shows that, writing $\Lb(x)$
for the local time of the Brownian bridge at time $t=1$,
\begin{align}\label{bbr}
  \Ez\normHax{\Lb}{\oooo} <\infty.
\end{align}
By \citet{Vervaat}, the Brownian excursion can be constructed from
a Brownian bridge by  shifts in both time and space,
and thus $\LX$ is a (random) shift of $\Lb$.
Consequently, with this coupling, 
\begin{align}\label{bb=}
\normHaoo{\LX}=\normHax{\Lb}{\oooo},  
\end{align}
and thus \eqref{thoo} and \eqref{bbr} are equivalent.
This yields an alternative proof of \refT{THoo}.
\end{remark}

{
We have so far considered $\ga<1/2$.
For $\ga\ge1/2$, it is well-known that Brownian motion is \as{} not locally
\HCga; 
see \eg{} \cite[I.(2.7)]{RY} or \cite[Section 1.2]{MPeres} for even more
precise results. Moreover, the same holds for the local time $\bL^x_t$
of Brownian
motion, regarded as a function of the space variable $x$ (for any fixed $t>0$); 
we do not know an explicit reference but
this follows easily, for example, 
from the first Ray--Knight theorem
\cite[XI.(2.2)]{RY}, which says that if we stop at $\tau_1$, the hitting
time of 1, then the process $(\bL^{1-x}_{\tau_1})$, $0\le x\le 1$, 
is a 2-dimensional squared
Bessel process, which (at least away from 0) has the same smoothness as a
Brownian motion (since it can be written as the sum of squares of two
independent Brownian motions); we omit the details.
}

Similarly,  $\LX=L^x_1$, which is the local time of a standard Brownian
excursion is \as{} not \HCh. One way to see this is that if we stop a
Brownian motion when its local time at 0 reaches 1, \ie, 
at $\tau:=\inf\set{t:\bL^0_t=1}$, then the part of the Brownian motion
before time $\tau$ and above
any fixed $\gd>0$ is \as{} included in a finite number of excursions, and
these are independent, conditioned on the number of them and their lengths.
Hence, if the local time of a Brownian excursion were \HCh{} with positive
probability, then so would $\bL_\tau^x$, restricted to $x\ge\gd$, be, 
and then $\bL_t^x$, $x\ge\gd$, would be \HCh{} with positive probability for
some rational $t$, and thus for all $t>0$ by scaling, which contradicts the
argument above.

\subsection{Distance profile}

We next consider the asymptotic distance profile $\GLX$.
We first note that the nice re-rooting invariance property of the standard
Brownian excursion $\be$ presented in Section \ref{SSCRT} yields the
following corollary of \refT{THoo}.

\begin{corollary}\label{CTH2}
  Let $0<\ga<\frac12$.
Then
\begin{align}
  \Ez \normHaoo{\GLX} <\infty.
\end{align}
In particular, the random function $\GLX\in\Haoo$ \as
\end{corollary}

\begin{proof}
It follows from the
identity \eqref{dproI}  that
\begin{align}  \label{ineq1}
\normHaoo{\GLX} \le \int_{0}^{1}  \normHaoo{  \LXs} \dd s,
\end{align}
where for every $s \in\oi$, $\LXs$ denotes the local time of the
process $\be^{[s]}$ defined in \eqref{PathT}, which is distributed as a
standard Brownian excursion \eqref{Reroot}. 
Hence,
\refT{THoo} yields
\begin{align}  \label{ineq11}
\Ez \normHaoo{\GLX} \le \int_{0}^{1} \Ez \normHaoo{  \LXs} \dd s
=
\Ez \normHaoo{  \LX} <\infty
.\end{align}
%
%The final statement is proved as in \refT{TH1}, since $\GLX$ has compact
%support 
%$[0,\sup_{s,t} \ddd(s,t;\be)]\subseteq[0,2\sup\be]$, see \eqref{t1dx} and
%\eqref{ddd}. 
\end{proof}

However, it turns out that the averaging in \eqref{dproI} actually makes
$\GLX$ smoother than $\LX$; we have the following stronger result,
which improves \refC{CTH2}.

\begin{theorem}\label{TF}
The asymptotic distance profile $\GLX\in \Ha$ for every $\ga<1$, a.s.
Furthermore, $\GLX$
is \as{} absolutely continuous and
with a derivative $\GLXd$ (defined \aex) that belongs to $L^p(\bbR)$ for
every $p<\infty$.    
\end{theorem}

In order to prove \refT{TF}, 
we first prove an estimate of the Fourier transform
of $\GLX$.
We choose to define Fourier transforms as, for  $f\in L^1(\bbR)$, 
\begin{align}\label{ft}
\widehat f(\xi):=\intoooo f(x) e^{\ii \xi x}\dd x,
\qquad -\infty<\xi<\infty.
\end{align}
Note that by \eqref{t1dx}, the Fourier transform $\hGLX$ can be written as
\begin{align}\label{fa}
  \hGLX(\xi):=\intoo e^{\ii\xi x}\GLX(x)\dd x
=\iint_{s,t\in\oi} e^{\ii\xi\ddd(s,t;\be)}\dd s\dd t
.\end{align}

\begin{lemma}
  \label{LF}
There exists a constant $C$ such that
  \begin{align}\label{fb}
    \Ez|\hGLX(\xi)|^2 \le C\bigpar{\xi^{-4}\land 1},
%\qquad\xi\in\oooo
\qquad -\infty<\xi<\infty
.  \end{align}
\end{lemma}
\begin{proof}
Note first that the profile $\GLX$ is a (random) non-negative 
function with integral 1, \eg{} by \eqref{t1d}.
Hence, for every $\xi$,
\begin{align}\label{fc}
  \bigabs{\hGLX(\xi)}\le\intoo \bigabs{\GLX(x)}\dd x=1.
\end{align}
Thus, \eqref{fb} is trivial for $|\xi|\le1$.
Assume in the remainder of the proof $|\xi|\ge1$.

By \eqref{fa},
\begin{align}
  \label{fd}
\Ez\bigabs{\hGLX(\xi)}^2&
=\E \idotsint_{t_1,\dots t_4\in\oi} e^{\ii \xi (\ddd(t_1,t_2;\be)-\ddd(t_3,t_4;\be))}
\dd t_1\dotsm\dd t_4
\notag\\&
=\E  e^{\ii \xi (\ddd(U_1,U_2;\be)-\ddd(U_3,U_4;\be))},
\end{align}
where $U_1,\dots,U_4$ are \iid{} uniform $U(0,1)$ random variables,
independent of $\be$.

Recall from \refSS{SSCRT} that the Brownian excursion $\be$ defines the
continuum random tree $\Tbe$, with a  quotient map
$\rho_{\be}:\oi\to \Tbe$. Let $\bU_i:=\rho_\be(U_i)$, $i=1,\dots,4$, be the
points in $\Tbe$ corresponding to $U_i$; 
these are \iid{} uniformly random points in $\Tbe$, 
and $\ddd(U_i,U_j;\be)$ equals
the distance $\ddd(\bU_i,\bU_j)$
between $\bU_i$ and $\bU_j$ in the real tree $\Tbe$.
Then \eqref{fd} becomes
\begin{align}
  \label{fe}
\Ez\bigabs{\hGLX(\xi)}^2&
%\notag\\&
=\E  e^{\ii \xi (\ddd(\bU_1,\bU_2)-\ddd(\bU_3,\bU_4))}.
\end{align}
Furthermore, we can simplify the calculations by rerooting $\Tbe$ at
$\bU_4$; this preserves the distribution, see \eqref{Reroot}, and
$\bU_1,\dots,\bU_3$ are still independent uniformly random points in the
tree; hence, we also have, with $o=\rho_\be(0)$ the root of $\Tbe$.
\begin{align}
  \label{feo}
\Ez\bigabs{\hGLX(\xi)}^2&
%\notag\\&
=\E  e^{\ii \xi (\ddd(\bU_1,\bU_2)-\ddd(\bU_3,o))}.
\end{align}

To calculate \eqref{feo}, it suffices to consider the (real) subtree of
$\Tbe$ spanned by the root $o$ and $\bU_1,\dots,\bU_3$.
\citet[Corollary 22]{AldousIII} showed that the distribution of the random
real tree $\aT_k$ spanned in this way by $k\ge1$ \iid{} uniform random points 
$\bU_1,\dots,\bU_k\in\Tbe$ can be described as follows.
Let the \emph{shape} $\tau$ of $\aT_k$ be the tree regarded as a combinatorial
tree, \ie, ignoring the edge lengths.
Then $\tau$ is \as{} a rooted binary tree with $k$ leaves, labelled by
$1,\dots,k$ (corresponding to $\bU_1,\dots,\bU_k$).
Let $\TAU{k}$ be the set of possible shapes, \ie, the set of 
binary trees with $k$ labelled leaves. Then $|\TAU{k}|=(2k-3)!!$.
Each shape $\tau$ has $k-1$ internal vertices, and thus $2k$ vertices
and $2k-1$ edges. For each shape $\tau$, label the edges $1,\dots,2k-1$ in
some order, and for each tree $\aT_k$ with shape $\tau$, let
$\ell_1,\dots,\ell_{2k-1}$ be the corresponding edge lengths.
Then $\aT_k$ is described by $(\tau,\ell_1,\dots,\ell_k)$, and the
distribution of $(\tau,\ell_1,\dots,\ell_k)$ is
given by the density,
for  $\tau\in\TAU{k}$ and $\ell_i>0$,
\begin{align}\label{aldk}
%  s e^{-s^2/2}\dd\ell_1\dotsm\dd\ell_{2k-1}, 
2^{2k}  s e^{-2s^2}\dd\ell_1\dotsm\dd\ell_{2k-1}, 
\qquad \text{with } s:=\ell_1+\dotsm\ell_{2k-1}
%\text{ for } \ell_i>0 \text{ and }\tau\in\TAU{k}.
.\end{align}
Recall that our normalization differs from
\cite{AldousIII}, where $T_{2\be}$ is used instead of $T_{\be}$, and thus all
  edges are twice as long; hence \eqref{aldk} is obtained  from
  \cite[(33)]{AldousIII} by a trivial change of variables.

For $k=3$ we have $|\TAU3|=3!!=3$ different shapes.
It is easy to verify that in each of them, we may label the five edges such that
$\ddd(\bU_1,\bU_2)-\ddd(\bU_3,o)=\ell_1+\ell_2-\ell_3-\ell_4$.
Consequently, \eqref{feo} and \eqref{aldk} yield, with $s$ as in \eqref{aldk},
\begin{align}  \label{ff}
\Ez\bigabs{\hGLX(\xi)}^2&
%\notag\\&
%=\E  e^{\ii \xi (\ddd(\bU_1,\bU_2)-\ddd(\bU_3,o))}.
= 3\int_{\bbR_+^5} e^{\ii\xi(\ell_1+\ell_2-\ell_3-\ell_4)} 2^6 s e^{-2s^2}
\dd\ell_1\dotsm\dd\ell_5
\notag\\&
=192 \intoo \psi(s) s e^{-2s^2}\dd s,
\end{align}
where, with
$\Sigma_5(s)
:=\set{(\ell_1,\dots,\ell_5):\ell_i>0 \text{ and } \ell_1+\dots+\ell_5=s}$,
\begin{align}\label{fg}
  \psi(s):=
 \int_{\Sigma_5(s)} e^{\ii\xi(\ell_1+\ell_2-\ell_3-\ell_4)} 
\dd\ell_1\dotsm\dd\ell_4.
\end{align}
We can easily find the Laplace transform of $\psi(s)$: for $\gl>0$,
\begin{align}\label{fh}
  \intoo e^{-\gl s}\psi(s)\dd s &
= \int_{\bbR_+^5} e^{-\gl(\ell_1+\dots+\ell_5)}e^{\ii\xi(\ell_1+\ell_2-\ell_3-\ell_4)} 
\dd\ell_1\dotsm\dd\ell_5
\notag\\&
=\frac{1}{(\gl-\ii\xi)^2(\gl+\ii\xi)^2\gl}.
\end{align}
A partial fraction expansion of \eqref{fh} yields,
%using that \eqref{fh} is even in $\xi$,
\begin{align}
\label{fi}
  \intoo e^{-\gl s}\psi(s)\dd s &
=\frac{a(\xi)}{\gl}
+\frac{b_+(\xi)}{\gl-\ii\xi}
+\frac{b_-(\xi)}{\gl+\ii\xi}
+\frac{c_+(\xi)}{(\gl-\ii\xi)^2}
+\frac{c_-(\xi)}{(\gl+\ii\xi)^2},
\end{align}
%where $a(\xi)=\xi^{-4}$. 
for some coefficients $a(\xi),b_\pm(\xi),c_\pm(\xi)$.
It is easy to calculate these,
but it suffices to note that \eqref{fh} is homogeneous of degree $-5$ in
$(\gl,\xi)$, and thus $a(\xi)=a\xi^{-4}$,
$b_\pm(\xi)=b_\pm\xi^{-4}$ and $c_\pm(\xi)=c_\pm\xi^{-3}$ for some
complex $a,b_\pm,c_\pm$.
By inverting the Laplace transform \eqref{fi}, we find
\begin{align}\label{fj}
  \psi(s)
=
a\xi^{-4} + b_+\xi^{-4}e^{\ii\xi s}+ b_-\xi^{-4}e^{-\ii\xi s}
+ c_+\xi^{-3} se^{\ii\xi s} + c_-\xi^{-3} se^{-\ii\xi s}
%a\xi^{-4} + b\xi^{-4}\bigpar{e^{\ii\xi s}+e^{-\ii\xi s}}
%+ c\xi^{-3}\bigpar{se^{\ii\xi s}+se^{-\ii\xi s}}.
.\end{align}
Finally, we substitute \eqref{fj} in \eqref{ff}. 
We define
\begin{align}
h_m(s):=s^m e^{-2s^2}\indic{s>0},
\end{align}
and obtain, recalling \eqref{ft},
\begin{align}\label{fk}
  \Ez\bigabs{\hGLX(s)}^2
=192\xi^{-4}
\bigpar{a\hh_1(0)+b_+\hh_1(\xi)+b_-\hh_1(-\xi)
+c_+\xi\hh_2(\xi)+c_-\xi\hh_2(-\xi)
}.
\end{align}
We have $\hh_1(\xi)=O(1)$ since $h_1$ is integrable, and 
and  $\xi\hh_2(\xi)=\ii\widehat{h_2'}(\xi)=O(1)$
since $h_2$ is differentiable with integrable derivative.
Consequently, the result \eqref{fb} follows.
\end{proof}

\begin{remark}
It is easy to see that in the proof above, $a=1$ and 
$\hh_1(0)=\intoo s e^{-2s^2}\dd s=1/4$; moreover,
$\hh_1(\xi)=O(\xi^{-2})$ and $\hh_2(\xi)=O(\xi^{-3})$.
Hence, the proof yields
\begin{align}\label{fl}
  \Ez\bigabs{\hGLX(\xi)}^2
=48\xi^{-4} + O\bigpar{\xi^{-6}}.
\end{align}
Thus, the estimate in \refL{LF} is sharp.
  \end{remark}

  \begin{remark}
The result \eqref{fb} (or \eqref{fl})
can also be obtained in the same way from \eqref{fd}.
However,
for $k=4$ we have $|\TAU4|=5!!=15$ different shapes that we have to
consider, and they yield several different terms in \eqref{ff}.
The leading term in the partial fraction expansion \eqref{fi} will now be
$a\xi^{-4}/\gl^3$. 
  \end{remark}

  \begin{proof}[Proof of \refT{TF}]
Let $0<\ga<3/2$. Then, by \refL{LF}, 
\begin{align}\label{tfw}
  \E \intoooo \bigabs{|\xi|^{\ga}\hGLX(\xi)}^2\dd\xi
&=
2 \intoo |\xi|^{2\ga}\Ez\bigabs{\hGLX(\xi)}^2\dd\xi
\notag\\&
%\le C \intoo \bigpar{\xi^{2\ga} \bmin \xi^{2\ga-4}}\dd\xi <\infty
\le C \intoi \xi^{2\ga}\dd\xi+ C\int_1^\infty \xi^{2\ga-4}\dd\xi <\infty
.\end{align}
Hence, \as, 
%$\intoooo \bigabs{|\xi|^{\ga}\hGLX(\xi)}^2\dd\xi<\infty$, which 
%(together with $\GLX\in L^2$)
%says that
the function $\GLX$ belongs to the (generalized) Sobolev space 
$\SOB{\ga}$ 
% Stein 1970: \cL^2_\ga
% BL: H^\ga_2
(also called potential space; many different notations exist),
which is defined as the space of all functions $f\in L^2(\bbR)$
such that
\begin{align}\label{SOBga}
\norm{f}_{\SOB\ga}^2:=
\norm{f}_{L^2}^2+
\intoooo \bigabs{|\xi|^{\ga}\hGLX(\xi)}^2\dd\xi
<\infty
,\end{align}
see \eg{} \cite[Chapter V]{Stein} or \cite[Chapter 6]{BL}.
Furthermore, this Sobolev space
equals the Besov space $\BES{\ga}$, see
% Stein 1970: \gL^{2,2}_\ga
% BL: B^\ga_{2,2}
again 
\cite[Chapter V (there denoted $\gL_\ga^{2,2}$)]{Stein} or \cite[Chapter 6]{BL}.
$\BES\ga$ is an $L^2$ version of $\Ha$, and
for $0<\ga<1$,
$\BES{\ga}$
may be defined as  the set of functions in $L^2$ such that
\begin{align}\label{B22ga}
\norm{f}_{\BES\ga}^2:=
\norm{f}_{L^2}^2+
 \intoo \lrpar{\frac{\norm{f(\cdot+u)-f(\cdot)}_{L^2(\bbR)}}{u^\ga}}^2
  \frac{\dd u}{u} <\infty;
\end{align}
the equivalence (within constant factors) of \eqref{B22ga} with
\eqref{SOBga}
follows easily from the Plancherel formula.
(For a much more general result, see
\cite[Theorem 6.2.5]{BL} or \cite[(V.60)]{Stein}.)
For $1<\ga<2$, \eqref{B22ga} fails, but 
$\BES\ga=\set{f\in L^2:f'\in\BES{\ga-1}}$, which again equals $\SOB\ga$.

If $1\le\ga<3/2$, it follows that the derivative 
(in distribution sense)
$\GLXd\in\SOB{\ga-1}$. 
By 
the Sobolev (or Besov) embedding theorem
\cite[Theorem 6.5.1]{BL}, see also \cite[Theorem V.2]{Stein} and, for a simplified version, \cite[Lemma 6.2]{SJ185}, this implies
$\GLXd\in L^p(\bbR)$ with $1/p = 1/2-(\ga-1)$. Since $\ga$ may be chosen
arbitrarily close to $3/2$, it follows that $\GLXd\in L^p$ for every
$p<\infty$. (Recall that $\GLX$, and thus $\GLXd$, has compact support, so
small $p$ is not a problem.) The continuous function $\GLX$ thus has a
distributional derivative $\GLXd$ that is integrable, which implies
that
$\GLX$ is the integral of this derivative, so $\GLX$ is absolutely
continuous, with a derivative in the usual sense \aex, which equals the
distributional derivative $\GLXd$.

Finally, $\GLX\in\Ha$ for $\ga<1$ by $\GLX'\in L^p$ (with $p=1/(1-\ga)$)
and \Holder's inequality, or directly by $\GLX\in\SOB{\ga+1/2}=\BES{\ga+1/2}$
and the Besov embedding theorem
\cite[Theorem 6.5.1]{BL},  \cite[V.6.7]{Stein}, noting $\Ha=\BESoo{\ga}$.
  \end{proof}

\begin{remark}
For the Fourier transform $\hLX$, we have instead of \eqref{fb} and
\eqref{fl} the analogous estimate
  \begin{align}\label{fbl}
 \Ez|\hLX(\xi)|^2 \le C\bigpar{\xi^{-2}\land 1},
%\qquad\xi\in\oooo
\qquad -\infty<\xi<\infty
,  \end{align}
and, more precisely,
\begin{align}\label{fll}
 \Ez\bigabs{\hLX(\xi)}^2
=4\xi^{-2} + O\bigpar{\xi^{-4}}.
\end{align}
These are proved by the same argument as above, now with 
$\bigpar{d(\bU_1,o)-d(\bU_2,o)}$ in the exponent in \eqref{fe}
and using \eqref{aldk} with $k=2$; we omit the details.

Note that the exponent of $\xi$ is $-2$ in \eqref{fll}, but $-4$ in \eqref{fl}.
Hence, $\hGLX(\xi)$ decays faster than $\hLX(\xi)$ (at least in an average
sense), which intuitively means that $\GLX$ is smoother than $\LX$, as seen
in the results above. 

Note also that the result of \citet{Vervaat} mentioned in \refR{Rbridge}
implies that $|\hLX(\xi)|=|\hLb(\xi)|$.
The expectation $\E|\hLb(\xi)|^2$ can easily be calculated and estimated
directly, since $\bb$ is a Gaussian process, which leads to another proof of
\eqref{fbl} and \eqref{fll}; we leave the details to the reader.
\end{remark}

The proof above shows that
\begin{align}\label{fog1}
\GLX\in \BES\ga,
\qquad \ga<3/2,   
\end{align}
%for every $\ga<3/2$, 
and thus 
\begin{align}\label{fog2}
\GLXd\in \BES\ga,
\qquad \ga<1/2.
\end{align}
More precisely, it follows from the proof that
\begin{align}\label{fog3}
\Ez\norm{\GLX'}_{\BES{\ga-1}}^2
\le C
\Ez\norm{\GLX}_{\BES\ga}^2<\infty 
\qquad\text{for $\ga<3/2$},  
\end{align}
%for $\ga<3/2$, 
but not for any $\ga\ge3/2$.

For comparison, \eqref{fll} implies that
%\begin{align}
%\LX\in \BES\ga,
%\qquad 0<\ga<1/2,
%\end{align}
\begin{align}\label{fog4}
\Ez\norm{\LX}_{\BES\ga}^2<\infty 
\qquad\text{for $\ga<1/2$},  
\end{align}
but not for  $\ga\ge\frac12$.
In this case, the range of $\ga$ in \eqref{fog4}
is thus the same as the range of \Holder{} continuity,
note that since $\Ha=\BESoo\ga$ for $\ga\in(0,1)$, this range
equals the set of $\ga\in(0,1)$ such that $\LX\in\BESoo\ga$.
Another, similar, example is the Brownian bridge $\bb$;
a simple calculation
shows  that %if $\bb$ is a Brownian bridge, then 
$\Ez \norm{\bb}_{\BES{\ga}}^2<\infty$ for $\ga<\frac12$, but not for larger
$\ga$.

This suggests (but does not prove) that $\GLX$ is even smoother
than shown by \refT{TF}. 
However, the derivative $\GLX'$ is \emph{not} \Holder{} continuous
on $\oooo$, as
might be guessed from the analogy between  \eqref{fog3} and \eqref{fog4}.
In fact, the (two-sided) derivative $\GLX'$ does not exist at 0,
and thus $\GLX'$
is not even continuous, at least not at $0$.
\begin{theorem}\label{TNC}
  $\GLX$ does not \as{} have a (two-sided) derivative at $0$.
\end{theorem}

To see this, we note first the following.
\begin{lemma}\label{LNC}
  For every $x\ge0$,
  \begin{align}
    \label{le}
\E\GLX(x) = \E \LX(x) = 4x e^{-2x^2}.
  \end{align}
\end{lemma}
\begin{proof}
  The result for $\LX$ is well-known; it was shown by
  \citet[(6.2)]{Chung-excursion}, and it is equivalent to the case $k=1$ of
  \eqref{aldk}. 

The result for $\GLX$ follows by \eqref{dproI} and \eqref{Reroot}. 
\end{proof}

\begin{proof}[Proof of \refT{TNC}]
Suppose that $\GLX$ is differentiable at 0.
Since  $\GLX(x)=0$ for $x<0$, the derivative has to be 0, and thus
$\GLX(x)/x\to0$ as $x\downto0$.
Furthermore, \refT{TM2} shows that $\E\bigpar{\GLX(x)/x}^2 \le C$, and thus
$\GLX(x)/x$ is uniformly integrable for $x>0$. 
Consequently, if $\GLX$ were \as{} differentiable at 0, then 
$\E\bigpar{\GLX(x)/x}\to0$ as $x\downto0$, which contradicts \refL{LNC}.
\end{proof}

Nevertheless, it is quite possible that $\GLX$ is continuously
differentiable on $\ooo$, with a one-sided derivative at 0.
We end with some open problems suggested by the results above.
\begin{problem}\quad
  \begin{romenumerate}

  \item 
  Is $\GLX$ \as{} Lipschitz, \ie, is the derivative $\GLXd$ \as{} bounded?
\item 
Does $\GLXd(0)$ exist \as{} as a right derivative?
  \item 
    Is the derivative $\GLXd$ \as{} continuous on $\ooo$?
  \item 
    Is the derivative $\GLXd$ \as{} in $\Haoo$ for some $\ga>0$?
For every $\ga<\frac12$?
  \end{romenumerate}
\end{problem}

\subsection{Finite $n$} 
We have in this section so far  considered only the asymptotic profiles $\LX$
  and $\GLX$. Consider now the 
profile $L_n$ and distance profile $\gL_n$
of a \cGWt{} $\GGWpn$, for a given offspring distribution $\bp$. 
We assume throut the setion that $\mu(\bp)=1$ and that $0<\gss(\bp)<\infty$;
we will often add the condition that $\bp$ has a finite fourth moment 
$\mu_4(\bp):=\sum_k k^4 p_k$.

%For finite $n$, 
It was shown by
\citet{DrmotaG}, see also \cite[Theorem 4.24]{Drmota}, that if the offspring
distribution has an exponential  moment, then
\begin{align}\label{bg4}
  \E \bigabs{n\qqw L_n(xn\qq)-n\qqw L_n(yn\qq)}^4
\le C |x-y|^2,
\end{align}
which by the standard argument in \eqref{be8}--\eqref{be9}
yields 
\begin{align}\label{bg5}
  \Ez\normHaA{n\qqw L_n(n\qq\cdot)} <\infty
\end{align}
for $\ga<1/4$.
This and \refT{T0} then yield
\eqref{th1} for $\ga<1/4$.
We conjecture that \eqref{bg4} can be extended to higher moments, yielding
\eqref{bg5} and thus
another proof of \eqref{TH1} for all $\ga<\frac12$, but this seems to
require some non-trivial work.

Moreover, 
as said in \cite[footnote on page 127]{Drmota}, the proof of \eqref{bg4} 
does not really require an exponential moment, but it seems to require at
least a fourth moment for the offspring distribution.
We do not know whether such a moment condition really is necessary for
\eqref{bg4}, and we state the following problem.
%\end{remark}

\begin{problem} Let $p\ge 2$. Is it true for any offspring distribution
$\bp$ with $\mu(\bp)=1$ and $\gss(\bp)<\infty$ then
\begin{align}\label{bgp}
  \Ez \bigabs{n\qqw L_n(xn\qq)-n\qqw L_n(yn\qq)}^p
\le C_p |x-y|^{p/2}\;?
\end{align}
  Is this true assuming a $p$th moment for $\bp$? 
Assuming an exponential moment?
\end{problem}

\begin{remark}
  A rerooting argument as in the proof of \refT{TM1} shows that
\eqref{bgp} would imply
\begin{align}\label{bgpl}
  \Ez \bigabs{n^{-3/2} \gL_n(xn\qq)-n^{-3/2}\gL_n(yn\qq)}^p
\le C_p |x-y|^{p/2}.
\end{align}
Again we can ask under which conditions this holds.
\end{remark}

We can also use generating functions and singularity analysis to estimate
the Fourier transform of $L_{n}$ and $\gL_n$. Recall that we in \refSS{SSprof} 
have defined $L_n$ and $\gL_n$ as functions on $\bbR$, using
linear interpolation. We first consider their restrictions to the integers, 
and the corresponding Fourier transforms,
for $- \infty < \xi < \infty$,
\begin{align} \label{dfou}
  \hzLn (\xi) 
:= \sum_{k=0}^{\infty} L_{n}(k)e^{\ii\xi k}
  \hspace*{4mm} \text{and} \hspace*{4mm}  {\hzgLn}(\xi) 
:= \sum_{k=0}^{\infty} \gL_n(k)e^{\ii \xi k}.
\end{align}
Note that these are periodic functions with period $2\pi$, 
and that, as a consequence of \eqref{pro1} and \eqref{dpro1}, 
\begin{align}\label{hL0}
  |\hzLn(\xi)|\le \hzLn(0)=n \hspace*{4mm} \text{and} \hspace*{4mm} 
  |\hzgLn(\xi)|\le \hzgLn(0)=n^2.
\end{align}

\begin{lemma}  \label{LFN} 
Let  $\bp$ be an offspring distribution with $\mu(\bp)=1$.    
  \begin{romenumerate}
  \item 
If\/ $\bp$ has a finite second moment, then
\begin{align} \label{bg6}
 \Ez | \hzLn(\xi n^{-1/2}) |^{2}  
&\le C n^{2}(\xi^{-2} \wedge 1), \qquad |\xi|\le \pi n^{1/2}.
\end{align}
\item 
If\/ $\bp$ has a finite  fourth moment, then
\begin{align} 
\label{bg7}
\Ez| \hzgLn( \xi n^{-1/2}) |^{2} 
&\le Cn^{4} (\xi^{-4} \wedge 1), \qquad |\xi|\le \pi n^{1/2}.
\end{align}
  \end{romenumerate}
\end{lemma}

The proofs of  \eqref{bg6} and \eqref{bg7} are based on lengthy
calculations of some generating functions, and are given in 
\refApp{GeneF}. 

We return to the interpolated profiles $L_n$ and $\gL_n$ and their Fourier
transforms $\hLn$ and $\hgLn$ defined by \eqref{ft}. It follows from
\eqref{Ltau} that
\begin{align}\label{bp1}
  \hLn(\xi)=\hzLn(\xi)\htau(\xi)=
\hzLn(\xi) \frac{\sin^2(\xi/2)}{(\xi/2)^2}
\end{align}
and similarly
\begin{align}\label{bp2}
  \hgLn(\xi)=\hzgLn(\xi)\htau(\xi)=
\hzgLn(\xi) \frac{\sin^2(\xi/2)}{(\xi/2)^2}.
\end{align}
In particular, $|\hLn(\xi)|\le|\hzLn(\xi)|$ and 
$|\hgLn(\xi)|\le|\hzgLn(\xi)|$; hence the estimates in \refL{LFN} holds also
for $\hLn$ and $\hgLn$.

We have stated \refTs{T0} and \ref{T1} with convergence (in distribution)
in $C\oooq$;
however, it is obvious that then the same statement holds with convergence
in $C[-\infty,\infty]$. 
Note that both sides of \eqref{t0} and \eqref{t1} are non-negative functions
on $\oooo$ with integral 1. It is easily seen that in the set of such
functions, convergence \aex{} (and \emph{a fortiori} convergence 
in $C[-\infty,\infty]$, \ie, uniform convergence) implies convergence in
$L^1(\bbR)$, and thus (uniform) convergence of the Fourier transforms.
Furthermore, the Fourier transforms of the \lhs{s} of \eqref{t0} and
\eqref{t1} are $n\qw\hLn(\xi n\qqw)$ and $n\qww\hgLn(\xi n\qqw)$, respectively.
Hence, \refTs{T0} and \ref{T1} imply that for any fixed $\xi$
\begin{align}\label{bp3}
  n\qw\hLn(\xi n\qqw)&\dto\hLX\bigpar{\tfrac{2}{\gs}\xi},
\\\label{bp4}
n\qww\hgLn(\xi n\qqw)&\dto \hGLX\bigpar{\tfrac{2}{\gs}\xi}.
\end{align}
Consequently, by Fatou's lemma, 
the estimates \eqref{bg6} and \eqref{bg7} in Lemma \ref{LFN} 
yield another proof of the estimates \eqref{fbl} and \eqref{fb} above.

Moreover, we can argue similarly to the proof of \refT{TF} and obtain a
corresponding 
result on the smoothness of $\gL_n$. Of course, this profile is
by construction Lipschitz, but what is relevant is 
a smoothness estimate that is uniform in $n$.
\begin{theorem}\label{THn}
%  Let  $\bp$ be an offspring distribution with $\mu(\bp)=1$.    
If the offspring distribution $\bp$ has a finite  fourth moment, then
  \begin{align}\label{bp7}
\Ez\bignorm{n^{-3/2}\gL_n\bigpar{x n\qq}}_{\Ha}^2
\le C
\end{align}
uniformly in $n$,
for every fixed $\ga<1$.
\end{theorem}

For the proof, we need a discrete version of (a simple case of) the Sobolev
embedding theorem. 
We do not know an explicit reference, so
for completeness, we state this and give a proof.

\begin{lemma}
  \label{Lsobolev}
Let $f:\bbR\to\bbR$ be a function initially defined on $\bbZ$ and extended
by linear interpolation to $\bbR$. Suppose that
$\sum_{-\infty}^\infty|f(k)|<\infty$, and define, as in \eqref{dfou}, 
$
\hzf(\xi):=\sum_k f(k) e^{\ii\xi k}.
$
Let $0<\ga<1$. Then
\begin{align}\label{sob1}
  \normHa{f}\le C_\ga \Bigpar{\intpipi |\xi|^{2\ga+1} |\hzf(\xi)|^2\dd\xi}\qq
.\end{align}
\end{lemma}

\begin{proof}
  Let $j,k\in\bbZ$. Then, by Fourier inversion,
  \begin{align}\label{sob2}
|f(j)-f(k)|&
=
\lrabs{\frac{1}{2\pi}\int_{-\pi}^{\pi}\bigpar{e^{-\ii j\xi}-e^{-\ii k\xi}}
\hzf(\xi)\dd\xi  }  
\notag\\&
\le \int_{0}^{\pi}\bigpar{|j-k|\xi \bmin1}\bigabs{\hzf(\xi)}\dd\xi  
. \end{align}
Hence, using the \CSineq,
\begin{align}\label{sob3}
  |f(j)-f(k)|^2
\le
\intpi \bigpar{|j-k|^2\xi^2 \bmin1} \xi^{-2\ga-1} \dd\xi
\intpi \bigabs{\hzf(\xi)}^2 \xi^{2\ga+1}\dd\xi  
.\end{align}
Writing $M:=\intpi \bigabs{\hzf(\xi)}^2 \xi^{2\ga+1}\dd\xi$,  
and changing variables to $t:=|j-k|\xi$ in the first integral
(assuming $j\neq k$),
we  obtain
\begin{align}\label{sob33}
  |f(j)-f(k)|^2
\le
|j-k|^{2\ga}\intoo \bigpar{t^2 \bmin1} t^{-2\ga-1} \dd t
\cdot M
= C M |j-k|^{2\ga}
\end{align}
for some $C<\infty$ (depending on $\ga$).
Hence,
\begin{align}\label{sob4}
  |f(x)-f(y)| \le C M\qq |x-y|^\ga
\end{align}
whenever $x$ and $y$ are integers. Since $f$ is linear between integers, it
follows that \eqref{sob4} holds for all $x,y\in\bbR$ (with another $C$),
\ie, $\normHa{f}\le C M\qq$, as claimed in \eqref{sob1}.
\end{proof}

\begin{proof}[Proof of \refT{THn}]
  \refL{Lsobolev} and a change of variables together with \eqref{bg7} yield
  \begin{align}
\Ez\normHa{\gL_n}^2&
\le C\E \intpipi |\xi|^{2\ga+1} \bigabs{\hzgLn(\xi)}^2\dd\xi
\notag\\&
= C n^{-\ga-1}\int_{-\pi n\qq}^{\pi n\qq} |\xi|^{2\ga+1} \E\bigabs{\hzgLn(\xi n\qqw)}^2
\dd\xi
\notag\\&
\le C n^{3-\ga}\intoooo \bigpar{|\xi|^{2\ga-3}\bmin |\xi|^{2\ga+1}}\dd\xi
= C n^{3-\ga},
  \end{align}
since the final integral converges for $0<\ga<1$.
A change of variables yields \eqref{bp7}.
\end{proof}

We end this subsection with a couple of open problems suggested by \refT{THn}.

The fourth moment condition in \refT{THn} is used in the proof,
but it seems likely that it can be weakened, see \refR{R4}. Hence, we ask
the following.
\begin{problem}\label{P2}
  Does \eqref{bp7} hold assuming only a finite second moment for $\bp$?
\end{problem}

Just as for the limit $\GLX$, we do not know whether \eqref{bp7} holds for
$\ga=1$. Since 
\HCx1 equals Lipschitz, and
$\gL_n$ is defined by linear interpolation, 
it is
easy to see that 
  \begin{align}\label{lip}
\bignorm{n^{-3/2}\gL_n\bigpar{x n\qq}}_{\Hax1}
=n\qw  \max_k|\gL_n(k+1)-\gL_n(k)|.
\end{align}
Consequently, the question whether \eqref{bp7} holds for $\ga=1$ is
equivalent to the following. (Cf.\ \eqref{egln}.)

\begin{problem}\label{PnLip}
Assume that $\mu(\bp)=1$ and that 
$\bp$ has a finite second moment (or fourth moment, or an exponential  moment).
Is then 
  \begin{align}
\E \max_k|\gL_n(k+1)-\gL_n(k)|^2 \le Cn^2\;?
  \end{align}
A slightly weaker version is: is
  \begin{align}
\E \max_k|\gL_n(k+1)-\gL_n(k)| \le Cn\;?
  \end{align}
\end{problem}

\section{Further remarks}\label{Sfurther}

\begin{remark}
  In analogy with \refSs{Sroot} and \ref{Sunroot}, we might, as a third
alternative, also define random \emph{unlabelled simply generated trees} 
by using the weights
\eqref{wT} on the set
$\fU_n$ of all unlabelled unrooted trees of order $n$.
(These are defined as equivalence classes of labelled trees in $\fL_n$ under
isomorphisms.)
It is usually more challenging to consider unlabelled trees;
we conjecture that results similar to those above hold for 
 unlabelled \sgt{s} too,
but we leave this as an open problem.

There is some existing work \cite{Wang, RStu} 
on yet another related  model for random trees:
\emph{simply generated unrooted plane trees} 
(i.e., unlabelled unrooted trees where each vertex is endowed with
a cyclic ordering of its neighbourhood). In particular, the results in
\cite[Sections 2 and 3]{RStu} seem to be reminiscent of some of our results
in Sections 3--5. Informally, through a series of approximations (using
rooted/labelled/differently weighted versions) one ends up comparing those
tree models to a pair of Galton--Watson trees (with the roots joined by an
edge), 
conditioned on  
the sum of vertices. 
One of the two
trees being typically macroscopic while the other is  microscopic. 
\end{remark}

\begin{remark}\label{Rinc}
Another class of ``simply generated'' trees is 
simply generated \emph{increasing} trees,
see 
\cite{BergeronFS92} and \cite[Section 1.3.3]{Drmota}.
These random trees are quite different: 
They have
(under weak conditions)  logarithmic height; moreover, 
both the profile and distance profile are degenerate, in the sense
that the distribution of depths or distances divided by $\log n$ converges
to a point mass at some positive constant.
With a more refined scaling, the profile and distance profile are both Gaussian,
see \citet{PP2004b}.
\end{remark}

\begin{remark}
  The degenerate behaviour seen in \refR{Rinc},
with almost all distances close to $c\log n$ for some constant $c>0$,  
seems to be typical for many  classes of random trees
  with logarithmic height;
see \cite{SJ352}.
\end{remark}

\begin{problem}
  We have in this paper assumed that the offspring distribution is critical
  and has finite variance. 
It would be interesting to extend the results to the case of infinite
variance, with the offspring distribution in the domain of attraction of a
stable variable. In this case  the tree converges in a suitable
sense to a random stable tree
\cite{Duquesnestable, Kortchemskistable};
moreover, there is a corresponding limit result
for the profile \cite{Kersting,AngtuncioUB}.
However,
there seems to be several technical challenges to adapt the arguments above
to this case, and we leave it as an open problem.
\end{problem}

%\newpage

\appendix

\section{Tightness}\label{Atight}

We show here a general technical result used in the proof of 
\refT{Teo}, and thus of
our main theorem.

We recall the standard criterion for tightness in $\coi$ in \eg{}
\citet[Theorem 8.2]{Billingsley}, which we formulate as follows,
using the modulus of continuity defined in \eqref{go}.
(The equivalence of our formulation and the one in \cite{Billingsley} is a
simple exercise.)
\begin{lemma}\label{LU1}
 A sequence of random functions $(X_n(t))\xoo$ in $\coi$
is tight if and only if the following two conditions hold:
\begin{romenumerate}
\item\label{LU1a} 
The sequence $X_n(0)$ of random variables is tight.
\item\label{LU1b}
 $\go(\gd;X_n)\pto0$ as $\gd\to0$, uniformly in $n$.
\end{romenumerate}
%\nopf
\end{lemma}

The uniform convergence in probability in \ref{LU1b} means that for every
$\eps,\eta>0$, there exists $\gd>0$ such that
\begin{align}\label{a3}
  \sup_n \P\bigpar{\go(\gd; X_n)\ge\eps}\le\eta.
\end{align}
Equivalently, for every sequence $\gd_k\to0$ and every sequence $(n_k)\xoo$,
\begin{align}\label{a4}
  \go(\gd_k;X_{n_k})\pto0\qquad\text{as \ktoo}
.\end{align}

We use this to prove the following general lemma on tightness of averages.

\begin{lemma}\label{LU2}
  Suppose that $(X_n(t))\xoo$ is a sequence of random functions in $\coi$,
that $(N_n)\xoo$ is an arbitrary sequence of positive integers, and that
for each $n$, $(X_{ni}(t))_{i=1}^{N_n}$ is a, possibly dependent, family of
identically distributed random functions with
$X_{ni}\eqd X_n$.
Assume that
\begin{romenumerate}
\item \label{LU2a}
The sequence $(X_n)$ is tight in $\coi$.
\item \label{LU2b}
The sequence of random variables
  \begin{align}\label{lu2b}
    \norm{X_n}:=
\sup_{t\in\oi}|X_n(t)|
  \end{align}
is uniformly integrable.
\end{romenumerate}
Then the sequence of averages
\begin{align}\label{lu2}
  Y_n(t):=\frac{1}{N_n}\sum_{i=1}^{N_n} X_{ni}(t)
\end{align}
is tight in $\coi$.
\end{lemma}

\begin{remark}\label{RU2}
  We state \refL{LU2} for $\coi$; it transfers immediately to \eg{} $\cooq$.
It follows also that the lemma holds for $\coo$, with \ref{LU2b}
replaced by the assumption that $\sup_{[0,b]}|X_n|$ over any finite interval 
$[0,b]$ is
uniformly integrable.
\end{remark}

\begin{proof}
Let $\gd_k\to0$ and let $n_k$ be arbitrary positive integers.
Then \eqref{a4} holds by assumption \ref{LU2a} and \refL{LU1}.

Furthermore,   for any $\gd>0$, we obviously have, by \eqref{go},
$\go(\gd;X_n)\le2\norm{X_n}$. In particular,
$\go(\gd_k;X_{n_k})\le2\norm{X_{n_k}}$, and it follows from \ref{LU2b} that
the sequence $\go(\gd_k;X_{n_k})$, $k\ge1$, is uniformly integrable.
Hence, \eqref{a4} implies 
\begin{align}\label{a5}
\E \go(\gd_k;X_{n_k})\to0,\qquad\text{as \ktoo}
.\end{align}

Next, it follows from the definitions \eqref{lu2} and \eqref{go} that, for
any $\gd>0$,
\begin{align}\label{a6}
  \go\xpar{\gd;Y_n}
=\frac{1}{N_n}\go\Bigpar{\gd;\sum_{i=1}^{N_n}X_{ni}} 
\le
\frac{1}{N_n}\sum_{i=1}^{N_n}\go\bigpar{\gd;X_{ni}} 
\end{align}
and hence,
\begin{align}\label{a7}
\E \go\xpar{\gd;Y_n}
\le
\frac{1}{N_n}\sum_{i=1}^{N_n}\E\go\bigpar{\gd;X_{ni}} 
=\E\go\bigpar{\gd;X_{n}}. 
\end{align}

 Consequently, by \eqref{a5} and \eqref{a7},
\begin{align}\label{a8}
\E \go(\gd_k;Y_{n_k})
\le
\E \go(\gd_k;X_{n_k})\to0,\qquad\text{as \ktoo}
.\end{align}
Since the sequences $(\gd_k)_k$ and $(n_k)_k$ are arbitrary,
this shows that the sequence $(Y_n)_n$ satisfies \eqref{a4} and thus
condition \refL{LU1}\ref{LU1b}.

Furthermore, condition \refL{LU1}\ref{LU1a} holds too, because
by \eqref{lu2}, \eqref{lu2b} and assumption \ref{LU2b},
\begin{align}
  \E|Y_n(0)|
\le 
\frac{1}{N_n}\sum_{i=1}^{N_n}\E| X_{ni}(0)|
=\E|X_n(0)|
\le\Ez\norm{X_n}
=O(1),
\end{align}
which implies that the sequence $Y_n(0)$ is tight.
Hence, the conclusion follows by \refL{LU1}.
\end{proof}

\begin{remark}
  The assumption on uniform integrability in \refL{LU2}\ref{LU2b}
cannot be weakened
  to $\Ez\norm{X_n}=O(1)$.
For an example, let 
\begin{align}
  X_n(t):=\xi_n n(1\bmin nt), \qquad t \ge 0,
\end{align}
where
$\xi_n\sim\Be(1/n)$, \ie, $\P(\xi_n=1)=1/n$ and
otherwise $\xi_n=0$.
Then $\Ez\norm{X_n}=1$. On the other hand, take $N_n:=n$ and let 
$X_{ni}\eqd X_n$ be independent, $1\le i\le N_n$. Then, for every $n$,
with probability $e\qw+o(1)$ exactly one $X_{ni}$ is not identically $0$,
and in this case $Y_n=(1\bmin nt)$ and thus $\go(1/n;Y_n)=1$.
Hence, $\P\bigpar{\go(1/n;Y_n)=1} \ge e\qw+o(1)$ and thus \eqref{a4} and
\refL{LU1}\ref{LU1b} do not hold for $(Y_n)$; thus the sequence $Y_n$ is
not tight.
\end{remark}

\section{Generating functions}\label{Agen}
The relations in \refSs{Smark1} and \ref{Smark2} between random rooted and
unrooted \sgt{s} can also be expressed using generating functions.
We discuss this briefly here; the results are not used in the main body of
the paper, but they further illuminate the relation and the constructions
in \refSs{Smark1}--\ref{Smark2}.

Let $(w_k)\zoo$ be given sequence of weights and define $\phi_k$ by
\eqref{phiw}.
Recall that by \eqref{Phi} and \eqref{phiw},
the ordinary generating function of $(\phi_k)\zoo$ is
\begin{align}\label{Phiw}
\Phi(z)=\sumko \phi_k z^k
  =\sumko \frac{w_{k+1}}{k!}z^{k}.
\end{align}
Furthermore, let $\Psi(z)$ be 
the exponential generating function of $(w_k)\xoo$, \ie,
\begin{align}\label{Psi}
  \Psi(z):=\sumk \frac{w_k}{k!}z^k.
\end{align}
Thus,  \eqref{Phiw} implies
\begin{align}\label{Psi'}
  \Psi'(z)
=\sumk \frac{w_k}{(k-1)!}z^{k-1}
=\sumko \frac{w_{k+1}}{k!}z^{k}
=\Phi(z).
\end{align}

We consider generating functions for the rooted and unrooted \sgt{s}.
We begin with the rooted case, which is standard.
Let $\za_n$ be the total weight \eqref{wTr} of all ordered rooted
trees in $\fT_n$, 
\ie,
\begin{align}
  a_n:=\sum_{T\in\fT_n}\phi(T),
\end{align}
and let $\ZA(z)$ be the corresponding ordinary generating
function, 
\begin{align}\label{ZA}
  \ZA(z):=\sumn \za_n z^n
=\sum_{T\in\fT}\phi(T)z^{|T|}
.\end{align}
Then, as is well-known and easy to see
\cite[Section II.5.1 p.~126]{FlajoletS},
\cite[Section 3.1.4]{Drmota},
\begin{align}\label{AAA}
  \ZA(z)=z\Phi\bigpar{\ZA(z)}.
\end{align}

Similarly, for the unrooted case, let $\zb_n$ be the total weight \eqref{wT} of all 
unrooted labelled  trees in
$\fL_n$, \ie,
\begin{align}
  \zb_n:=\sum_{T\in\fL_n}w(T),
\end{align}
and let $\ZB(z)$ be the corresponding exponential generating function,
\begin{align}\label{ZB}
  \ZB(z):=\sumn \frac{\zb_n}{n!}z^n
=\sum_{T\in\fL} w(T)\frac{z^{|T|}}{|T|!}
.\end{align}

We now return to the constructions in \refSs{Smark1} and \ref{Smark2}
in order to find relations between $A(z)$ and $\ZB(z)$.

\subsection{Marking a vertex}
In the construction in \refS{Smark1}, we first mark a vertex.
We have $n$ choices for each tree in $\fT_n$, and thus the sum
of the weight \eqref{wT} over all marked
(\ie, rooted) labelled trees of order $n$ is $n\zb_n$.

We then order the tree; this introduces more choices, but these are
compensated for by changing the weight to \eqref{wtb}. 
Hence, the sum of the weight \eqref{wtb} over all labelled ordered trees of
order $n$ is also $n\zb_n$.

Finally, we forget the labelling; since each tree has $n!$ labellings, the
sum of the weight \eqref{wtb} over all ordered trees in $\fT_n$ is thus
$n\zb_n/n!$. Hence, using \eqref{ZB},
\begin{align}\label{ZB1}
  \sum_{T\in \fT} w\x(T) z^{|T|} = \sumn \frac{n\zb_n}{n!} z^n
=z \ZB'(z).
\end{align}
Furthermore, by standard arguments and the definitions
\eqref{phiw}--\eqref{phio},
\begin{align}\label{kula}
  \sum_{T\in \fT} w\x(T) z^{|T|} &
= \sumko \sum_{T_1,\dots,T_k\in\fT}
z^{1+|T_1|+\dots+|T_k|}\phio_k \prod_{j=1}^k\prod_{v\in T_j} \phi_{\dout(v)}
\notag\\&
=
z\Psi\bigpar{\ZA(z)}.
\end{align}
Thus,
\begin{align}\label{BA}
  \ZB'(z)=\Psi\bigpar{\ZA(z)}.
\end{align}

\subsection{Marking an edge}
In the construction in \refS{Smark2}, we first mark a directed edge.
There are $2(n-1)$ choices, and thus the sum of the weight \eqref{wT} over
all marked labelled trees of order $n$ is $2(n-1)\zb_n$.
We then order the trees, and as in the preceding subsection compensate for
the number of orderings by changing the weight. Thus, the sum of the weight
\eqref{ww} over all pairs of ordered rooted trees $(T_+,T_-)$ with
$|T_+|+|T_-|=n$, labelled together with $1,\dots,n$, is also $2(n-1)\zb_n$.
Hence, the sum of the weight \eqref{ww} over unlabelled such pairs
$(T_+,T_i)$ is $2(n-1)\zb_n/n!$. Consequently, 
using \eqref{hw2} and \eqref{ZA},
\begin{align}
  \sumn \frac{2(n-1)\zb_n}{n!}z^n
  =\sum_{T_+,T_-\in\fT}\hw(T_+)\hw(T_-)z^{|T_+|+|T_-|}
  =\ZA(z)^2.
\end{align}
Hence,
\begin{align}\label{BAA}
  2z\ZB'(z)-2\ZB(z)=\ZA(z)^2.
\end{align}

\subsection{A sanity check}
Differentiating  \eqref{BA} yields, using \eqref{Psi'} and \eqref{AAA},
\begin{align}\label{B''AA}
  \ZB''(z)=\Psi'\bigpar{A(z)}A'(z)
=\Phi\bigpar{A(z)}A'(z)
=z\qw A(z)A'(z),
\end{align}
while differentiating \eqref{BAA} yields
\begin{align}\label{BB''A}
  2z \ZB''(z)=2(z\ZB'(z))'-2\ZB'(z) = 2A(z)A'(z).
\end{align}
Hence, \eqref{BA} and \eqref{BAA} are equivalent.

\section{Proof of \eqref{bg6} and \eqref{bg7}} \label{GeneF}

Let $\GGW$ be a (unconditioned) {\GWt} with offspring distribution $\bp$ such that $\mu(\bp)=1$. Define
\begin{align}\label{CdPQ}
 P(\GGW, x) := \sum_{v \in \GGW} x^{d(v,o)}, 
\hspace*{4mm} 
Q(\GGW, x) := \sum_{v,w \in \GGW} x^{d(v,w)},
\end{align}
and the generating functions
\begin{align} 
  \Phi(z) & :=\sumko p_kz^{k}, \label{CdPhi}\\
  A(z) & := \E[z^{|\GGW|}],\label{CdA}\\
  B(z,x) & := \E[z^{|\GGW|} P(\GGW, x) ],\label{CdB} \\
  C(z, x, y) & := \E[z^{|\GGW|} P(\GGW, x) P(\GGW, y)],\label{CdC} \\
  D(z, y) & := \E[z^{|\GGW|} Q(\GGW, x) ], \label{CdD}\\
  E(z, x, y) & := \E[z^{|\GGW|} Q(\GGW, x) P(\GGW, y)], \label{CdE}\\
  F(z, x, y) & := \E[z^{|\GGW|} Q(\GGW, x) Q(\GGW, y)]. \label{CdF}
\end{align}
These functions are defined and continuous at least for 
$|z|<1$ and $|x|,|y|\le1$, and analytic for
$|x|, |y|, |z| <1$. 

It has been shown in \cite[(2.1)]{SJ222} that
\begin{align} \label{C1}
B(z,x) & = \frac{A(z)}{1-xz  \Phi^{\prime}( A(z) )}. 
\end{align}
Furthermore, by taking $x=y$ in \cite[(2.2)]{SJ222},
\begin{align} \label{C2}
D(z,x) & = \frac{ x^{2}z \Phi^{\prime \prime}( A(z) )   B(z,x)^{2}+ 2xz \Phi^{\prime}( A(z) )  B(z,x)  + A(z)}{1- z\Phi^{\prime}( A(z) )}.
\end{align}
It only remains to compute the generating functions $C$, $E$ and $F$. We
follow the approach used in \cite{SJ222} to compute \eqref{C1} and
\eqref{C2}.

Let us condition on the degree $d(o)$ of the root of $\GGW$. If $d(o) =
\ell$, then $\GGW$ has $\ell$ subtrees $\GGW_{1}, \dots, \GGW_{\ell}$ at the
root $o$, and conditioned on $d(o) = \ell$, these are independent and with
the same distribution as $\GGW$. 
% we denote their roots (the neighbours of $o$), by $o_{1}, \dots, o_{\ell}$. 

Assume $d(o) = \ell$. First, $|\GGW| = 1 + \sum_{i=1}^{\ell} |\GGW_{i}|$ and
thus 
\begin{align}\label{C0}
z^{|\GGW|} = z \prod_{i=1}^{\ell} z^{|\GGW_{i}|}.   
\end{align}
Hence, as is well-known, see \eqref{AAA},
\begin{align}
  A(z) = z \Phi(A(z)).
\end{align}
Next,
separating the cases $v \in \GGW_{i}$, $i =1, \dots, \ell$ and $v = o$, we
see that 
\begin{align} \label{C3}
P(\GGW, x)  = x \sum_{i = 1}^{\ell} P(\GGW_{i}, x)  + 1.
\end{align}
Similarly, 
using here and below $\sumx$ to denote multiple sums where the indices
are distinct integers in \set{1,\dots,\ell},
%considering the cases $v \in \GGW_{i}$, $w \in \GGW_{j}$, $i,j =1, \dots, \ell$ and $v = o, w = o$,
\begin{align} \label{C4}
Q(\GGW, x)  & 
=  \sum_{i =1 }^{\ell} Q(\GGW_{i}, x) 
+ x^{2}\sumxij P(\GGW_{i}, x) P(\GGW_{j}, x)  
+ 2x \sum_{i = 1}^{\ell} P(\GGW_{i}, x)  + 1.
\end{align}

We start by computing $C$. From \eqref{C0} and \eqref{C3}, we see that 
\begin{align}\label{CPP}
z^{|\GGW|} P(\GGW, x) P(\GGW, y) & =   xyz  \sum_{i = 1}^{\ell}
                                   z^{|\GGW_{i}|}P(\GGW_{i}, x) P(\GGW_{i},
                                   y) \prod_{k\neq i}^{\ell} z^{|\GGW_{k}|}  
\notag\\
& \hspace*{5mm} +  xyz \sumxij z^{|\GGW_{i}|} P(\GGW_{i}, x)
           z^{|\GGW_{j}|} P(\GGW_{j}, y) \prod_{k\neq i,j}^{\ell}
           z^{|\GGW_{k}|} 
\notag\\
& \hspace*{5mm} + xz \sum_{i = 1}^{\ell} z^{|\GGW_{i}|} P(\GGW_{i}, x)
           \prod_{k\neq i}^{\ell} z^{|\GGW_{k}|}  
\notag\\
& \hspace*{5mm} +  yz \sum_{i = 1}^{\ell} z^{|\GGW_{i}|}  P(\GGW_{i}, y) \prod_{k\neq i}^{\ell} z^{|\GGW_{k}|}+ z  \prod_{k\neq i}^{\ell} z^{|\GGW_{k}|}.
\end{align}
Hence,
\begin{align}\label{CPPE}
\E\bigsqpar{ z^{|\GGW|} P(\GGW, x) P(\GGW, y)\mid d(o) = \ell}  
& = xyz \ell C(z, x, y) A(z)^{\ell -1}   \notag\\
& \hspace*{5mm}  + xyz \ell(\ell-1) B(z,x)B(z,y) A(z)^{\ell -2} \notag\\
& \hspace*{5mm} + xz \ell B(z,x) A(z)^{\ell -1}  \notag\\
& \hspace*{5mm}  + yz \ell B(z,y) A(z)^{\ell -1} + z A(z)^{\ell}
\end{align}
and thus
\begin{align}\label{C5a}
C(z,x,y) & = xyz \Phi^{\prime}( A(z) ) C(z,x, y) +   xyz \Phi^{\prime \prime}( A(z) )   B(z,x)B(z,y)  \notag\\
& \qquad + xz \Phi^{\prime}( A(z) )  B(z,x)  + yz \Phi^{\prime}( A(z) )  B(z,y) + A(z) \notag\\
& = xyz \Phi^{\prime}( A(z) ) C(z,x, y) +   xyz \Phi^{\prime \prime}( A(z) )   B(z,x)B(z,y)  \notag\\
& \qquad +  B(z,x) +B(z,y)-A(z)
%+ yz \Phi^{\prime}( A(z) )  B(z,x) 
\end{align}
which gives
\begin{align} \label{C5}
C(z,x, y) & =  \frac{xyz \Phi^{\prime \prime}( A(z) ) B(z,x)B(z,y) + B(z,x)  +B(z,y)-A(z) }{1-xyz \Phi^{\prime}( A(z) )}.
\end{align}

Next, we compute $E$. From \eqref{C3} and  \eqref{C4}, we see that
\begin{align}\label{CQP}
z^{|\GGW|} Q(\GGW, x) P(\GGW, y) & =  y z  \sum_{i = 1}^{\ell} z^{|\GGW_{i}|}Q(\GGW_{i}, x) P(\GGW_{i}, y) \prod_{k\neq i}^{\ell} z^{|\GGW_{k}|}  \notag\\
& \hspace*{5mm} +  y z \sumxij z^{|\GGW_{i}|} Q(\GGW_{i}, x) z^{|\GGW_{j}|} P(\GGW_{j}, y) \prod_{k\neq i,j}^{\ell} z^{|\GGW_{k}|} \notag\\
& \hspace*{5mm} +  2x^{2} y z \sumxij z^{|\GGW_{i}|} P(\GGW_{i}, x) P(\GGW_{i}, y) z^{|\GGW_{j}|} P(\GGW_{j}, x) \prod_{k\neq i,j}^{\ell} z^{|\GGW_{k}|} \notag\\
& \hspace*{5mm} +  x^{2}  yz \sumxijl z^{|\GGW_{i}|} P(\GGW_{i}, x)  z^{|\GGW_{j}|} P(\GGW_{j}, x) z^{|\GGW_{l}|}  P(\GGW_{l}, y) \prod_{k\neq i,j,l}^{\ell} z^{|\GGW_{k}|} \notag\\
& \hspace*{5mm} + 2x yz \sum_{i = 1}^{\ell} z^{|\GGW_{i}|} P(\GGW_{i}, x) P(\GGW_{i}, y) \prod_{k\neq i}^{\ell} z^{|\GGW_{k}|}  \notag\\
& \hspace*{5mm} + 2x yz \sumxij  z^{|\GGW_{i}|} P(\GGW_{i}, x) z^{|\GGW_{j}|} P(\GGW_{j}, y) \prod_{k\neq i,j}^{\ell} z^{|\GGW_{k}|}  \notag\\
& \hspace*{5mm} + y z \sum_{i =1}^{\ell} z^{|\GGW_{i}|} P(\GGW_{i}, y)  \prod_{k\neq i}^{\ell} z^{|\GGW_{k}|}  + z^{|\GGW|} Q(\GGW, x). 
\end{align}
Hence,
\begin{align}\label{CQPE}
\E\bigsqpar{ z^{|\GGW|} Q(\GGW, x) P(\GGW, y)\mid d(o) = \ell}  
\hskip-8em&\hskip8em 
= y z \ell E(z,x, y) A(z)^{\ell -1}   \notag\\
& \hspace*{5mm} + y z \ell(\ell-1) D(z,x) B(z, y) A(z)^{\ell -2} \notag\\
& \hspace*{5mm} + 2x^{2}y z \ell(\ell-1) C(z,x,y) B(z,x) A(z)^{\ell -2} \notag\\
& \hspace*{5mm} + x^{2}yz \ell(\ell-1)(\ell-2) B(z,x)^{2}B(z,y) A(z)^{\ell -3} \notag\\
& \hspace*{5mm} + 2x yz \ell C(z,x, y) A(z)^{\ell -1}  + \notag\\
& \hspace*{5mm} + 2 xyz \ell (\ell -1) B(z,x)B(z,y) A(z)^{\ell -2} \notag\\
& \hspace*{5mm} + yz \ell 
B(z,y) A(z)^{\ell -1} 
+ \E\bigsqpar{ z^{|\GGW|} Q(\GGW, x) | d(0) = \ell}
\end{align}
and thus
\begin{align} \label{C6}
E(z,x, y)  & = y z \Phi^{\prime}( A(z) )E(z,x, y) +   yz \Phi^{\prime \prime}( A(z) )D(z,x) B(z, y) \notag \\
& \hspace*{5mm} + 2x^{2} yz \Phi^{\prime \prime}( A(z) ) C(z,x,y) B(z,x)  \notag \\
& \hspace*{5mm} + x^{2}yz \Phi^{\prime \prime \prime}( A(z) ) B(z,x)^{2}B(z,y) \notag \\
& \hspace*{5mm} + 2x yz \Phi^{\prime}( A(z) ) C(z,x, y)   +  2 xyz \Phi^{\prime \prime}( A(z) ) B(z,x)B(z,y) \notag \\
& \hspace*{5mm} +yz \Phi^{\prime}( A(z) ) B(z,y) + D(z,x)
.\end{align}
Consequently, in analogy with \eqref{C5},
\begin{align}\label{C6b}
  E(z,x, y)  & = \frac{\dots}{1-y z \Phi^{\prime}( A(z) )},
\end{align}
where the numerator consists of all terms except the first
on the \rhs{} of \eqref{C6}.

Finally, we compute $F$. From \eqref{C4}, we see that 
\begin{align}\label{CQQ}
z^{|\GGW|} Q(\GGW, x) Q(\GGW, y) 
\hskip-6em&\hskip6em =  
z  \sum_{i = 1}^{\ell} z^{|\GGW_{i}|}Q(\GGW_{i}, x) Q(\GGW_{i}, y) \prod_{k\neq i}^{\ell} z^{|\GGW_{k}|}  \notag\\
&  \hspace*{5mm} + z  \sumxij  z^{|\GGW_{i}|}Q(\GGW_{i}, x) z^{|\GGW_{j}|} Q(\GGW_{j}, y) \prod_{k\neq i,j}^{\ell} z^{|\GGW_{k}|}  \notag\\
& \hspace*{5mm} +  2y^{2} z \sumxij z^{|\GGW_{i}|} Q(\GGW_{i}, x) P(\GGW_{i}, y) z^{|\GGW_{j}|} P(\GGW_{j}, y) \prod_{k\neq i,j}^{\ell} z^{|\GGW_{k}|} \notag\\
& \hspace*{5mm} +  y^{2}z \sumxijl z^{|\GGW_{i}|} Q(\GGW_{i}, x)  z^{|\GGW_{j}|} P(\GGW_{j}, y) z^{|\GGW_{l}|}  P(\GGW_{l}, y) \prod_{k\neq i,j,l}^{\ell} z^{|\GGW_{k}|} \notag\\
& \hspace*{5mm} + 2y z \sum_{i = 1}^{\ell} z^{|\GGW_{i}|} Q(\GGW_{i}, x) P(\GGW_{i}, y) \prod_{k\neq i}^{\ell} z^{|\GGW_{k}|}  \notag\\
& \hspace*{5mm} + 2 y z \sumxij  z^{|\GGW_{i}|} Q(\GGW_{i}, x) z^{|\GGW_{j}|} P(\GGW_{j}, y) \prod_{k\neq i,j}^{\ell} z^{|\GGW_{k}|}  \notag\\
%%&\RED \hspace*{5mm} + z \sum_{i =1}^{\ell} z^{|\GGW_{i}|} Q(\GGW_{i}, x)  \prod_{k\neq i}^{\ell} z^{|\GGW_{k}|} \notag\\
& \hspace*{5mm} +  2x^{2} z \sumxij z^{|\GGW_{i}|} Q(\GGW_{i}, y) P(\GGW_{i}, x)  z^{|\GGW_{j}|} P(\GGW_{j}, x) \prod_{k\neq i,j}^{\ell} z^{|\GGW_{k}|} \notag\\
& \hspace*{5mm} + x^{2}z \sumxijl z^{|\GGW_{i}|} P(\GGW_{i}, x)  z^{|\GGW_{j}|} P(\GGW_{j}, x) z^{|\GGW_{l}|}  Q(\GGW_{l}, y) \prod_{k\neq i,j,l}^{\ell} z^{|\GGW_{k}|} \notag\\
& \hspace*{5mm} +  4x^{2}y^{2}z \sumxijl z^{|\GGW_{i}|} P(\GGW_{i}, x) P(\GGW_{i}, y)  z^{|\GGW_{j}|} P(\GGW_{j}, x) z^{|\GGW_{l}|}  P(\GGW_{l}, y) \prod_{k\neq i,j,l}^{\ell} z^{|\GGW_{k}|} \notag\\
& \hspace*{5mm} +  2x^{2}y^{2}z \sumxij z^{|\GGW_{i}|} P(\GGW_{i}, x) P(\GGW_{i}, y)   z^{|\GGW_{j}|} P(\GGW_{j}, x)  P(\GGW_{j}, y) \prod_{k\neq i,j}^{\ell} z^{|\GGW_{k}|} \notag\\
& \hspace*{5mm} +  x^{2}y^{2}z\sumxijlr z^{|\GGW_{i}|} P(\GGW_{i}, x) z^{|\GGW_{j}|} P(\GGW_{j}, x)   z^{|\GGW_{l}|} P(\GGW_{l}, y) z^{|\GGW_{r}|}  P(\GGW_{r}, y) \prod_{k\neq i,j,l,r}^{\ell} z^{|\GGW_{k}|} \notag\\
& \hspace*{5mm} +  4x^{2}yz \sumxij z^{|\GGW_{i}|} P(\GGW_{i}, x) P(\GGW_{i}, y) z^{|\GGW_{j}|} P(\GGW_{j}, x)  \prod_{k\neq i,j}^{\ell} z^{|\GGW_{k}|} \notag\\
& \hspace*{5mm} + 2x^{2}yz   \sumxijl z^{|\GGW_{i}|}P(\GGW_{i}, x) z^{|\GGW_{j}|}P(\GGW_{j}, x) z^{|\GGW_{l}|}P(\GGW_{l}, y)  \prod_{k\neq i,j,l}^{\ell} z^{|\GGW_{k}|}  \notag\\
%%&\RED \hspace*{5mm} + x^{2}z   \sumxij  z^{|\GGW_{i}|}P(\GGW_{i}, x) z^{|\GGW_{j}|}P(\GGW_{j}, x)   \prod_{k\neq i,j}^{\ell} z^{|\GGW_{k}|}  \notag\\
& \hspace*{5mm} + 2xz   \sum_{i =1}^{\ell} z^{|\GGW_{i}|} Q(\GGW_{i}, y) P(\GGW_{i}, x)    \prod_{k\neq i}^{\ell} z^{|\GGW_{k}|}  \notag\\
& \hspace*{5mm} + 2xz   \sumxij  z^{|\GGW_{i}|}P(\GGW_{i}, x) z^{|\GGW_{j}|}Q(\GGW_{j}, y)   \prod_{k\neq i,j}^{\ell} z^{|\GGW_{k}|}  \notag\\
& \hspace*{5mm} +  4xy^{2} z \sumxij z^{|\GGW_{i}|} P(\GGW_{i}, x) P(\GGW_{i}, y) x^{|\GGW_{j}|} P(\GGW_{j}, y) \prod_{k\neq i,j}^{\ell} z^{|\GGW_{k}|} \notag\\
& \hspace*{5mm} + 2xy^{2}z \sumxijl z^{|\GGW_{i}|} P(\GGW_{i}, x)  z^{|\GGW_{j}|} P(\GGW_{j}, y) z^{|\GGW_{l}|}  P(\GGW_{l}, y) \prod_{k\neq i,j,l}^{\ell} z^{|\GGW_{k}|} \notag\\
& \hspace*{5mm} + 4xyz \sum_{i = 1}^{\ell} z^{|\GGW_{i}|} P(\GGW_{i}, x) P(\GGW_{i}, y) \prod_{k\neq i}^{\ell} z^{|\GGW_{k}|}  \notag\\
& \hspace*{5mm} + 4xyz \sumxij  z^{|\GGW_{i}|} P(\GGW_{i}, x) z^{|\GGW_{j}|} P(\GGW_{j}, y) \prod_{k\neq i,j}^{\ell} z^{|\GGW_{k}|}  \notag\\
%%&\RED \hspace*{5mm} + 2xz \sum_{i =1}^{\ell} z^{|\GGW_{i}|} P(\GGW_{i}, x)  \prod_{k\neq i}^{\ell} z^{|\GGW_{k}|} \notag\\
& \hspace*{5mm} 
+ z^{|\GGW|}Q(\mathcal{T},x)
+ z^{|\GGW|}Q(\mathcal{T},y)
-z^{|\GGW|}
.\end{align}
Hence,
\begin{align}\label{CQQE}
\E[ z^{|\GGW|} Q(\GGW, x) Q(\GGW, y)\mid d(o) = \ell]  
\hskip-6em&\hskip6em 
=  z \ell F(z,x, y) A(z)^{\ell -1}  \notag\\
& \hspace*{5mm} +  z \ell(\ell-1) D(z,x) D(z, y) A(z)^{\ell -2} \notag\\
& \hspace*{5mm} + 2y^{2} z \ell(\ell-1) E(z,x, y)B(z,y) A(z)^{\ell -2} \notag\\
& \hspace*{5mm} +  y^{2} z \ell(\ell-1)(\ell-2) D(z,x)B(z,y)^{2} A(z)^{\ell -3} \notag\\
& \hspace*{5mm} + 2 yz \ell E(z,x, y) A(z)^{\ell -1}  \notag\\
& \hspace*{5mm} +  2 yz \ell (\ell -1) D(z,x)B(z,y) A(z)^{\ell -2} \notag\\
%%&\RED \hspace*{5mm} + z\ell D(z,x) A(z)^{\ell -1} \notag\\
& \hspace*{5mm} + 2x^{2}z \ell(\ell-1) E(z,y, x) B(z,x) A(z)^{\ell -2} \notag\\
& \hspace*{5mm} + x^{2}z \ell(\ell-1)(\ell-2)  B(z,x)^{2} D(z,y) A(z)^{\ell -3} \notag\\
& \hspace*{5mm} + 4x^{2}y^{2}z \ell(\ell-1)(\ell-2)  C(z,x,y) B(z,x) B(z,y)  A(z)^{\ell -3} \notag\\
& \hspace*{5mm} + 2x^2 y^2z \ell(\ell-1) C(z,x, y)^{2}  A(z)^{\ell -2} \notag\\
& \hspace*{5mm} + x^2 y^2z \ell(\ell-1)(\ell-2)(\ell-3) B(z,x)^{2} B(z, y)^{2} A(z)^{\ell -4} \notag\\
& \hspace*{5mm} + 4 x^{2} yz \ell(\ell-1) C(z,x, y) B(z, x) A(z)^{\ell -2} \notag\\
& \hspace*{5mm} + 2 x^{2} yz \ell(\ell-1)(\ell -2) B(z,x)^{2} B(z, y) A(z)^{\ell -3} \notag\\
%%&\RED \hspace*{5mm} + x^{2} z \ell(\ell-1) B(z,x)^{2} A(z)^{\ell -2} \notag\\
& \hspace*{5mm} + 2x z \ell E(z,y, x) A(z)^{\ell -1} \notag\\
& \hspace*{5mm} + 2x z \ell(\ell-1) D(z,y) B(z,x) A(z)^{\ell -2} \notag\\
& \hspace*{5mm} + 4x y^2 z \ell(\ell-1) C(z,x, y) B(z,y) A(z)^{\ell -2} \notag\\
& \hspace*{5mm} + 2xy^{2} z \ell(\ell-1)(\ell-2) B(z,x) B(z, y)^{2} A(z)^{\ell -3} \notag\\
& \hspace*{5mm} + 4xyz \ell C(z,x, y) A(z)^{\ell -1} \notag\\
& \hspace*{5mm} + 4xyz\ell(\ell -1) B(z,x) B(z, y) A(z)^{\ell -2} \notag\\
%%&\RED \hspace*{5mm} + 2xz\ell B(z,x) A(z)^{\ell -1} \notag\\
& \hspace*{5mm} + \E\bigsqpar{ z^{|\GGW|}\bigpar{Q(T,x)+ Q(T, y)-1}
\mid d(o) = \ell}
\end{align}
and thus
\begin{align} \label{C7}
F(z,x, y)  & = z F(z,x, y) \Phi^{\prime}(A(z))  \notag \\
& \hspace*{5mm} +  z D(z,x) D(z, y) \Phi^{\prime \prime}(A(z)) \notag \\
& \hspace*{5mm} + 2y^{2} z E(z,x, y)B(z,y) \Phi^{\prime \prime}(A(z)) \notag \\
& \hspace*{5mm} +  y^{2} z D(z,x)B(z,y)^{2} \Phi^{\prime \prime \prime}(A(z)) \notag \\
& \hspace*{5mm} + 2 y z E(z,x, y) \Phi^{\prime}(A(z)) \notag  \\
& \hspace*{5mm} +  2 yz D(z,x)B(z,y) \Phi^{\prime \prime}(A(z)) \notag \\
%%&\RED \hspace*{5mm} + {\RED z} D(z,x) \Phi^{\prime}(A(z)) \notag \\
& \hspace*{5mm} + 2x^{2}z E(z,y, x) B(z,x) \Phi^{\prime \prime}(A(z)) \notag \\
& \hspace*{5mm} + x^{2}z  B(z,x)^{2} D(z,y) \Phi^{\prime \prime \prime}(A(z)) \notag \\
& \hspace*{5mm} + 4x^{2}y^{2}z C(z,x, y) B(z,x) B(z, y) \Phi^{\prime \prime \prime}(A(z)) \notag \\
& \hspace*{5mm} + 2x^2y^2z
C(z,x, y)^{2}  \Phi^{\prime \prime}(A(z)) \notag \\
& \hspace*{5mm} + x^2 y^2z
B(z,x)^{2} B(z, y)^{2} \Phi^{\prime \prime \prime \prime}(A(z)) \notag \\
& \hspace*{5mm} + 4 x^{2} y z C(z,x, y) B(z, x) \Phi^{\prime \prime}(A(z)) \notag \\
& \hspace*{5mm} + 2 x^{2} yz  B(z,x)^{2} B(z, y) \Phi^{\prime \prime \prime}(A(z))  \notag \\
%&\RED \hspace*{5mm} + x^{2} z B(z,x)^{2} \Phi^{\prime \prime}(A(z)) \notag \\
& \hspace*{5mm} + 2x z E(z,y,x) \Phi^{\prime}(A(z)) \notag \\
& \hspace*{5mm} + 2x z D(z,y) B(z,x) \Phi^{\prime \prime}(A(z)) \notag \\
& \hspace*{5mm} + 4xy^2z C(z,x, y) B(z,y) \Phi^{\prime \prime}(A(z)) \notag \\
& \hspace*{5mm} + 2xy^{2} z B(z,x) B(z, y)^{2} \Phi^{\prime \prime \prime}(A(z)) \notag \\
& \hspace*{5mm} + 4xy z  C(z,x, y) \Phi^{\prime}(A(z)) \notag \\
& \hspace*{5mm} + 4xyz B(z,x) B(x, y) \Phi^{\prime \prime}(A(z)) \notag \\
%%&\RED \hspace*{5mm} + 2x z B(z,x) \Phi^{\prime}(A(z)) \notag \\
&\hspace*{5mm} + D(z,x)+D(z, y)-A(z)
.\end{align}
Consequently, in analogy with \eqref{C5} and \eqref{C6b},
\begin{align}\label{C7b}
  F(z,x, y)  & = \frac{\dots}{1- z \Phi^{\prime}( A(z) )},
\end{align}
where the numerator consists of all terms except the first
on the \rhs{} of \eqref{C7}.

From now on, assume for simplicity that $\bp$ has also span $1$;
the general case follows by standard modifications.

We will use standard singularity analysis, see e.g. Flajolet and Sedgewick
\cite{FlajoletS}, and define the domain, for $0 < \beta < \pi/2$ and $\delta >0$,
\begin{align*}
\Delta(\beta, \delta) := \{z \in \mathbb{C}: |z|<1 + \delta, z \neq 1, |{\rm arg}(z-1)| > \pi/2-\beta \}.
\end{align*}

The next lemma can be found in \cite[Lemma 2.1]{SJ222}, together with \cite[Lemma A.2]{SJ167}.  %|A(z)|<1
In what follows,
$C_{1}, C_{2}, \dots$ and $c_{1}, c_{2}, \dots$ denote some positive
constants that may depend on the offspring distribution $\bp$ only. 
\begin{lemma} \label{LemmaC}
Suppose that $\bp$ has $\mu(\bp)=1$, $0<\gss(\bp)<\infty$ and span $1$. Then there exists $\beta, \delta > 0$ such that $\Phi$ extends to an analytic function in $\Delta(\beta, \delta)$ and, for some $c_{1}, c_{2} > 0$, if $x,z \in \Delta(\beta, \delta)$, then
\begin{align}
|A(z)|&<1,\label{lcA}\\
|1- z \Phi^{\prime}(A(z))| & \ge \ccname{\cclc2} |1-z|^{1/2}, \label{lc2}\\
|1- xz \Phi^{\prime}(A(z))| & \ge \cc |1-x|.\label{lc3}
\end{align}
Consequently, $B(z, x)$ and $D(z,x)$ extend to analytic functions of $x, z \in \Delta(\beta, \delta)$, and for all $x, z \in \Delta(\beta, \delta)$,
\begin{align}
|B(z,x)| & \le \CC |1-x|^{-1}, \label{lcB}\\
|D(z,x)| & \le \CC |1-z|^{-1/2}|1-x|^{-2}.\label{lcD}
\end{align}
\end{lemma}

We can now show similar estimates for $C(z,x,y)$, $E(z,x,y)$ and $F(z,x,y)$.
For simplicity we consider only the case $|x|=1$ and $y=\bar x$, where $\bar x$ denotes the complex conjugate of $x$; note that then
$xy=1$ and $|1-y|=|1-x|$. 

\begin{lemma} \label{LemmaCEF}
Let $\bp$ be as in \refL{LemmaC}, and assume also $\mu_4(\bp)<\infty$.
Let $\beta, \delta$ be as in \refL{LemmaC}.
If  $|x|=1$ with $x\neq1$, then
$C(z, x, \bx), E(z,x,\bx)$ and $F(z,x,\bx)$ extend to analytic functions of 
$z \in \Delta(\beta, \delta)$, 
such that, for all $z \in \Delta(\beta, \delta)$,
\begin{align} 
|C(z,x, \bar{x} )| & \le \CC |1-z|^{-1/2} |1-x|^{-2}, \label{lcC} \\
|E(z,x, \bar{x} )| & \le \CC |1-z|^{-1/2} |1-x|^{-4}, \label{lcE}\\
|F(z,x, \bar{x} )| & \le \CC |1-z|^{-3/2} |1-x|^{-4} +\CC |1-z|^{-1} |1-x|^{-5}
\label{lcF}
.\end{align}
\end{lemma}

\begin{proof}
We first use \eqref{C5}. The denominator simplifies to $1-z\Phi'(A(z))$, 
which is non-zero for $z\in \Delta(\beta, \delta)$
by \eqref{lc2}; thus \eqref{C5} yields the desired
analytic extension. 
Moreover, the assumption $\gss(\bp)<\infty$ together with \eqref{lcA} 
implies $|\Phi''(A(z))|\le C$, 
and thus
\eqref{lcA} and \eqref{lcB} show that the
numerator of \eqref{C5} is bounded by $\CC |1-x|^{-2}$, while
\eqref{lc2} shows that the absolute value of the 
denominator is at least $\cclc2|1-z|\qq$;
thus \eqref{lcC} follows.
%Note that the assumption $\gss(\bp)<\infty$ is used together with
%\eqref{lcA} to obtain $|\Phi''(A(z))|\le C$.

The proofs of \eqref{lcE} and \eqref{lcF} are similar, using
\eqref{C6b} and \eqref{C7b}, together with the already proved parts of the
lemma. 
We  have $|\Phi'''(A(z))|\le C$ and $|\Phi''''(A(z))|\le C$ by
\eqref{lcA} and
the assumption $\mu_4(\bp)<\infty$.
\end{proof}

We have now all the ingredients to prove \eqref{bg6} and \eqref{bg7}. Recall that 
$\hzLn$ and $\hzgLn$ denote the Fourier transform of the height profile
$L_{n}$ and distance profile $\gL_{n}$ of a \cGWt{} of order $n$; see
\eqref{dfou}. 

\begin{proof}[Proof of \eqref{bg6}]
Let $- \pi \le \xi \le\pi$ and take $x:=e^{\ii\xi}$, $y:=\bx=e^{-\ii\xi}$.
Let $\beta, \delta$ be as in \refL{LemmaC}, and
suppose that $z \in \Delta(\beta, \delta)$.
Note that by the definitions \eqref{CdC}, \eqref{CdPQ} and \eqref{dfou},
\begin{align}
  C(z,x, \bar{x}) = \sum_{n=1}^{\infty} \mathbb{P}(|\GGW| = n) \cn(x, \bar{x}) z^{n},
\end{align}
where $\cn(x, \bar{x}) := \Ez \bigabs{\hzLn(\xi)}^{2}$. 

The inequality (\ref{lcC}) and singularity analysis \cite[Theorem VI.3]{FlajoletS} imply that
\begin{align}\label{sing}
|\mathbb{P}(|\GGW| = n) \cn(x, \bar{x})| \le \CC n^{-1/2} |1-x|^{-2}.
\end{align}
It is well-known (see e.g., \cite{Otter}, \cite[Proposition 24]{AldousIII}, \cite[Theorem 3.1]{MM} or \eqref{b4}) that 
\begin{align} \label{C10}
\mathbb{P}(|\GGW| = n) \sim \cc n^{-3/2}.
\end{align}
Thus, \eqref{sing} yields
\begin{align} \label{C11}
\Ez\bigabs{\hzLn(\xi)}^2=
\cn(x, \bar{x})
\le \CC n |1-x|^{-2}
= \CCx n |1-e^{\ii\xi}|^{-2}
\le \CC n\xi\qww.
\end{align}
Replacing $\xi$ by $\xi n\qqw$, we obtain \eqref{bg6} for 
$1\le|\xi|\le\pi n\qq$. 
Finally, the case $|\xi|\le1$ is trivial, by \eqref{hL0}.
%
%\begin{align} \label{C14}
%|\widehat{L_{n}}(\xi n^{-1/2})| \le \sum_{k=0}^{\infty} L_{n}(n^{-1/2}k) \le n;
%\end{align}
%see e.g., \eqref{pro1}. 
\end{proof}

\begin{proof}[Proof of \eqref{bg7}]
Let $\xi,x,y$ and $z$ be as in the proof of \eqref{bg6}.
Note that by \eqref{CdF}, \eqref{CdPQ} and \eqref{dfou},
\begin{align}
  F(z,x, \bar{x}) 
= \sum_{n=1}^{\infty} \mathbb{P}(|\GGW| = n) f_{n}(x, \bar{x}) z^{n},
\end{align}
where $f_{n}(x, \bar{x}) := \E |\hzgLn(\xi)|^{2}$. 
The inequality (\ref{lcF}) and singularity analysis 
\cite[Theorem VI.3 and its proof]{FlajoletS} imply that
\begin{align}
|\mathbb{P}(|\GGW| = n) f_{n}(x, \bar{x})| \le 
\CC n^{1/2} |1-x|^{-4} + \CC |1-x|^{-5}.
\end{align}
Thus, (\ref{C10}) implies that
\begin{align}
\Ez\bigabs{\hzgLn(\xi)}^2&=
f_{n}(x, \bar{x}) \le \CC n^{2} |1-x|^{-4} + \CC n^{3/2}|1-x|^{-5}
\notag\\&
\le \CC n^2 \xi^{-4} + \CC n^{3/2} \xi^{-5}
.\end{align}
This implies \eqref{bg7} for $1\le|\xi|\le\pi n\qq$.

Finally, the case $|\xi|\le1$ is again trivial by \eqref{hL0}.
\end{proof}

\begin{remark}
  \label{R4}
The condition $\mu_4(\bp)$ was used in the proof of \refL{LemmaCEF} to
obtain $\Phi''''(A(z))=O(1)$. However, this term appears only once in
\eqref{C7}, and is then part of a term that contributes only $|1-z|\qqw
|1-x|^{-4}$ to \eqref{lcF}. There is  thus a margin to the bound in
\eqref{lcF}, which suggests that the proof might work also if the condition
$\mu_4(\bp)<\infty$ is relaxed; perhaps a second moment is enough here too. We have not investigated this further and leave
this to the reader.
\end{remark}

\begin{problem}
  Do the estimates in \refL{LemmaCEF} 
%(and thus in \refT{THn})
hold assuming only a second moment
$\gss(\bp)<\infty$?
\end{problem}
Note that this would imply a positive answer to \refP{P2}.

\newcommand\AAP{\emph{Adv. Appl. Probab.} }
\newcommand\JAP{\emph{J. Appl. Probab.} }
\newcommand\JAMS{\emph{J. \AMS} }
\newcommand\MAMS{\emph{Memoirs \AMS} }
\newcommand\PAMS{\emph{Proc. \AMS} }
\newcommand\TAMS{\emph{Trans. \AMS} }
\newcommand\AnnMS{\emph{Ann. Math. Statist.} }
\newcommand\AnnPr{\emph{Ann. Probab.} }
\newcommand\CPC{\emph{Combin. Probab. Comput.} }
\newcommand\JMAA{\emph{J. Math. Anal. Appl.} }
\newcommand\RSA{\emph{Random Structures Algorithms} }
\newcommand\DMTCS{\jour{Discr. Math. Theor. Comput. Sci.} }

\newcommand\AMS{Amer. Math. Soc.}
\newcommand\Springer{Springer-Verlag}
\newcommand\Wiley{Wiley}

\newcommand\vol{\textbf}
\newcommand\jour{\emph}
\newcommand\book{\emph}
\newcommand\inbook{\emph}
\def\no#1#2,{\unskip#2, no. #1,} %(typeset after year) 
\newcommand\toappear{\unskip, to appear}

\newcommand\arxiv[1]{\texttt{arXiv}:#1}
\newcommand\arXiv{\arxiv}

\def\nobibitem#1\par{}

\section*{Acknowledgements }

We thank Tom Hillegers for questions that inspired the present paper,
and 
Stephan Wagner for helpful comments.

GBO is supported by the Knut and Alice Wallenberg
Foundation, a grant from the Swedish Research Council and The Swedish
Foundations' starting grant from the Ragnar S\"oderberg Foundation.
SJ is supported by the Knut and Alice Wallenberg Foundation.

\end{document}